\documentclass{amsart}

\usepackage{amssymb}
\usepackage{mathtools}
\usepackage{stmaryrd}
\usepackage[show]{ed}

\numberwithin{equation}{section}

\newtheorem{theo}{Theorem} 

\newtheorem{lemma}{Lemma}[section]
\newtheorem{prop}[lemma]{Proposition}
\newtheorem{corol}[lemma]{Corollary}
\newtheorem{theoint}[lemma]{Theorem}
\newtheorem{claim}[lemma]{Claim}

\theoremstyle{remark}
\newtheorem{remark}[lemma]{Remark}

\newtheorem{step}{Step}
\newtheorem{cas}{Case}

\theoremstyle{definition}
\newtheorem{defi}[lemma]{Definition}

\newcommand{\hdot}{\dot{H}^1}

\newcommand{\NN}{\mathbb{N}}
\newcommand{\RR}{\mathbb{R}}

\newcommand{\lambdabf}{\boldsymbol{\lambda}}
\newcommand{\ldabf}{\boldsymbol{\lambda}}
\newcommand{\iotabf}{\boldsymbol{\iota}}
\newcommand{\nubf}{\boldsymbol{\nu}}
\newcommand{\betabf}{\boldsymbol{\beta}}

\newcommand{\alphabf}{\boldsymbol{\alpha}}
\newcommand{\mubf}{\boldsymbol{\mu}}

\newcommand{\lda}{\lambda}

\newcommand{\eps}{\varepsilon}

\newcommand{\tlh}{\tilde{h}}
\newcommand{\tj}{\widetilde{\jmath}}
\newcommand{\tu}{\tilde{u}}
\newcommand{\tU}{\widetilde{U}}

\newcommand{\indxt}{\indic_{\{|x|\geq |t|\}}}
\newcommand{\tgamma}{\tilde{\gamma}}
\newcommand{\tlambda}{\tilde{\lambda}}

\newcommand{\vv}{\vec{v}}

\newcommand{\HHH}{\mathcal{H}}

\newcommand{\NNN}{\mathcal{N}}
\newcommand{\OOO}{\mathcal{O}}
\newcommand{\PPP}{\mathcal{P}}

\newcommand{\Ssp}{\mathsf{S}}
\newcommand{\Wsp}{\mathsf{W}}

\newcommand{\ZZZZ}{\boldsymbol{\mathcal{Z}}}

\newcommand{\tJ}{\widetilde{J}}
\newcommand{\tK}{\widetilde{K}}

\newcommand{\tW}{\widetilde{W}}

\newcommand{\indic}{1\!\!1}

\newcommand{\ent}[1]{\left\lfloor #1 \right\rfloor} 

\DeclareMathOperator{\wlim}{w-lim}

\DeclareMathOperator{\loc}{loc}
\DeclareMathOperator{\vect}{span}

\DeclareMathOperator{\rad}{rad}
\title[Soliton resolution for radial critical wave equation]{Soliton resolution for the radial critical wave equation in all odd space dimensions}

\author[T.~Duyckaerts]{Thomas Duyckaerts$^1$}
\author[C.~Kenig]{Carlos Kenig$^2$}
\author[F.~Merle]{Frank Merle$^3$}
\thanks{$^1$LAGA (UMR 7539), Universit\'e Paris 13, Sorbonne Paris Cit\'e, and Institut Universitaire de France}
\thanks{$^2$University of Chicago. Partially supported by NSF Grants DMS-14363746 and DMS-1800082}
\thanks{$^3$AGM (UMR 8088), Universit\'e de Cergy-Pontoise, and Institut des Hautes \'Etudes Scientifiques}


\begin{document}
\begin{abstract}
 Consider the energy-critical focusing wave equation in odd space dimension $N\geq 3$. The equation has a nonzero radial stationary solution $W$, which is unique up to scaling and sign change. In this paper we prove that any radial, bounded in the energy norm solution of the equation behaves asymptotically as a sum of modulated $W$s, decoupled by the scaling, and a radiation term. 
 
 The proof essentially boils down to the fact that the equation does not have purely nonradiative multisoliton solutions. The proof overcomes the fundamental obstruction for the extension of the 3D case (treated in \cite{DuKeMe13})  by reducing
 the study of a multisoliton solution to a finite dimensional system of ordinary differential equations on the modulation parameters.
 The key ingredient of the proof is to show that this system of equations creates some radiation, contradicting the existence of pure multisolitons. 
\end{abstract}

\maketitle
\tableofcontents
\section{Introduction}
Consider the wave equation on $\RR^N$, $N\geq 3$, with an energy-critical focusing nonlinearity:
\begin{equation}
 \label{NLW}
 \partial_t^2u-\Delta u=|u|^{\frac{4}{N-2}}u,
\end{equation} 
and initial data
\begin{equation}
 \label{ID}
 \vec{u}_{\restriction t=0}=(u_0,u_1)\in \HHH,
\end{equation} 
where $\vec{u}:=(u,\partial_t u)$, and $\HHH:=\hdot(\RR^N)\times L^2(\RR^N)$. The equation is locally well-posed in $\HHH$ (see e.g. \cite{LiSo95}, \cite{KeMe08}, \cite{BuCzLiPaZh13}): for any initial data $(u_0,u_1)\in \HHH$, there exists a unique maximal solution $\vec{u}\in C^0((T_-,T_+),\HHH)$. The energy:
$$E(\vec{u}(t))=\frac{1}{2}\int_{\RR^N} |\nabla_{t,x}u(t,x)|^2\,dx-\frac{N-2}{2N}\int_{\RR^N} |u(t,x)|^{\frac{2N}{N-2}}\,dx$$
of a solution is conserved, where 
$$\nabla u=(\partial_{x_j}u)_{1\leq j\leq N},\quad \nabla_{t,x}u=(\partial_tu,\nabla u).$$
The equation \eqref{NLW} has the following scaling invariance. For $f\in \hdot(\RR^N)$ and $\lambda>0$, we denote
$$f_{(\lambda)}(x)=\frac{1}{\lambda^{\frac N2-1}} f\left( \frac{x}{\lambda} \right).$$
If $u$ is a solution of \eqref{NLW}, then 
$$\frac{1}{\lambda^{\frac{N}{2}-1}}u\left( \frac{t}{\lambda},\frac{x}{\lambda} \right)=u_{(\lambda)}\left( \frac{t}{\lambda},x \right)$$
 is also a solution.

As many other nonlinear dispersive equations, equation (\ref{NLW}) admits solitary waves (or solitons) that are well-localized solutions traveling at a fixed speed. The soliton resolution conjecture predicts that any global solution of this type of equations decouples asymptotically as a sum of decoupled solitons, a radiative term (typically, a solution to a linear equation) and a term going to zero in the energy space. For finite time blow-up solutions, a similar decomposition should hold depending on the nature of the blow-up.

Our main result (Theorem \ref{T:resolution} below) is the soliton resolution for equation \eqref{NLW}, when $N$ is odd and $(u_0,u_1)$ is radial. To put this result into perspective, we start with a discussion on the soliton resolution conjecture for general nonlinear dispersive equations.

This conjecture arose from numerical simulations and the theory of integrable systems. It was observed in 1955 by Fermi, Pasta and Ulam \cite{FermiPastaUlam55} in one of the first numerical experiments that a discretization of a wave equation with a quadratic nonlinearity leads to localized, soliton-like solutions. In 1965, Zabusky and Kruskal \cite{ZabuskyKruskal65}  highlighted numerically  the emergence of solitons and multisolitons solutions of the completely integrable KdV. This explained the result in \cite{FermiPastaUlam55}, as Kruskal found that, as the spacial mesh in the discretization tends to $0$, the solutions of the Fermi-Pasta-Ulam problem converge to solutions of the KdV equation see \cite{Kruskal78}. We refer to \cite{IvancevicIvancevic10BO} for a survey on numerical work.

The first theoretical results in the direction of the soliton resolution were obtained for the completely integrable KdV, mKdV and 1-dimensional cubic NLS, using the method
of inverse scattering. Namely, for KdV, a solution with smooth initial data decaying sufficiently fast at infinity decomposes, for positive $x$, as a finite sum of solitons and a term going to $0$ at infinity (see \cite{EcSc83}, \cite{Eckhaus86}). Note that this is only a partial result, due to the restriction on the initial data and also to the fact that the dispersive component, localized in $\{x<0\}$ is not completely described. We refer to \cite{Schuur86BO} for mKdV, and  \cite{ZakharovShabat71},  \cite{SegurAblowitz76}, \cite{Segur76}, \cite{Novoksenov80}, \cite{BoJeMcL18} for cubic NLS in one space dimension. A characteristic feature of these integrable systems, already observed in \cite{ZabuskyKruskal65}, is that the collision between solitons is elastic: a solution behaving as a sum of solitons as $t\to+\infty$ also behaves as a sum of solitons, with the same parameters, as $t\to-\infty$.

Very few complete results are known for non-integrable models. A typical dispersive partial differential equation for which the soliton resolution is believed to hold unconditionally is the energy critical wave maps. For this equation, the first known related results were ``bubble'' theorems, stating that any solution developing a singularity in finite or infinite time converges locally in space, along a sequence of times, to a soliton (see \cite{ChTZ93,Struwe03b} for the equivariant case, \cite{StTa10b} for the general case). Using the ``channels of energy'' method coming from our previous works \cite{DuKeMe11a}, \cite{DuKeMe13}, that is closely related to the techniques that we will develop in this article, it was proved that the soliton resolution holds for wave maps, in an equivariant setting, with an additional assumption ruling out a multisoliton configuration \cite{CoKeLaSc15a, CoKeLaSc15b}, and that it holds for a sequence of times without this condition (see \cite{Cote15} and \cite{JiaKenig17}). The limiting case of a pure two-soliton is treated in \cite{JendrejLawrie19}, where it is shown in particular that the collision between the two solitons is inelastic.

For wave maps without symmetry assumption, but with the same $S^2$ target, a weak form of the soliton resolution was proved along a sequence of times (see \cite{Grinis17}), and the complete resolution is only known close to the ground-state \cite{DuJiKeMe17b}.

The proof of the soliton resolution conjecture seems out of reach for other non-integrable nonlinear dispersive equations, such as nonlinear Schr\"odinger and Klein-Gordon equations. Known results include scattering below a threshold given by the ground state of the equation (see e.g. \cite{KeMe06}, \cite{DuHoRo08}, \cite{IbMaNa11}, \cite{Dodson15}), local study close to the ground state solution (see \cite{NaSc11Bo}), and in some particular cases the existence of a global compact attractor (see \cite{Tao07DPDE}). We refer to the introduction of \cite{DuJiKeMe17} for a more complete discussion and more references on the subject.

Going back to equation \eqref{NLW}, it is known that if $\|(u_0,u_1)\|_{\HHH}$ is sufficiently small, then $T_+=\infty$ and the solution scatters to a linear solution. It is also well-known that in general finite energy solutions to equation \eqref{NLW} may blow up in finite time. Indeed, using the finite speed of propagation for equation \eqref{NLW} to localize ODE type blow up solutions, one can easily construct solutions $\vec{u}$ with $T_+<\infty$ and $\|\vec{u}(t)\|_{\HHH}\to\infty$, as $t\to T_+$. These solutions are called \emph{type I} blow-up solutions. It is expected that these solutions, after a self-similar change of variable, satisfy a decomposition similar to the soliton resolution. This type of result is only known in the 1-d setting (see \cite{MerleZaag12} and references therein) and very little is known in the energy-critical case (see \cite{Donninger17} for a local study).

To rule out the ODE type behavior, we will focus on solutions that are bounded in the energy space, i.e. such that
\begin{equation}
\label{bounded}
\sup_{t\in [0,T_+)}\,\|\vec{u}(t)\|_{\HHH}<\infty.
\end{equation}
The dynamics of these solutions is very rich. Apart from the scattering solutions mentioned above, equation (\ref{NLW}) admits also various types of finite energy steady states $Q\in\dot{H}^1$, i.e.
\begin{equation}
\label{Ell}
-\Delta Q=|Q|^{\frac{4}{N-2}}Q\,,\quad \mathrm{in}\,\,\RR^N,
\end{equation}
(see \cite{Ding86}, \cite{dPMPP11}, \cite{dPMPP13}).
Among them, a distinguished role is played by the {\it ground state} $W$.
\begin{equation}
\label{defW}
W:=\left(1+\frac{|x|^2}{N(N-2)}\right)^{1-\frac{N}{2}},
\end{equation} 
which is, as a consequence of \cite{Pohozaev65,GiNiNi81}, the 
 unique $\hdot$ radial solution of \eqref{Ell}
 on $\RR^N$, up to scaling and sign change, and the non-zero solution of \eqref{Ell} with least energy (see \cite{Talenti76}).
 
Stationary solutions are not the only global, non-scattering solution. It is indeed possible to construct solutions of the form 
\begin{equation*}
u(t,x)=W_{(\lambda(t)}\left(x\right)+v_L(t,x),
\end{equation*} 
where $v_L$ is a small solution of the free wave equation: see \cite{KrSc07} (for $\lambda(t)=1$) and  \cite{DoKr13} ($\lambda(t)=t^{\eta}$, $|\eta|$ small). There also exist, at least in high space dimensions, global solutions that are asymptotically of the form $W+W_{(\lambda(t))}$, where $\lambda(t)$ goes to $0$ as $t$ goes to infinity (see \cite{Jendrej19}).

There are also solutions blowing up in finite time that are bounded in the energy space. These solutions are called \emph{type II blow-up solutions}.
In \cite{KrScTa09}, \cite{HiRa12}, \cite{KrSc14} and \cite{Jendrej17}, type II blow-up solutions of the form of a rescaled ground state plus a small dispersive term  were constructed. More precisely the solution is given by 
\begin{equation*}
u(t,x)=W_{(\lambda(t)}\left(x\right)+\epsilon(t,x),
\end{equation*}
where $\frac{\lambda(t)}{T_+-t}\to 0+$ as $t\to T_+$, and $\vec{\epsilon}(t)=(\epsilon,\,\partial_t\epsilon)$ is small in the energy space. It is expected that multisoliton concentration is also possible for type II blow-up solutions, and it is an open problem to construct such a solution. 

In the radial setting,
$W$ is the unique steady state, and thus the only soliton up to sign change and scaling.  The soliton resolution conjecture predicts that any radial solution that does not blow up with a type I blow-up decomposes asymptotically as a sum of $\pm W$, decoupled by time-dependent scalings, a radiation term and a term going to zero in the energy space. The radiation term should be a solution to the linear wave equation in the global case, and a fixed element of $\HHH$ in the finite time blow-up case. We note that all the solutions mentioned above are in accordance with this conjecture. The resolution was proved in \cite{DuKeMe13} by the authors, for $N=3$. For other dimensions (still in the radial case), soliton resolution is only known along a sequence of times, see \cite{CoKeLaSc18}, \cite{Rodriguez16} and \cite{JiaKenig17}. For the non-radial settting, for a sequence of times, see \cite{DuJiKeMe17}.

With the method of proof used in \cite{CoKeLaSc18,Rodriguez16,JiaKenig17,DuJiKeMe17}, relying on monotonicity laws giving convergence only after averaging in time, we cannot hope for more than a decomposition for a particular sequence of times. The difficulty in obtaining the resolution for all times is illustrated by the harmonic map heat flow equation, for which the decomposition for a sequence of times is known, but the soliton resolution for all times does not hold in full generality because of an example of Topping \cite{Topping97}. 

The soliton resolution for radial solutions of \eqref{NLW} holds in full generality \cite{DuKeMe13} when $N=3$. The key fact in the proof is the following dispersive estimate for radial non-zero solutions $u$ of \eqref{NLW}, with $(u_0,u_1)\neq (\pm W_{(\lambda)},0)$, $N=3$. Assume (for simplicity) that $u$ exists globally in time. Then
\begin{equation}
 \label{plus}
\sum_{\pm} \lim_{t\to \pm\infty} \int_{r\geq |t|} |\nabla_{t,x}u(t,x)|^2\,dx>0.
\end{equation} 
The proof of \eqref{plus} relies fundamentally, among other things, on the following ``energy channel'' property of radial solution $v$ of the linear wave equation in space dimension $3$ (see \cite{DuKeMe11a,DuKeMe13}). Let $R>0$, $P(R)=\left\{ \left( \frac{a}{r},0 \right),\; a\in \RR\right\}\subset \HHH(R):=(\hdot\times L^2)(\{x\in \RR^3,\; |x|>R\})$. Then
\begin{equation}
 \label{linplus}
\sum_{\pm}\lim_{t\to\pm \infty} \int_{|x|>R+|t|}|\partial_{t,r}v(t)|^2\,dx\geq \frac 12 \left\|\Pi_{P(R)}^{\bot}(v_0,v_1)\right\|^2_{\HHH(R)},
\end{equation} 
where $\Pi^{\bot}_{P(R)}$ denotes the orthogonal projection onto the orthogonal complement of $P(R)$ in $\HHH(R)$. The analog of \eqref{linplus} for higher odd dimension was obtained in \cite{KeLaSc14},  \cite{KeLaLiSc15}, but the exceptional subspace $P(R)$ is replaced by a finite dimensional subspace of $\HHH(R)$ with dimension increasing to infinity with $N$. The fact that the dimension of $P(R)$ is strictly greater than one for $N\geq 5$ is responsible for the failure of this method, since we only have here the one parameter scaling invariance to deal with this failure, to start the proof. Let us mention however that it is possible, using \eqref{linplus}, to prove that in odd space dimensions $N\geq 5$, any radial solution of \eqref{NLW} that does not satisfy \eqref{plus} is asymptotically close, for large $r$, to one of the elements of $P(R)$ (see \cite{DuKeMe19Pc}).

The radial solution $u$ of \eqref{NLW} is said to be a \emph{pure multisoliton} (asymptotically as $t\to \pm \infty$) when there exist $J\geq 2$ scaling parameters $0<\lambda_J(t)\ll \ldots\ll \lambda_2(t)\ll \lambda_1(t)$ and signs $(\iota_j)_j\in \{\pm 1\}^{J}$ such that
$$\vec{u}(t,x)=\sum_{j=1}^J \left(\iota_j W_{(\lambda_j(t))},0\right)+o(1),\quad t\to\pm \infty,$$
where $o(1)$ goes to $0$ in $\HHH$ (see \cite{Jendrej19} for an example of pure multisoliton). If $u$ is both a pure multisoliton as $t\to+\infty$ and $t\to -\infty$, we say that the collision between the solitons is \emph{elastic}. In space dimension $3$, the fact that \eqref{plus} is valid for any nonstationary solution $u$ rules out elastic collisions, in stark contrast to the integrable case \cite{IvancevicIvancevic10BO}. One of the main results in this work, is a slightly weaker form of \eqref{plus}, namely that a radial solution $u$ of \eqref{NLW} that stays close to a sum of decoupled solitons for a sufficiently long time must satisfy
\begin{equation}
 \label{plusM}
\sum_{\pm} \lim_{t\to \pm\infty} \int_{r\geq |t|-M} |\nabla_{t,x}u(t,x)|^2\,dx>0,
\end{equation} 
for some large $M>0$ (see Propositions \ref{P:EDO} and \ref{P:rigidity}). As a consequence, there does not exist a radial solution of \eqref{NLW} that 
is a pure multisoliton both as $t\to\infty$ and $t\to-\infty$, i.e. the collision of solitons is inelastic for \eqref{NLW}, when $N$ is odd. The proof of this property depends heavily on the ``energy channels'' property for the linearized  wave equation $\left(\partial_t^2-\Delta-\frac{N+2}{N-2}W^{\frac{4}{N-2}}\right)u=0$, established in \cite{DuKeMe19Pa}. The $M$ in \eqref{plusM} is needed to eliminate the extra dimensions arising from $P$ in \eqref{linplus}, when $N\geq 5$, odd. 

The main result in this work, namely full resolution for \eqref{NLW}, $N$ odd, in the radial case (assuming bounded energy norm) combines the sequence of times result in \cite{Shen14} with a strengthened version of \eqref{plusM}. The result of \cite{Shen14} allows us to reduce ourselves to studying the dynamics close to a sum of solitons plus a dispersive term, and the strenghtened version of \eqref{plusM} allows us to take advantage of the fact that the collision of two or more solitons produce dispersion, which then gives the full decomposition. We view this as a ``road map'' to attack soliton resolution in non-integrable settings. 

We now turn to the main results of this paper. If $a$ and $b$ are integers with $a<b$, we denote $\llbracket a,b\rrbracket=[a,b]\cap\NN$.

\begin{theo}
\label{T:resolution}
Assume that $N\geq 5$ is odd. Let $u$ be a radial solution of \eqref{NLW}, with maximal time of existence $T_+$, such that 
\begin{equation}
 \label{u_bounded}
\sup_{0\leq t<T_+} \|\vec{u}(t)\|_{\HHH}<\infty.
 \end{equation} 
Then there exists $J\geq 0$, signs $(\iota_j)_j\in \{\pm 1\}^J$, scaling parameters $(\lambda_j)_j\in (0,\infty)^J$ such that 
$$\forall j\in \llbracket 1,J-1\rrbracket,\quad \lim_{t\to T_+}\frac{\lambda_j(t)}{\lambda_{j+1}(t)}=+\infty$$
and
\begin{itemize}
 \item (Type II blow-up case). If $T_+<\infty$, then $J\geq 1$ and
 there exists $(v_0,v_1)\in \HHH$ such that 
 $$\lim_{t\to T_+}\left\|\vec{u}(t)-(v_0,v_1)-\sum_{j=1}^J(\iota_jW_{(\lambda_j)},0)\right\|_{\HHH}=0.$$
 Furthermore:
$$\lim_{t\to T_+} \frac{\lambda_1(t)}{T_+-t}=0.$$
 \item (Global in time case). If $T_+=+\infty$, then  
 there exists a solution $v_L$ of the linear wave equation such that 
 $$\lim_{t\to +\infty}\left\|\vec{u}(t)-\vec{v}_L(t)-\sum_{j=1}^J(\iota_jW_{(\lambda_j)},0)\right\|_{\HHH}=0.$$
Furthermore, if $J\geq 1$, 
 $$\lim_{t\to +\infty} \frac{\lambda_1(t)}{t}=0.$$
 \end{itemize}
\end{theo}

As mentioned above, the proof of Theorem \ref{T:resolution} yields the fact that the collision between radial solitons in odd space dimension $N\geq 5$ is inelastic:
\begin{theo}
 \label{T:inelastic}
Assume that $N\geq 5$ is odd.
Let $u$ be a radial, global, solution of \eqref{NLW} such that 
$$ \sup_{t\in \RR}\left\|\vec{u}(t)\right\|_{\HHH}<\infty$$
and
\begin{equation}
\label{no_radiation}
\forall A>0, \quad \sum_{\pm \infty} \lim_{t\to\pm\infty} \int_{|x|>|t|-A} |\nabla_{t,x}u(t,x)|^2\,dx=0.
\end{equation} 
Then $(u_0,u_1)=(0,0)$ or there exists $\lambda>0$, $\iota\in \{\pm 1\}$ such that $(u_0,u_1)=\left( \iota W_{(\lambda)},0 \right)$.
\end{theo}
Note that \eqref{no_radiation} exactly means that the linear component $v_L$ is identically $0$ both as $t\to +\infty$ and $t\to -\infty$, so that Theorem \ref{T:inelastic} rules out being asymptotically a multisoliton at both $t=+\infty$ and $t=-\infty$, in stark contrast to the completely integrable case \cite{IvancevicIvancevic10BO}. For results on inelastic soliton collisions for equation \eqref{NLW} without a radiality assumption, see \cite{MartelMerle18}.

The outline of the paper is as follows. 
The preliminary section \ref{S:preliminaries} is mainly devoted to the Cauchy theory for equation \eqref{NLW}. We recall well-posedness results from \cite{GiSoVe92}, \cite{GiVe95}, \cite{KeMe08}, and more particularly from \cite{BuCzLiPaZh13}, where the high-dimensional case is treated. Using finite speed of propagation, we also recall a local and global Cauchy theory for the equation \eqref{NLW} in the exterior of a wave cone $\{|x|>R+|t|\}$, $R\geq 0$, as developed in \cite{DuKeMe19Pc}. Section \ref{S:multisoliton} concerns the bound from below of the exterior energy for linear equations with a potential. After recalling the main result of \cite{DuKeMe19Pa}, we state and prove an exterior energy bound for the linearized operator close to a multisoliton. In Section \ref{S:exterior}, we consider solutions of the equation \eqref{NLW} such that 
$$ \lim_{t\to+\infty} \int_{|x|>|t|+R} |\nabla_{t,x}u|^2\,dx=0,$$
for some fixed $R>0$. 
We recall from \cite{DuKeMe19Pc} that the initial data of these solutions (that we call non-radiative solutions) have a prescribed asymptotic behaviour. We also consider the case of non-radiative solutions that are close to a multisoliton, proving a bound from below of the exterior scaling parameter $\lambda_1$. In Section \ref{S:reduction}, we reduce the proof of the soliton resolution to the study of a finite dimensional dynamical system on the scaling parameters $\lambda_j$ and some of the coefficients arising in the expansion of the solution. Finally in Section \ref{S:end_of_proof} we prove a blow-up/ejection result for this dynamical system and conclude the proof. Section \ref{S:inelastic} is dedicated to a short sketch of the proof of Theorem \ref{T:inelastic}, which is a byproduct of part of the proof of Theorem \ref{T:resolution}. A few computations are gathered in the appendix.
\section{Preliminaries}
\label{S:preliminaries}
\subsection{Notations}
\label{SS:notations}
We denote $\hdot=\hdot(\RR^N)$, $L^2=L^2(\RR^N)$, $\HHH=\hdot\times L^2$.
If $\lambda>0$, $f\in \dot{H}^1$ and $g\in L^2$, we let:
$$ f_{(\lambda)}(x)=\frac{1}{\lambda^{\frac{N}{2}-1}}f\left( \frac{x}{\lambda} \right),\quad g_{[\lambda]}(x)=\frac{1}{\lambda^{\frac{N}{2}}}g\left( \frac{x}{\lambda} \right),$$
so that 
$$\|f_{(\lambda)}\|_{\hdot}=\|f\|_{\hdot}\text{ and }\|g_{[\lambda]}\|_{L^2}=\|g\|_{L^2}.$$

If $A$ is a space of distributions on $\RR^N$, we will denote by $A_{\rad}$ the subspace of $A$ consisting of the elements of $A$ that are radial. We will, without making a distinction, consider a radial function as depending on the variable $x\in \RR^N$ or the variable $r=|x|$.

If $\Omega$ is an open subset of $\RR^{n}$, ($n=N$ or $n=N+1$), and $A=A(\RR^n)$ a Banach space of distributions on $\RR^n$, we recall that $A(\Omega)$ is the set of restrictions of elements of $A$ to $\Omega$, with the norm
$$\|u\|_{A(\Omega)}:=\inf_{\tilde{u}}\|\tilde{u}\|_{A(\RR^n)}.$$
where the infimum is taken over all $\tilde{u}\in A(\RR^n)$ such that $\tilde{u}_{\restriction \Omega}=u$. 
To lighten notation, if $R>0$ and $n=N$, we will denote by 
$$A(R):= A_{\rad}\left(\left\{x\in \RR^N\text{ s.t. }|x|>R\right\}\right).$$
We will mainly use this notation with $\HHH$, so that $\HHH(R)$ is the space of radial distributions $(u_0,u_1)$ defined for $r>R$ such that
$$ u_0\in L^{\frac{2N}{N-2}}((R,+\infty), r^{N-1}dr), \quad \int_R^{\infty} (\partial_ru_0)^2r^{N-1}\,dr<\infty$$
and 
$$ u_1\in L^2\left((R,+\infty),r^{N-1}dr\right).$$

We will often consider solutions of the wave equation in the exterior of wave cones. For $R>0$, we denote 
$$\Gamma_R(t_0,t_1)=\{|x|> R+|t|,\; t\in [t_0,t_1]\}.$$
To lighten notations, we will denote
$$\Gamma_R(T)=\Gamma_R(0,T),\quad \Gamma_R=\Gamma_{R}(0,\infty).$$
We denote by $S_L(t)$ the linear wave group: 
\begin{equation}
\label{SL}
S_L(t)(u_0,u_1)=\cos(t\sqrt{-\Delta})u_0+\frac{\sin(t\sqrt{-\Delta})}{\sqrt{-\Delta}}u_1,
\end{equation} 
so that the general solution (in the Duhamel sense) of 
\begin{equation}
 \label{LW_in}
\left\{
\begin{aligned}
(\partial_t^2-\Delta)u&=f\\
\vec{u}_{\restriction t=t_0}&=(u_0,u_1)\in \HHH,
\end{aligned}\right.
\end{equation} 
where $I$ is an interval and $t_0\in I$
is 
\begin{equation}
\label{Duhamel}
u(t)=S_L(t-t_0)(u_0,u_1)+\int_{t_0}^{t}S_L(t-s)(0,f(s))\,ds. 
\end{equation} 
We note that by finite speed of propagation, the restriction of $u$ to $\Gamma_R(T)$ depends only on the restriction of $f$ to $\Gamma_R(T)$ and the restriction of $(u_0,u_1)$ to $\{r>R\}$.

\subsection{Local and global Cauchy theory}
\label{SS:defi}
We will denote by $\dot{W}^{s,p}(\RR^N)$ the homogeneous Sobolev space defined as the closure of $C^{\infty}_0(\RR^N)$ with respect to the norm $\|\cdot\|_{\dot{W}^{s,p}}$ defined by
$$ \|f\|_{\dot{W}^{s,p}}:=\|D^s f\|_{L^p},$$
where $D^s$ is the Fourier multiplier of symbol $|\xi|^{s}$. We denote by $\dot{B}^{s}_{p,q}$ the standard homogeneous Besov space, which can be defined using Littlewood-Paley decomposition or the real interpolation method: $\dot{B}^{s}_{p,q}=\left[L^p,\dot{W}^{1,p}\right]_{s,q}$, $0<s<1$, $1\leq p,q\leq \infty.$

We define, following \cite{BuCzLiPaZh13}:
\begin{gather*}
\Ssp:= L^{\frac{2(N+1)}{N-2}}(\RR^{1+N}),\quad \Wsp:=L^{\frac{2(N+1)}{N-1}}\left( \RR,\dot{B}^{\frac{1}{2}}_{\frac{2(N+1)}{N-1},2} (\RR^N) \right)\\
\Wsp':= L^{\frac{2(N+1)}{N+3}}\left(\RR,\dot{B}^{\frac{1}{2}}_{\frac{2(N+1)}{N+3},2}(\RR^N) \right).
\end{gather*}
If $I$ is an interval, we will denote by $\Ssp(I)$, $\Wsp(I)$, $\Wsp'(I)$ the restriction of these spaces to $I\times \RR^N$.

We will need the following Strichartz estimates (see \cite{St77a} \cite{GiVe95}): if $t_0\in I$, $f\in \Wsp'(I)$,  $(u_0,u_1)\in \HHH$, then $u$ (defined by \eqref{Duhamel}) is in $\Ssp(I)\cap \Wsp(I)$ and
\begin{equation}
\label{Strichartz}
\sup_{t\in \RR} \|\vec{u}(t)\|_{\HHH}+\|u\|_{\Ssp(I)}+
\|u\|_{\Wsp(I)} 
\lesssim \|(u_0,u_1)\|_{\HHH(I)} +\|f\|_{\Wsp'(I)}.
\end{equation}
We denote $F(u)=|u|^{\frac{4}{N-2}}u$. 

\begin{defi}
\label{D:solution}
 Let $I$ be an interval with $t_0\in I$, $(u_0,u_1)\in \HHH$. If $N\geq 6$, we call solution of \eqref{NLW} on $I\times \RR^N$, with initial data
\begin{equation}
 \label{S0}
\vec{u}_{\restriction t=t_0}=(u_0,u_1)
\end{equation} 
a function $u\in C^0(I,\hdot)$ such that $\partial_t u\in C^0(I,L^2)$ and
\begin{equation}
 \label{S1}
\forall t\in I,\quad u(t)=S_L(t-t_0)(u_0,u_1)+\int_{t_0}^tS_L(s-t_0)F(u(s))\,ds.
\end{equation} 
If $N\in \{3,4,5\}$, a solution is defined in the same way, with the additional requirement that $u\in \Ssp(J\times \RR^N)$ for all compact intervals $J\subset I$.
\end{defi}
It is known (see \cite{GiVe95}, \cite{KeMe08} and \cite{BuCzLiPaZh13}), that for all initial data $(u_0,u_1)$, there is a unique maximal solution $u$ defined on a maximal interval $(T_-,T_+)$ and that satisfies the following blow-up criterion:
$$T_+<\infty\Longrightarrow \|u\|_{\Ssp([t_0,T_+))}=\infty.$$

We next recall from \cite{DuKeMe19Pc} the definition and some properties of solutions of \eqref{NLW} on the exterior $\Gamma_R(t_0,t_1)$ of wave cones. We will use the following continuity property of multiplication by characteristic functions on a Besov space (see \cite[Lemma 2.3]{DuKeMe19Pc}):
\begin{lemma}
 \label{L:char_Besov}
 Let $R\geq 0$.
 \begin{itemize}
  \item The multiplication by the characteristic function $\indic_{\{|x|>R\}}$ is 
a continuous function from $\dot{B}^{\frac 12}_{\frac{2(N+1)}{N+3},2}(\RR^N)$ into itself, and from
$\dot{W}^{\frac{2}{N},\frac{2(N+1)}{N+3}}(\RR^N)$ into itself. In both cases, the operator norm is independent of $R$.
\item Let $I$ be an interval. The multiplication by the characteristic function $\indic_{\{|x|>R+|t|\}}$ is continuous from $\Wsp'(I)$ into itself into itself. The operator norm is independent of $R$ and $I$.
\end{itemize}
 \end{lemma}
We also recall the following chain rule for fractional derivative outside wave cones:
\begin{equation}
 \label{fractional_cones}
 \|\indic_{\Gamma_R(T)}F(u)\|_{\Wsp'((0,T))}\lesssim \|u\|_{\Ssp(\Gamma_R(T))}^{\frac{4}{N-2}}\|u\|_{\Wsp((0,T))},
\end{equation} 
which is proved in \cite{DuKeMe19Pc} as a consequence of Lemma \ref{L:char_Besov}, H\"older's inequality and the usual chain rule for fractional derivative (\cite[Lemma 2.10]{BuCzLiPaZh13}).

\begin{defi}
\label{D:sol_cone}
 Let $t_0<t_1$, $R\geq 0$. Let $(u_0,u_1)\in \HHH(R)$. A solution $u$ of \eqref{NLW} on $\Gamma_R(t_0,t_1)$ with initial data $(u_0,u_1)$ is the restriction to $\Gamma_R(t_0,t_1)$ of a solution $\tilde{u}\in C^0([t_0,t_1],\hdot)$ with $\partial_t\tilde{u}\in C^0([t_0,t_1],L^2)$, to the equation:
 \begin{equation}
  \label{NLW_trunc}
\partial_t^2\tu-\Delta \tu=|\tu|^{\frac{4}{N-2}}\tu\indic_{\{|x|>R+|t|\}},
\end{equation} 
with an initial data
\begin{equation}
 \label{ID_trunc}
 \vec{\tu}_{\restriction t=t_0}=(\tilde{u}_0,\tilde{u}_1),
\end{equation}
where $(\tilde{u}_0,\tilde{u}_1)\in \HHH$ is an extension of $(u_0,u_1)$
\end{defi}
Note that by finite speed of propagation, the value of $u$ on $\Gamma_R(t_0,t_1)$ does not depend on the choice of $(\tilde{u}_0,\tilde{u}_1)$, provided $(\tilde{u}_0,\tilde{u}_1)$ and $(u_0,u_1)$ coincide for $r>R$.

Using
Lemma \ref{L:char_Besov} and finite speed of propagation, the Cauchy theory in \cite{BuCzLiPaZh13} (or \cite{KeMe08} for the case $N\in\{3,4,5\}$) adapts easily to the case of solutions outside wave cones.
We give some of the statements, and omit the proofs that are the same as in \cite{KeMe08}, \cite{BuCzLiPaZh13}. We refer to \cite[Section 2]{DuKeMe19Pc} for a more complete exposition. The space $\Ssp(\Gamma_R(T))$ in the following proposition is defined in Subsection \ref{SS:notations}.

\begin{prop}[Local well-posedness]
 \label{P:LWP_cone}
 Let $R\geq 0$, $(u_0,u_1)\in \HHH(R)$ and $T>0$. Assume   
 $$\|(u_0,u_1)\|_{\HHH(R)}\leq A.$$
 Then there exists $\eta=\eta(A)$ such that if 
 $$\|S_L(t)(u_0,u_1)\|_{\Ssp(\Gamma_R(T))}<\eta,$$
 then there exists a unique solution $u$ to \eqref{NLW} on $\Gamma_R(T)$. Furthermore for all $t\in [0,T]$,
$$\|\vec{u}(t)-\vec{S}_L(t)(u_0,u_1)\|_{\HHH(R+|t|)}\leq C\eta^{\theta_N}A^{1-\theta_N}$$
for some constant $\theta_N$ depending only on $N$.
\end{prop}
(See \cite[Theorem 3.3]{BuCzLiPaZh13}).
We have the following blow-up criterion, see \cite[Lemma 2.8]{DuKeMe19Pc}:
if $u\in S\left(\Gamma_{R}(T_R^+)\right)$, then $T^+_R=+\infty$. Furthermore, $u$ scatters to a linear solution for $\{|x|>R+|t|\}$: there exists a solution $v_L$ of the linear wave equation on $\RR\times \RR^N$ such that 
\begin{equation}
\label{exterior_scattering}
\lim_{t\to +\infty} \left\|\vec{u}(t)-\vec{v}_L(t)\right\|_{\HHH(R+|t|)}=0. 
\end{equation}

We also have the following long-time perturbation theory result (see \cite[Theorem 2.20]{KeMe08}, \cite[Theorem 3.6]{BuCzLiPaZh13}, \cite[Proposition A.1]{Rodriguez16}).
\begin{prop}
\label{P:LTPT}
Let $A>0$. There exists $\eta_0=\eta_0(A)$ with the following property.
 Let $R>0$, $T\in (0,\infty]$, $(u_0,u_1)\in \HHH(R)$ and $(v_0,v_1)\in \HHH(R)$. Assume that $v$ is a restriction to $\Gamma_R(0,T)$ of a function $V$ such that $\vec{V}\in C^0([0,T],\HHH)$ and 
 $$\partial_t^2V-\Delta V=\indic_{\{|x|>R+|t|\}}\left(F(V)+e_1+e_2\right),$$
 with 
 \begin{gather*}
\sup_{0\leq t\leq T}\|V(t)\|_{\HHH(R+|t|)}+\|V\|_{\Wsp(0,T)}\leq A\\
\|(u_0,u_1)-(v_0,v_1)\|_{\HHH(R)}+\|e_1\|_{\Wsp'(0,T)}+\|e_2\|_{L^1((0,T),L^2)}=\eta\leq \eta_0, 
\end{gather*}
Then the solution with initial data $(u_0,u_1)$ is defined on $\Gamma_R(T)$ and 
$$ \|v-u\|_{\Ssp(\Gamma_R(T))}\leq C\eta^{c_N},$$
for some constant $c_N\in (0,1]$ depending only on $N\geq 3$.
\end{prop}
\begin{remark}
 In \cite{KeMe08,BuCzLiPaZh13,Rodriguez16}, $e_2=0$, but the argument easily adapts to the setting of Proposition \ref{P:LTPT}.
\end{remark}

\subsection{Profile decomposition}
\label{SS:profile}
Let $\big\{(u_{0,n},u_{1,n})\big\}_n$ be a bounded sequence of radial functions in $\HHH$. We say that it admits a profile decomposition if for all $j\geq 1$, there exist a solution $U^j_F$ to the free wave equation with initial data in $\HHH$ and sequences of parameters $\{\lambda_{j,n}\}_n\in (0,\infty)^{\NN}$, $\{t_{j,n}\}_n\in \RR^{\NN}$ such that
\begin{equation}
 \label{psdo_orth}
 j\neq k \Longrightarrow \lim_{n\to\infty}\frac{\lambda_{j,n}}{\lambda_{k,n}}+\frac{\lambda_{k,n}}{\lambda_{j,n}}+\frac{|t_{j,n}-t_{k,n}|}{\lambda_{j,n}}=+\infty,
\end{equation} 
and, denoting 
\begin{gather}
 \label{rescaled_lin}
 U^j_{F,n}(t,r)=\frac{1}{\lambda_{j,n}^{\frac N2-1}}U^j_F\left( \frac{t-t_{j,n}}{\lambda_{j,n}},\frac{r}{\lambda_{j,n}} \right),\quad j\geq 1\\
 w_{n}^J(t)=S_L(t)(u_{0,n},u_{1,n})-\sum_{j=1}^J U^j_{F,n}(t),
\end{gather} 
one has 
\begin{equation}
 \label{wnJ_dispersive}
 \lim_{J\to\infty}\limsup_{n\to\infty}\|w_n^J\|_{\Ssp(\RR)}=0.
\end{equation} 
We recall (see \cite{BaGe99}, \cite{Bulut10}) that any bounded sequence in $\HHH$ has a subsequence that admits a profile decomposition. We recall also that the properties above imply that the following weak convergences hold:
\begin{equation}
 \label{wlim_w}
 j\leq J\Longrightarrow \left( \lambda_{j,n}^{\frac{N}{2}-1}w_n^J\left(t_{j,n},\lambda_{j,n}\cdot \right),\lambda_{j,n}^{\frac{N}{2}}\partial _tw_n^J\left(t_{j,n},\lambda_{j,n}\cdot \right)\right) \xrightharpoonup[n\to\infty]{} 0 \text{ in }\HHH.
\end{equation} 

If $\{(u_{0,n},u_{1,n})\}_n$ admits a profile decomposition, we can assume, extracting subsequences and time-translating the profiles if necessary, that the following limit exists:
$$\lim_{n\to\infty}\frac{-t_{j,n}}{\lambda_{j,n}}=\tau_j\in \{-\infty,0,\infty\}.$$
Using the existence of wave operator for the equation \eqref{NLW} if $\tau_j\in \{\pm\infty\}$ or the local well-posedness if $\tau_j=0$, we define the nonlinear profile $U^j$ associated to $\left(U^j_F,\{\lambda_{j,n}\}_n,\{t_{j,n}\}_n\right)$ as the unique solution to the nonlinear wave equation \eqref{NLW} such that 
$$\lim_{t\to\tau_j} \left\|\vec{U}^j(t)-\vec{U}^j_F(t)\right\|_{\HHH}=0.$$
We also denote by $U^j_n$ the rescaled nonlinear profile:
$$ U^j_n(t,r)=\frac{1}{\lambda_{j,n}^{\frac N2-1}}U^j\left( \frac{t-t_{j,n}}{\lambda_{j,n}},\frac{r}{\lambda_{j,n}} \right).$$
Then we have the following superposition principle outside the wave cone $\Gamma_0:=\left\{(t,x)\in \RR\times \RR^N\; : \;|x|>t>0\right\}.$
\begin{prop}
 \label{P:NL_profile}
 Let $\{(u_{0,n},u_{1,n})\}_n$ be a bounded sequence in $\HHH_{\rad}$. Assume that for all $j$ such that $\tau_j=0$, the nonlinear profile $U^j$ can be extended to a solution on $\Gamma_0$ (in the sense of Definition \ref{D:sol_cone}) such that $U^j\in \Ssp(\Gamma_0)$. Then for large $n$, there is a solution $u_n$ defined on $\Gamma_0$ with initial data $\{(u_{0,n},u_{1,n})\}_n$ at $t=0$. Furthermore, denoting, for $J\geq 1$, $(t,r)\in \Gamma_0$
 $$R_n^J(t,r)=u_n(t,r)-\sum_{j=1}^J U_n^j(t,r)-w_n^J(t,r),$$
 we have 
 $$\lim_{J\to\infty} \lim_{n\to\infty}\|R_n^J\|_{\Ssp(\Gamma_0)}+\sup_{t\geq 0}\left\|\vec{R}_n^J(t)\right\|_{\HHH(t)}=0.$$
\end{prop}
We omit the proof, which is similar to the proof when the solution is not restricted to the exterior of a wave cone (see \cite[Proposition 2.3]{Rodriguez16}). Let us emphasize the fact that under the assumptions of Proposition \ref{P:NL_profile}, the profiles $U_n^j(t,r)$ are well-defined on $\Gamma_0$, so that the conclusion of the proposition makes sense. If $\tau_j=0$, this follows from the assumption that $U^j$ is defined on $\Gamma_0$ and if $\tau_j=+\infty$, from the fact that $U^j$ is globally defined in the future. Finally, if $\tau_j=-\infty$, it follows from the fact that $U^j$ is globally defined in the past, and also, using small data theory, defined on a cone $\Gamma_R(T,+\infty)$ where $T$ is fixed in the interval of existence of $U^j$ and $R$ is large (see also the discussion after Proposition 2.11 in \cite{DuKeMe19Pc}). 

\subsection{Wave equation with a potential outside a wave cone}
\begin{lemma}
 \label{L:linear_approx}
 Let $N\geq 3$, and $M\in (0,\infty)$.
 There exists $C_M>0$ such that for all $V\in L^{\frac{2(N+1)}{N+4}}_{\loc}\left( \RR, L^{\frac{2(N+1)}{3}}\left(\RR^N\right) \right)$ with
 \begin{equation}
  \label{G84}
  \left\|\indxt V\right\|_{L^{\frac{2(N+1)}{N+4}} \left( \RR, L^{\frac{2(N+1)}{3}} \left(\RR^N\right) \right)} \leq M,
 \end{equation} 
 for all solution $u$ of 
 \begin{equation}
  \label{G83}
  \partial_t^2u-\Delta u+Vu=f_1+f_2,\quad \vec{u}_{\restriction t=0}=(u_0,u_1)\in \HHH,
 \end{equation} 
 where $f_1\in L^1\left(\RR,L^2(\RR^N)\right)$, $f_2\in \Wsp'$, one has:
 \begin{multline}
  \label{G90}
  \left\|u\indxt\right\|_{L^{\frac{2(N+1)}{N-2}}(\RR\times \RR^N)}+\sup_{t\in \RR}\left\|\indxt \nabla_{t,x}u(t)\right\|_{L^2}\\ 
  \leq C_M\left( \|(u_0,u_1)\|_{\HHH}+\left\| \indxt f_1\right\|_{L^1(\RR,L^2)}+\left\| \indxt f_2\right\|_{\Wsp'} \right).
 \end{multline} 
 If $N\in \{3,4,5\}$, one also has
 \begin{multline}
  \label{G91}
  \left\|u\indxt\right\|_{L^{\frac{N+2}{N-2}}\left( \RR,L^{\frac{2(N+2)}{N-2}} \right)}
  \\
  \leq C_M\left( \|(u_0,u_1)\|_{\HHH}+\left\| \indxt f_1\right\|_{L^1(\RR,L^2)}+\left\| \indxt f_2\right\|_{\Wsp'}\right).
 \end{multline}
 Finally there exists $g\in L^2([0,+\infty))$ such that 
\begin{equation}
\label{radiation}
 \lim_{t\to \infty} \int_{t}^{+\infty} \left| r^{\frac{N-1}{2}} \partial_r u(t,r)-g(r-t)\right|^2\,dr=0.
\end{equation} 
 \end{lemma}
\begin{proof}
 By Strichartz inequality, for all $T>0$,
 \begin{multline*}
  \left\|\indxt u\right\|_{L^{\frac{2(N+1)}{N-2}}\left( [0,T]\times \RR^N \right)}
  \lesssim \|(u_0,u_1)\|_{\HHH}\\+\left\| \indxt f_1\right\|_{L^1(\RR,L^2)}+\left\| \indxt f_2\right\|_{\Wsp'}+\left\| \indxt Vu\right\|_{L^1\left((0,T),L^2\right)}.
 \end{multline*}
Using H\"older inequality in the space variable, we deduce
\begin{multline*}
  \left\|\indxt u\right\|_{L^{\frac{2(N+1)}{N-2}}\left( [0,T]\times \RR^N \right)}\\
  \lesssim \|(u_0,u_1)\|_{\HHH}+\left\| \indxt f_1\right\|_{L^1(\RR,L^2)}+\left\| \indxt f_2\right\|_{\Wsp'}\\
  +\int_0^{T}\left\|\indxt V\right\|_{L^{\frac{2(N+1)}{3}}}\left\| \indxt u\right\|_{L^{\frac{2(N+1)}{N-2}}}\,dt,
\end{multline*}
and thus, using a Gr\"onwall type lemma (Lemma 8.1 of \cite{FaXiCa11}) we obtain
\begin{multline*}
 \left\|\indxt u\right\|_{L^{\frac{2(N+1)}{N-2}}\left( [0,T]\times \RR^N \right)}\\
 \leq C_M\left( \|(u_0,u_1)\|_{\HHH}+ \left\| \indxt f_1\right\|_{L^1(\RR,L^2)}+\left\| \indxt f_2\right\|_{\Wsp'}\right).
\end{multline*}
Using Strichartz and H\"older's inequalities again we deduce the rest of \eqref{G90} and \eqref{G91}.

By an argument similar to the one in the proof of \cite[Lemma 2.8]{DuKeMe19Pc}, one can prove that there exists a solution $u_F$ of the free wave equation such that 
$$\lim_{t\to +\infty}\int_{|x|>|t|} \left|\nabla_{t,x}(u-u_F)\right|^2\,dx=0.$$
Since there exists $g\in L^2(\RR)$ such that
\begin{equation*}
 \lim_{t\to +\infty} \int_{0}^{+\infty} \left| r^{\frac{N-1}{2}} \partial_r u_F(t,r)-g(r-t)\right|^2\,dr=0,
\end{equation*} 
(see e.g. the appendix of \cite{DuKeMe19}),
the property \eqref{radiation} follows.
\end{proof}

\section{Channels of energy for the linearized operator close to a multisoliton}
\label{S:multisoliton}
This section is devoted to the proof of an exterior energy bound, stated in \S \ref{SS:exterior_multi}, for the equation \eqref{NLW} linearized around a multisoliton. We start (see Subsection \ref{SS:exterior_free}) by recalling previous results obtained in \cite{DuKeMe12b} and \cite{DuKeMe19Pa}, on exterior energy bounds for the free wave equation and the linearized equation around a single soliton. 
\subsection{Channels of energy for the free and the linearized wave equations}
\label{SS:exterior_free}
In \cite{DuKeMe12b}, we have obtained the following exterior energy lower bound for solutions of the free wave equation:
\begin{theoint}
\label{T:equirepartition}
 Assume $N$ is odd. Let $u_F$ be a solution of the free wave equation:
 \begin{equation}
 \label{FW}
 \partial_t^2u_F-\Delta u_F=0
\end{equation} 
 with initial data in $(u_0,u_1)\in \HHH$. Then
 $$
 \sum_{\pm}\lim_{t\to \pm \infty}\int_{|x|>|t|}|\nabla_{t,x}u_F|^2\,dx=\|(u_0,u_1)\|^2_{\HHH}.$$
\end{theoint}
Let $W$ be the ground-state stationary solution of \eqref{NLW},  given by \eqref{defW}. Consider the linearized equation:
\begin{equation}
 \label{LW}
 \partial_t^2u+L_Wu=0,
\end{equation} 
where $L_W$ is the linearized operator:
\begin{equation}
 \label{def_LW}
 L_W=-\Delta-\frac{N+2}{N-2}W^{\frac{4}{N-2}}.
\end{equation} 
The existence and uniqueness of solutions of \eqref{LW} with initial data in $\HHH$ can be easily proved by standard semi-group theory.
In \cite{DuKeMe19Pa}, we have proved an analog of Theorem \ref{T:equirepartition} for solutions of \eqref{LW} that we will now describe. To lighten notations, we will restrict to radial functions in space dimension $N\geq 5$. Let 
$$\Lambda W:=x\cdot \nabla W+\left( \frac{N}{2}-1 \right)W.$$
Then 
$$\vect\big\{\Lambda W\big\}=\left\{Z\in \hdot_{\rad}\;:\; L_WZ=0\right\}.$$
Indeed the inclusion $\subset$ is due to the fact that \eqref{NLW} is invariant by scaling. The other inclusion is a well-known nondegeneracy property of $W$ (see e.g. \cite{Rey90}). Note that $\Lambda W\in L^2$ since $N\geq 5$.
Let 
$$\ZZZZ:=\vect\big\{\Lambda W\big\}\times \vect\big\{\Lambda W\big\}.$$
If $(u_0,u_1)\in \ZZZZ$, then  the solution $u$ of \eqref{LW} with initial data $(u_0,u_1)$ is given by $u(t,x)=u_0(x)+tu_1(x)$ and in particular
$$\sum_{\pm} \lim_{t\to \pm\infty} \int_{|x|\geq |t|}|\nabla_{t,x}u(t,x)|^2=0.$$
If $V$ is a closed subspace of $\HHH$, we denote by $V^{\bot}$ its orthogonal in $\HHH$, and $\pi_{V}$ the orthogonal projection on $V$. Theorem 1 of \cite{DuKeMe19Pa} states that the solutions with initial data in $\ZZZZ$ are the only solutions that do not satisfy an exterior energy lower bound:
\begin{theoint}
\label{T:LW}
 Assume $N\geq 5$ is odd. Then there exists a constant $C>0$ such that for all $(u_0,u_1)\in \HHH_{\rad}$,
 \begin{equation}
  \label{A2'}
 \left\|\pi_{\ZZZZ^{\bot}}(u_0,u_1)\right\|^2 \leq C \sum_{\pm} \lim_{t\to \pm\infty} \int_{|x|\geq |t|}|\nabla_{t,x}u(t,x)|^2,
 \end{equation} 
 where $u$ is the solution of \eqref{LW} with initial data $(u_0,u_1)$.
\end{theoint}
Theorem 1 of \cite{DuKeMe19Pa} is indeed more general: it holds without the assumption that $(u_0,u_1)$ is radial, and also in space dimension $N=3$, with a suitable definition of $\ZZZZ$. We refer to \cite{DuKeMe19Pa} for the details.

\subsection{Bound from below of the exterior energy close to a multisoliton}
\label{SS:exterior_multi}
As a corollary of Theorem \ref{T:LW}, we will prove an exterior energy lower bound for the linearized operator close to an approximate radial multisoliton solution. We will consider only radial solutions, and fix $J\geq 2$.

We denote by $G_J$ the subset of $(0,\infty)^J$:
\begin{equation}
\label{defGJ}
G_J=\left\{\lambdabf=(\lambda_j)_{1\leq j\leq J},\quad 0<\lambda_J<\lambda_{J-1}<\ldots <\lambda_1\right\}. 
\end{equation} 
If $\lambdabf\in G_J$, we denote 
\begin{gather}
 \label{G10}
 \gamma(\lambdabf)=\max_{2\leq j\leq J} \frac{\lambda_j}{\lambda_{j-1}} \in (0,1)\\
\label{G11}
 L_{\lambdabf}=-\Delta -\sum_{j=1}^J \frac{N+2}{N-2} W_{(\lambda_j)}^{\frac{4}{N-2}}\\
 \label{G12}
Z_{\lambdabf}=\vect\Big\{\left((\Lambda W)_{(\lambda_j)},0\right),\left(0,(\Lambda W)_{[\lambda_j]}\right)\Big)
\end{gather}
(see Subsection \ref{SS:notations} for the notations $(\Lambda W)_{(\lambda_j)}$ and $(\Lambda W)_{[\lambda_j]}$.)
Then:
\begin{corol}
 \label{Cor:G10}
 Assume $N\geq 5$ is odd.
For any $J\geq 2$, there exists $\gamma_{*}>0$ and $C>0$ with the following property. For any $\lambdabf$ with $\gamma(\lambdabf)\leq \gamma_*$, for any solution $u$ of 
\begin{equation}
 \label{G13}
 \partial_t^2u+L_{\lambdabf}u=0,\quad \vec{u}_{\restriction t=0}=(u_0,u_1)\in \HHH,
\end{equation} 
one has
\begin{multline}
 \label{G14}
 \left\|\pi_{Z_{\lambdabf}^{\bot}}(u_0,u_1)\right\|^2_{\HHH} \\
 \leq C\left(\sum_{\pm} \lim_{t\to \pm\infty} \int_{|x|\geq |t|} |\nabla_{t,x}u(t,x)|^2\,dx+\gamma(\lambdabf)^{2\theta_N}\left\|\pi_{Z_{\lambdabf}}(u_0,u_1)\right\|^2_{\HHH}\right),
\end{multline} 
where $\theta_5=\frac 12$, $\theta_7=\frac 32$ and $\theta_N=2$ if $N\geq 9$.
 \end{corol}
Corollary \ref{Cor:G10} also has a version for $N=3$, that we will not need here. We skip it for the sake of simplicity.
We prove Corollary \ref{Cor:G10} in the next subsection. In \S \ref{SS:close_multi}, we will apply this corollary to a solution of \eqref{NLW} close to a multisoliton manifold.
\subsection{Proof of the exterior energy lower bound for the linearized equation}
We prove Corollary \ref{Cor:G10} by contradiction. For this, we assume that there exists a sequence $\left\{\lambdabf_n\right\}_n$ with 
\begin{equation}
 \label{G15}
 \lim_{n\to\infty} \gamma(\lambdabf_n)=0,
\end{equation} 
and a sequence $\{(u_{0,n},u_{1,n})\}_n$ in $\HHH$ such that, denoting by $u_n$ the solution of 
\begin{equation}
 \label{G20}
 \partial_t^2u_n+L_{\lambdabf_n}u_n=0,\quad \vec{u}_{n\restriction t=0}=(u_{0,n},u_{1,n}),
\end{equation} 
one has
\begin{multline}
 \label{G21}
 \lim_{n\to \infty} \sum_{\pm}\lim_{t\to \pm\infty}\int_{|x|\geq |t|} |\nabla_{t,x}u_n(t,x)|^2\,dx\\
 +\gamma(\lambdabf_n)^{2\theta_N} \left\|\pi_{Z_{\lambdabf_n}} (u_{0,n},u_{1,n})\right\|_{\HHH}^2=0
\end{multline}
and 
\begin{equation}
 \label{G22}
 \left\|\pi_{Z_{\lambdabf_n}^{\bot}}(u_{0,n},u_{1,n})\right\|_{\HHH}=1.
\end{equation}
\setcounter{step}{0}
\begin{step}[Projection on the orthogonal of the singular directions]
\label{S:Gprojection}
 Let $v_n$ be the solution of 
 \begin{equation}
  \label{G23}
  \partial_t^2v_n+L_{\lambdabf_n}v_n=0,\quad \vec{v}_{n\restriction t=0}=\pi_{Z_{\lambdabf_n}^{\bot}}(u_{0,n},u_{1,n}).
 \end{equation} 
 We claim
 \begin{equation}
  \label{G24}
  \lim_{n\to \infty}\sum_{\pm}\lim_{t\to \pm\infty}\int_{|x|\geq |t|} |\nabla_{t,x}v_n(t,x)|^2\,dx=0.
 \end{equation} 
 In view of \eqref{G21}, it is sufficient to prove
 \begin{equation}
  \label{G25}
  \lim_{n\to \infty}\sum_{\pm}\lim_{t\to \pm\infty}\int_{|x|\geq |t|} |\nabla_{t,x}(u_n-v_n)(t,x)|^2\,dx=0.
 \end{equation} 
 By definition of $Z_{\lambdabf_n}$, we can write:
 \begin{equation}
  \label{G26}
  \pi_{Z_{\lambdabf_n}}(u_{0,n},u_{1,n})=\sum_{j=1}^{J}\left( \alpha_{j,n}(\Lambda W)_{(\lambda_{j,n})},\beta_{j,n}(\Lambda W)_{[\lda_{j,n}]} \right),
 \end{equation} 
 where by \eqref{G21}, 
 \begin{equation}
  \label{G27}
  \forall j\in \llbracket 1,J\rrbracket,\quad \lim_{n\to\infty} \left( |\alpha_{j,n}|+|\beta_{j,n}| \right)\gamma(\ldabf_n)^{\theta_N}=0.
 \end{equation} 
 We consider 
 \begin{equation}
  \label{G28}
  w_n=\sum_{j=1}^{J}\left( \alpha_{j,n}\left( \Lambda W \right)_{(\lda_{j,n})}+t\beta_{j,n}\left( \Lambda W \right)_{[\lda_{j,n}]} \right)
 \end{equation} 
 and prove that $w_n$ is, outside the wave cone, an approximate solution of the li\-nearized equation around the multisoliton in the following sense:
 \begin{equation}
  \label{G30}
  \lim_{n\to\infty} \left\| \indic_{\{|x|\geq |t|\}} \left( \partial_t^2+L_{\lambdabf_n} \right)w_n\right\|_{L^1(\RR,L^2)}=0.
 \end{equation} 
 Indeed 
 \begin{multline*}
  \left\| \indic_{\{|x|\geq |t|\}} \left( \partial_t^2+L_{\lambdabf_n} \right)w_n\right\|_{L^1(\RR,L^2)}\\
  \lesssim 
  \sum_{\substack{1\leq j\leq J\\ k\neq j}}|\alpha_{j,n}|\left\| \indic_{\{|x|\geq |t|\}} W_{(\lambda_{k,n})}^{\frac{4}{N-2}}\left( \Lambda W \right)_{(\lambda_{j,n})}\right\|_{L^1(\RR,L^2)}\\
  + \sum_{\substack{1\leq j\leq J\\ k\neq j}}|\beta_{j,n}|\left\| t\indic_{\{|x|\geq |t|\}} W_{(\lambda_{k,n})}^{\frac{4}{N-2}}\left( \Lambda W \right)_{[\lambda_{j,n}]}\right\|_{L^1(\RR,L^2)}.
 \end{multline*}
By Claim \ref{C:G30} in the appendix, we have
\begin{multline}
\label{Gbound}
 \left\| \indic_{\{|x|\geq |t|\}} W_{(\lambda_{k,n})}^{\frac{4}{N-2}}\left( \Lambda W \right)_{(\lambda_{j,n})}\right\|_{L^1(\RR,L^2)}\\
 +\left\| t\indic_{\{|x|\geq |t|\}} W_{(\lambda_{k,n})}^{\frac{4}{N-2}}\left( \Lambda W \right)_{[\lambda_{j,n}]}\right\|_{L^1(\RR,L^2)}
\lesssim \left(\gamma(\lambdabf_n)\right)^{\theta_N},
\end{multline} 
which yields \eqref{G30} in view of \eqref{G27}. To conclude Step \ref{S:Gprojection}, we see that \eqref{G30} implies
\begin{equation*}
  \lim_{n\to\infty} \left\| \indic_{\{|x|\geq |t|\}} \left( \partial_t^2+L_{\lambdabf_n} \right)(u_n-v_n-w_n)\right\|_{L^1(\RR,L^2)}=0,
 \end{equation*} 
 and since $(\vec{u}_n-\vec{v}_n-\vec{w}_n)_{\restriction t=0}=0$, \eqref{G25} follows from Lemma \ref{L:linear_approx}, the fact that $w_n$ satisfies
  \begin{equation*}
  \lim_{n\to \infty}\sum_{\pm}\lim_{t\to \pm\infty}\int_{|x|\geq |t|} |\nabla_{t,x}w_n(t,x)|^2\,dx=0,
 \end{equation*} 
 and the following bounds
 \begin{gather}
  \label{G10A}
  \indic_{|x|\geq |t|} W^{\frac{4}{N-2}}\in L^{\frac{2(N+1)}{N+4}}\left(\RR,L^{\frac{2(N+1)}{3}}\right)\\
  \label{G100}
  \left\|\indic_{|x|\geq |t|} \sum_{j=1}^J W_{(\lambda_{j,n})}^{\frac{4}{N-2}}\right\|_{L^{\frac{2(N+1)}{N+4}}\left(\RR,L^{\frac{2(N+1)}{3}}\right)}\leq J \left\|\indic_{|x|\geq |t|} W^{\frac{4}{N-2}}\right\|_{L^{\frac{2(N+1)}{N+4}}\left(\RR,L^{\frac{2(N+1)}{3}}\right)}.
 \end{gather} 
 The bound \eqref{G100} follows from \eqref{G10A} and scaling invariance. To prove \eqref{G10A}, we use the bound
 $$|W(x)|^{\frac{4}{N-2}}\lesssim \min\left( 1,\frac{1}{|x|^4} \right).$$
 This proves that  $W^{\frac{4}{N-2}}\in L^{\frac{2(N+1)}{3}}(\RR^N)$ and 
 \begin{equation*}
  \left\|\indxt W^{\frac{4}{N-2}}\right\|_{L^{\frac{2(N+1)}{3}}_x}^{\frac{2(N+1)}{3}}\lesssim \int_t^{\infty} \frac{1}{r^{\frac{8(N+1)}{3}}}r^{N-1}\,dr\lesssim \frac{1}{t^{\frac{5N+8}{3}}}.
 \end{equation*}
Hence
 \begin{equation*}
  \left\|\indxt W^{\frac{4}{N-2}}\right\|^{\frac{2(N+1)}{N+4}}_{L^{\frac{2(N+1)}{3}}_x}\lesssim \frac{1}{t^{\frac{5N+8}{N+4}}},
 \end{equation*}
 which yields \eqref{G10A} and concludes this step.
 \end{step}
\begin{step}[Profile decomposition]
\label{St:Profile}
As it is recalled in Subsection \ref{SS:profile}, extracting subsequences, we can assume that the sequence $\big\{(v_{0,n},v_{1,n}\big\}$ has a profile decomposition with profiles $\{U^j_F\}_{j\geq 1}$, and parameters $\{\lambda_{j,n}\}_n\in (0,\infty)^{\NN}$ and $\{t_{j,n}\}_n\in \RR^{\NN}$.
We denote by $U_{F,n}^j$ the rescaled linear profiles, defined in \eqref{rescaled_lin}, and by 
$$ w_n^K=S_L(t)(v_{0,n},v_{1,n})-\sum_{j=1}^K U_{F,n}^j,$$
the remainder, so that 
\begin{equation}
 \label{G112}
\lim_{K\to\infty} \limsup_{n\to\infty} \left\|w_n^K\right\|_{\Ssp(\RR)}=0.
 \end{equation} 
Reordering the profiles, we can assume 
$$1\leq j\leq J\Longrightarrow t_{j,n}=0,$$
(where $J$ is the number of solitons),
and that for $1\leq j\leq J$, the parameters $\lambda_{j,n}$ are the same $\lambda_{j,n}$ as in the beginning of the proof. Indeed, 
one can define the $J$ first profiles by
\begin{equation*}
 \vec{U}^j_F(0)=
\underset{n\to \infty}{\wlim} 
 \left( \lambda_{j,n}^{\frac{N}{2}-1}v_{0,n}\left(\lambda_{j,n}\cdot \right),\lambda_{j,n}^{\frac{N}{2}}v_{1,n}^J\left(\lambda_{j,n}\cdot \right)\right) \text{ in }\HHH,
\end{equation*} 
then carry on with the profile decomposition to extract all the other profiles.
Of course in doing so we do not exclude the fact that some of the profiles $U^j_F$, $1\leq j\leq J$, might be identically $0$. 

We will approximate $v_n$ as follows. If $1\leq j\leq J$, we let $U^j$ be the solution of 
\begin{equation}
 \label{G113} \left( \partial_t^2+L_W \right)U^j=0,\quad \vec{U}^j(0)=\vec{U}_F^j(0).
\end{equation} 
If $j\geq J+1$, we let $U^j=U^j_F$. We define
\begin{equation}
 \label{G12A}
 U^j_n(t,x)=\frac{1}{\lambda_{j,n}^{\frac N2-1}}U^j\left( \frac{t-t_{j,n}}{\lambda_{j,n}},\frac{x}{\lambda_{j,n}} \right),\quad
 v_n^K(t,x)=\sum_{j=1}^K U^j_n+w_n^K(t,x).
\end{equation} 
In this step we prove:
\begin{equation}
 \label{G120}
 \lim_{K\to \infty}\limsup_{n\to\infty}\left( \sup_{t\in \RR} \left\|\indxt \nabla_{t,x}\left( v_n(t,x)-v_n^K(t,x) \right)\right\|_{L^2} \right)=0.
\end{equation} 
Denote by $r_n^K=v_n-v_n^K$. Then
$$\partial_t^2r_n^K+L_{\lambdabf_n} r_n^K=-\sum_{j=1}^K (\partial_t^2+L_{\lambdabf_n})U_n^j-(\partial_t^2+L_{\lambdabf_n})w_n^K,\quad \vec{r}^K_{n\restriction t=0}=(0,0).$$
If $1\leq j\leq J$ we have by \eqref{G113},
\begin{equation*}
(\partial_t^2+L_{\lambdabf_n})U_n^j=-\frac{N+2}{N-2}\sum_{\substack{1\leq k\leq J\\ k\neq j}} W_{(\lambda_{k,n})}^{\frac{4}{N-2}}U_n^j.
\end{equation*} 
If $j\geq J+1$, then
\begin{equation*}
(\partial_t^2+L_{\lambdabf_n})U_n^j=-\frac{N+2}{N-2}\sum_{1\leq k\leq J} W_{(\lambda_{k,n})}^{\frac{4}{N-2}}U_n^j.
\end{equation*} 
Finally, for all $K\geq 1$,
\begin{equation*}
 (\partial_t^2+L_{\lambdabf_n})w_n^K=-\frac{N+2}{N-2} \sum_{1\leq k \leq J} W_{(\lambda_{k,n})}^{\frac{4}{N-2}} w_n^K.
\end{equation*} 
Using the pseudo-orthogonality \eqref{psdo_orth} of the parameters and the property \eqref{G112} of $w_n^K$, we obtain 
$$ \lim_{K\to\infty} \lim_{n\to\infty} \left\| (\partial_t^2+L_{\lambdabf_n})r_n^K\indxt\right\|_{L^1(\RR,L^2)}=0.$$
By \eqref{G100} and the approximation Lemma \ref{L:linear_approx}, we deduce \eqref{G120}.
\end{step}
\begin{step}[End of the proof]
 Using the profile decomposition of the preceding step, we prove the corollary. 
We claim
\begin{gather}
 \label{G122}
 \forall j\geq 1,\quad \left\|\vec{U}^j_n(0)\right\|_{\HHH}^2\lesssim \sum_{\pm}\lim_{t\to\pm\infty} \int_{|x|> |t|} |\nabla_{t,x}U^j_n(t,x)|^2\,dx\\
 \label{G123}
 \forall K\geq 1,\quad  \left\|w_n^K(0)\right\|_{\HHH}^2\lesssim \sum_{\pm}\lim_{t\to\pm\infty}\int_{|x|\geq |t|}\left|\nabla_{t,x}w_n^K(t,x)\right|^2\,dx,
 \end{gather}
 (where the implicit constants are independent of $j$, $K$, and $n$), and 
 \begin{gather}
\label{G124}
 \forall j\geq 1,\quad \lim_{n\to\infty}\lim_{t\to\pm\infty} \int_{|x|> |t|} |\nabla_{t,x}U^j_n(t,x)|^2\,dx=0\\
 \label{G125}
 \forall K\geq 1,\quad \sum_{\pm}\lim_{t\to\pm\infty}\int_{|x|\geq |t|}\left|\nabla_{t,x}w_n^K(t,x)\right|^2\,dx=0.
 \end{gather} 
 Of course, combining \eqref{G122}, \eqref{G123}, \eqref{G124}, \eqref{G125} and Step \ref{St:Profile} we would obtain
 $$\lim_{n\to\infty}\|\vec{v}_n(0)\|_{\HHH}=0,$$
 a contradiction with \eqref{G22}. It remains to prove these four assertions.
 
 Recall that $U^j$ (for $j\geq J+1$) and $w_n^K$ (for any $K\geq 1$) are solutions of the free wave equation. The inequalities \eqref{G122} for $j\geq J+1$, and \eqref{G123} thus follow from the exterior energy bound in odd dimension proved in \cite{DuKeMe12} (recalled in Theorem \ref{T:equirepartition} above). We next prove \eqref{G122} when $j$ satisfies $1\leq j\leq J$. According to the channels of energy for the linearized equation at $W$ (Theorem \ref{T:LW}), it is sufficient to prove
 \begin{equation}
  \label{G121}
  \int \nabla U^{j}(0,x)\cdot \nabla \Lambda W(x)\,dx=\int \partial_tU^j(0,x)\Lambda W(x)=0.
 \end{equation}
To prove \eqref{G121}, notice that by weak convergence:
\begin{align*}
 \int \nabla U^{j}(0,x)\cdot \nabla \Lambda W(x)\,dx&=\lim_{n\to\infty} \int \lambda_{j,n}^{\frac N2} \nabla v_{0,n}(\lambda_{j,n}x)\Lambda W(x)\,dx\\
&=\lim_{n\to\infty} \int \nabla v_{0,n} \nabla (\Lambda W)_{(\lambda_{j,n})}=0,
\end{align*}
since $(v_{0,n},v_{1,n})\in Z_{\lambdabf_n}^{\bot}$. By the same proof, $\int \partial_tU^j(0,x)\Lambda W(x)\,dx=0$, concluding the proof of \eqref{G121} and thus of \eqref{G122}.

We next prove \eqref{G124} and \eqref{G125}. We will use the pseudo-orthogonality of the parameters \eqref{psdo_orth}. We focus on the limits as $t\to +\infty$, the proof for the limits as $t\to -\infty$ is the same. Using the radiation term for the free wave equation (see appendix of \cite{DuKeMe19}) if $j\geq J+1$, or for the linearized wave equation (see \eqref{radiation}) if $1\leq j\leq J$, we obtain that for all $j\geq 1$, there exists $g^j\in L^2(\RR)$ such that
\begin{gather}
\label{G130}
 \lim_{t\to \infty} \int_{t}^{+\infty} \left| r^{\frac{N-1}{2}} \partial_r U^j(t,r)-g^j(r-t)\right|^2\,dr=0\\
\label{G131}
 \lim_{t\to \infty} \int_{t}^{+\infty} \left| r^{\frac{N-1}{2}} \partial_t U^j(t,r)+g^j(r-t)\right|^2\,dr=0\\
 \lim_{t\to \infty} \int_{t}^{+\infty} \frac 1{r^2}\left|U^j(t,r)\right|^2r^{N-1}\,dr=0.
\end{gather}
If $j\geq J+1$, the preceding limits hold true with $\int_t^{\infty}$ replaced by $\int_0^{\infty}$. Also, for all $K\geq 1$ and $n$, there exist $G_n^K\in L^2(\RR)$ such that
\begin{align}
 \label{G133}
\lim_{t\to\infty} \int_0^{+\infty} \left| r^{\frac{N-1}{2}} \partial_rw_n^K(t,r)-G_n^K(r-t)\right|^2\,dr&=0\\
\label{G134}
\lim_{t\to\infty} \int_0^{+\infty} \left| r^{\frac{N-1}{2}} \partial_tw_n^K(t,r)+G_n^K(r-t)\right|^2\,dr&=0
\end{align}
 Fix $j \geq 1$ and $\eps>0$. Then there exists $K\gg 1$ such that $K>j$ and (by \eqref{G120} and \eqref{G24})
\begin{equation}
\label{G135}
 \lim_{t\to+\infty} \int_{|x|>t} \left|\partial_{t,x}v_n^K(t,x)\right|^2\,dx \leq \eps.
\end{equation}
Using the definition \eqref{G12A} of $v_n^K$, we obtain
\begin{multline*}
 \int_{|x|>t}\nabla_{t,x}v_n^K(t)\cdot\nabla_{t,x} U_n^j(t)=\int_{|x|>t} |\nabla_{t,x}U_n^j(t)|^2\\
+ \sum_{\substack{1\leq k\leq K\\ j\neq k}} \int_{|x|>|t|} \nabla_{t,x}U_n^j(t)\cdot \nabla_{t,x}U_n^k(t)+\int_{|x|>|t|} \nabla_{t,x}U_n^j(t)\nabla_{t,x}w_n^K(t),
\end{multline*}
 where $U_n^j$ is as usual the modulated profile (see \eqref{G12A}).

If $j\neq k$ we have, in view of \eqref{G130}, \eqref{G131},
\begin{multline*}
\lim_{t\to +\infty}\int_{|x|>t} \nabla_{t,x}U_n^j(t)\cdot\nabla_{t,x}U_n^k(t)\\
=2\lim_{t\to\infty} \int_{t}^{+\infty} \frac{1}{\lambda_{j,n}^{1/2}} g^j\left( \frac{r-|t-t_{j,n}|}{\lambda_{j,n}}\right)\frac{1}{\lambda_{k,n}^{1/2}} g^k\left( \frac{r-|t-t_{k,n}|}{\lambda_{k,n}}\right)\,dr\\
=2\int_0^{+\infty} \frac{1}{\lambda_{j,n}^{1/2}}g^j\left( \frac{r+t_{j,n}}{\lambda_{j,n}} \right)\frac{1}{\lambda_{k,n}^{1/2}}g^k\left( \frac{r+t_{k,n}}{\lambda_{k,n}} \right)\,dr.
\end{multline*}
In view of the pseudo-orthogonality \eqref{psdo_orth} of the parameters, we deduce
\begin{equation} 
 \label{G150}
\lim_{n\to\infty}
\lim_{t\to +\infty}\int_{|x|>t} \nabla_{t,x}U_n^j(t)\cdot\nabla_{t,x}U_n^k(t)=0.
\end{equation} 
Next we consider
\begin{multline*}
\lim_{t\to +\infty}\int_{|x|>t} \nabla_{t,x}U_n^j(t)\cdot\nabla_{t,x}w_n^K(t)\\
=2\lim_{t\to\infty} \int_{t}^{+\infty} \frac{1}{\lambda_{j,n}^{1/2}} g^j\left( \frac{r-|t-t_{j,n}|}{\lambda_{j,n}}\right)G^K_{n}\left( r-t\right)\,dr\\
=2\int_0^{+\infty} \frac{1}{\lambda_{j,n}^{1/2}}g^j\left( \frac{r+t_{j,n}}{\lambda_{j,n}} \right)G_n^K(r)\,dr.
\end{multline*}
If $t_{j,n}=0$ for all $n$ we obtain
$$\lim_{t\to +\infty}\int_{|x|>t} \nabla_{t,x}U_n^j(t)\cdot\nabla_{t,x}w_n^K(t)=2\int_0^{+\infty} \frac{1}{\lambda_{j,n}^{1/2}} g^j\left( \frac{r}{\lambda_{j,n}} \right)G_n^K(r)\,dr,$$
and the right-hand side goes to $0$ as $n$ goes to infinity since the condition 
$$ \underset{n\to \infty}{\wlim} \left(\lambda_{j,n}^{\frac{N}{2}-1}w_n^K(0,\lambda_{j,n}\cdot),\lambda_{j,n}^{\frac{N}{2}}\partial_t w_n^K(0,\lambda_{j,n}\cdot)\right)=0$$
(see \eqref{wlim_w})
implies
\begin{equation}
 \label{G160}
\forall g\in L^2(\RR), \quad\lim_{n\to\infty}\int_{\RR}\frac{1}{\lambda_{j,n}^{1/2}} g\left( \frac{r}{\lambda_{j,n}} \right)G_n^K(r)\,dr=0.
\end{equation} 
If $\lim_{n}t_{j,n}/\lambda_{j,n}=+\infty$, we have
\begin{equation*}
 \left|\int_0^{+\infty} \frac{1}{\lambda_{j,n}^{1/2}}g^j\left( \frac{r+t_{j,n}}{\lambda_{j,n}} \right)G_n^K(r)\,dr\right|
\lesssim \|g^j\|_{L^2(r\geq t_{j,n}/\lambda_{j,n})}\underset{n\to \infty}{\longrightarrow} 0
\end{equation*}
Finally, if $\lim_{n\to\infty} t_{j,n}/\lambda_{j,n}=-\infty$, we have
$$ \lim_{n\to\infty} \int_{-\infty}^0 \left|\frac{1}{\lambda_{j,n}^{1/2}}g^j\left( \frac{t+t_{j,n}}{\lambda_{j,n}} \right)\right|^2\,dr=0,$$
and thus 
\begin{multline*}
\lim_{n\to\infty} \int_0^{+\infty} \frac{1}{\lambda_{j,n}^{1/2}}g^j\left( \frac{r+t_{j,n}}{\lambda_{j,n}} \right)G_n^K(r)\,dr
\\
=\lim_{n\to\infty}\int_{\RR}\frac{1}{\lambda_{j,n}^{1/2}}g^j\left( \frac{r+t_{j,n}}{\lambda_{j,n}} \right)G_n^K(r)\,dr=0, 
\end{multline*}
where we have used that by the weak limit property \eqref{wlim_w} of $w_n^K$,
$$ \lim_{n\to \infty} \int \nabla_{t,x}U_n^j(0,x)\cdot\nabla_{t,x}w_n^K(0,x)\,dx=0.$$
combining the properties above, we obtain
$$\lim_{n\to\infty}\lim_{t\to +\infty} \int_{|x|\geq |t|} |\nabla_{t,x}U^j_n(t,x)|^2\,dx=\lim_{n\to\infty} \lim_{|x|>|t|} \int \nabla_{t,x}v_n^K\cdot \nabla_{t,x}U^j_n=0.$$
This yields \eqref{G124}. By a similar proof, we obtain \eqref{G125}, concluding this step.
\end{step}
\section{Non-radiative solutions close to a multisoliton}
\label{S:exterior}
\subsection{Preliminaries}
\begin{defi}
\label{D:non-radiative}
Let $t_0\in \RR$, and $u$ be a solution of the nonlinear wave equation \eqref{NLW} (or another wave equation considered in this paper). We say that $u$ is \emph{non-radiative} at $t=t_0$ if $u$ is defined on $\{|x|>|t-t_0|\}$ and
$$\sum_{\pm} \lim_{t\to\pm \infty} \int_{|x|\geq |t-t_0|}|\nabla_{t,x}u(t,x)|^2\,dx=0.$$
We say that $u$ is \emph{weakly non-radiative} if for large $R>0$, $u$ is defined on $\{|x|>|t|+R\}$ and
$$\sum_{\pm} \lim_{t\to\pm \infty} \int_{|x|\geq |t|+R}|\nabla_{t,x}u(t,x)|^2\,dx=0.$$
\end{defi}
If $N\geq 3$ is odd, according to the equirepartition property recalled in Theorem \ref{T:equirepartition}, the only non-radiative solution of \eqref{FW} is zero. This fact persists if $N\geq 4$ is even (see \cite[Proposition 1]{DuKeMe19Pc}), as a consequence of the asymptotic formula in \cite{CoKeSc14}. 

In odd space dimension $N\geq 5$, the non-radiative solutions for the linearized wave equation around the stationary solutions $W$ are also known, as a consequence of the main result of \cite{DuKeMe19Pa} (recalled in Theorem \ref{T:LW} above in the radial case).

Radial \emph{weakly non-radiative} solutions of the free wave equation were explicited in \cite{KeLaLiSc15}. Let
$$\PPP=\left\{\left(\frac{1}{r^{N-2k_1}},0\right),\left(0, \frac{1}{r^{N-2k_2}}\right),\quad 1\leq k_1\leq \ent{\frac{N+2}{4}},\; 1\leq k_2\leq \ent{\frac{N}{4}}\right\},$$
and, for $R>0$, $P(R)$ the subspace of $\HHH(R)$ spanned by $\PPP$. According to \cite{KeLaLiSc15}, if $N\geq 3$ is odd, $v$ is a radial solution $\partial_t^2v-\Delta v=0$, then
$$\sum_{\pm}\lim_{t\to\pm \infty} \int_{R+|t|}|\nabla_{t,x}u(t,x)|^2\,dx=0,$$
if and only if $\vec{v}(0)\in P(R)$.

From \cite{DuKeMe13}, if $N=3$, the only radial, non-radiative solutions of \eqref{NLW} are the stationary solutions. The proof is specific to dimension $3$, and the results available in higher dimension are less precise. We next recall from \cite{DuKeMe19Pc} some of these results, that will be needed in the sequel.
Let $m=\frac{N-1}{2}$ be the number of elements of $\PPP$.  
As in \cite{DuKeMe19Pc}, we denote by 
$\PPP=\{\Xi_{k}\}_{k\in \llbracket 1,m\rrbracket}$, choosing $\Xi_k$ so that
\begin{equation}
 \label{normXi}
 \left\|\Xi_k\right\|_{\HHH(R)}=\frac{c_k}{R^{k-\frac{1}{2}}},
\end{equation} 
for some constant $c_k\neq 0$. In particular, we choose $\Xi_m(r)=\left(r^{2-N},0\right)$.
By scaling, one can check that if $U\in P(R)$ and $\left( \theta_k(R) \right)_{k\in \llbracket 1,m\rrbracket}$ are its coordinates in $(\Xi_1,\ldots,\Xi_m)$, then
 $$\left\|U\right\|_{\HHH(R)}\approx \sum_{k=1}^m\frac{\left|\theta_{k}(R)\right|} {R^{k-1/2}},$$
 where the implicit constant is independent of $R>0$ (see \cite[Claim 3.2]{DuKeMe19Pc}). Then:
\begin{theoint}
 \label{T:asymptotic_NR}
Assume $N\geq 5$ is odd.
There exists $\eps_0>0$ with the following property. Let $u$ be a radial weakly non-radiative solution of \eqref{NLW}. Then there exist $k_0\in \llbracket 1,m\rrbracket$, $\ell \in \RR$ (with $\ell \neq 0$ if $k_0<m$) such that, for all $t_0\in \RR$ and $R_0>0$, if $
 u$ is defined on $\{(t,r),\quad r>|t-t_0|+R_0\}$, $\left\|\vec{u}(t_0)\right\|_{\HHH(R_0)}<\eps_0$ and
$$\sum_{\pm}\lim_{t\to \pm\infty} \int_{|x|>|t-t_0|+R_0} \left|\nabla_{t,x}u(t,r)\right|^2\,dx=0,$$
then
\begin{equation*}
 \forall R>R_0,\quad \left\|\vec{u}(t_0)-\ell \Xi_{k_0}\right\|_{\HHH(R)}\leq C\max\left\{\left(\frac{R_{0}}{R}\right)^{(k_0-\frac{1}2)\frac{N+2}{N-2}},\left(\frac{R_{0}}{R}\right)^{k_0+\frac 12}\right\}.
\end{equation*}
\end{theoint}
See Theorem 2 and Remark 3.4 in \cite{DuKeMe19Pc}. If $k_0=m$, the theorem implies that $u$ is close, for large $r$, to one of the stationary solutions $0$ (if $\ell=0$) or $\pm W_{(\lambda)}$ for some $\lambda$ depending on $\ell$ (if $\ell\neq 0$). Under the stronger assumption that $u$ is non-radiative, Theorem 3 of \cite{DuKeMe19Pc} gives a uniqueness result in this case. 
\begin{theoint}
 \label{T:uniqueness_NR}
Assume $N\geq 5$ is odd. Let $u$ be a radial non-radiative solution of \eqref{NLW}. Let $k_0$ be as in Theorem \ref{T:asymptotic_NR}. Assume that $k_0=m$. Then $u$ is a stationary solution.
\end{theoint}
In the remainder of this section, we will assume $N\geq 5$ is odd and consider a nonradiative solution close to a multisoliton. In Subsection \ref{SS:close_multi}, we will use the exterior energy bound for the linearized equation proved in Section \ref{S:multisoliton} to give a first order expansion of the solution. In Subsection \ref{SS:lwrbnd} we will use Theorems \ref{T:asymptotic_NR} and \ref{T:uniqueness_NR} to give a lower bound of the exterior scaling parameter of the multisoliton. These properties will be crucial in the proofs of the soliton resolution in Sections \ref{S:reduction} and \ref{S:end_of_proof}.
\subsection{Estimates on the coefficients}
\label{SS:close_multi}
In this subsection, we assume as before that $N\geq 5$ is odd. We fix $J\geq 1$, $(\iota_j)\in \{\pm 1\}^J$, and consider a radial solution $u$ of \eqref{NLW}, defined on $\left\{(t,x)\in \RR^N\;:\;|x|>t\right\}$, which is \emph{non-radiative} at $t=0$ (see Definition \ref{D:non-radiative}). We assume that there exists $\lambdabf=(\lambda_j)^J\in G_J$ such that:
\begin{gather}
\label{F160}
 \left\|\vec{u}(0)-(M,0)\right\|_{\HHH}=:\delta\leq \eps_J \ll 1\text{ where }M=\sum_{j=1}^J \iota_j W_{(\lambda_j)}\\
\label{F161'}
\gamma\leq \eps_J\ll 1,
 \end{gather}
where as before $\gamma:=\gamma(\ldabf)=\max_{1\leq j\leq J-1} \lambda_{j+1}/\lambda_j$. 
Denote 
$$h_0=u_0-M.$$
By the implicit function theorem (see Lemma \ref{L:ortho_scaling}), we can change the scaling parameters $(\lambda_j)_j$ so that the following orthogonality relations hold:
\begin{equation}
 \label{F162}
 \forall j\in \llbracket 1,J\rrbracket,\quad \int \nabla_x h_0\nabla_x (\Lambda W)_{(\lambda_j)}=0.
\end{equation} 
We expand $u_1=\partial_t u(0)$ as follows:
\begin{equation}
 \label{F164}
 u_1=\sum_{j=1}^J\alpha_j\left(\Lambda W\right)_{[\lambda_j]}+g_1
\end{equation} 
where
\begin{equation}
 \label{F165}
 \forall j\in \llbracket 1,J\rrbracket,\quad \int g_1 (\Lambda W)_{[\lambda_j]} =0.
\end{equation}
We will prove:
\begin{prop}
\label{P:F18}
\begin{gather}
\label{F180}
    \left\|(h_0,g_1)\right\|_{\HHH}\lesssim\gamma^{\frac N4}+\delta^{\frac{N}{N-2}}  \\
 \label{F181}
  \left| \delta^2-\sum_{j=1}^J \alpha_j^2\left\|\Lambda W\right\|^2_{L^2}\right|\lesssim \gamma^{\frac{N-1}{2}}+\delta^{\frac{2(N-1)}{N-2}}
\end{gather} 
  \end{prop}
  We start by proving the following lemma:
  \begin{lemma}
\label{L:radial_multisoliton_NL}
Let $u$ be as above. Then
$$ \left\| \pi_{Z_{\lambdabf}^{\bot}}\Big((u_0,u_1)-(M,0)\Big)\right\|_{\HHH}\lesssim \gamma^{\frac{N}{4}}+\delta^{\frac{N}{N-2}}.
$$
\end{lemma}
\begin{proof}
 We let $h(t)=u(t)-M$. Then $\|\vec{h}(0)\|_{\HHH}=\delta$ and
 \begin{equation}
  \label{MNL1}
  \partial_t^2h+L_{\ldabf} h=F(h)+\NNN(h),
 \end{equation} 
 where 
 $$\NNN(h)=F\left( M+h \right)-\sum_{j=1}^JF(\iota_j W_{(\lambda_j)})-F(h)-\frac{N+2}{N-2}\sum_{j=1}^J W^{\frac{4}{N-2}}_{(\lambda_j)}h.$$
 By finite speed of propagation, $h$ coincide, for $|x|>|t|$, with the solution $\tlh$ of
\begin{equation}
  \label{MNL1'}
  \partial_t^2\tlh+L_{\ldabf} \tlh=(F(h)+\NNN(h))\indxt,
 \end{equation} 
 Let $T>0$ and denote by $\Gamma(T)=\Big\{(t,x),\; |t|\leq \min \{|x|,T\} \Big\}$
 By the fractional chain rule \eqref{fractional_cones},
 \begin{equation}
  \label{MNL2}
  \left\|F(h)\indxt\right\|_{\Wsp'((0,T))}=\left\|F(\tlh)\indxt\right\|_{\Wsp'((0,T))}\lesssim \|\tlh\|_{\Wsp((0,T))}\|\tlh\|_{\Ssp(\Gamma(T))}^{\frac{4}{N-2}}.
 \end{equation} 
We first assume $N\geq 7$. We use the inequality
\begin{multline}
 \label{esti_linearisation}
\left|F\left( \sum_{j=1}^J y_j+h \right)-\sum_{j=1}^JF(y_j)-F(h)-\frac{N+2}{N-2}\sum_{j=1}^J |y_j|^{\frac{4}{N-2}}h\right|\\\
\lesssim
\sum_{j\neq k} \min\left(|y_j|^{\frac{4}{N-2}}|y_k|,|y_k|^{\frac{4}{N-2}}|y_j|\right)+\sum_{j=1}^J |y_j|\,|h|^{\frac{N+1}{N-2}},
\end{multline}
proved in the appendix (see Claim \ref{Cl:pointwise2}). We obtain
$$|\NNN(h)|\lesssim \sum_{j\neq k} \min\left(W_{(\lambda_j)}^{\frac{4}{N-2}}W_{(\lambda_k)}\right)+\sum_{j=1}^J W^{\frac{1}{N-2}}_{(\lambda_j)}|h|^{\frac{N+1}{N-2}}.$$
If $j\neq k$, we have, by Claim \ref{C:G30'} in the appendix,
$$ \left\|\indxt\min\left\{W_{(\lambda_j)}^{\frac{4}{N-2}}W_{(\lambda_k)},W_{(\lambda_j)}^{\frac{4}{N-2}}W_{(\lambda_k)}\right\}\right\|_{L^1_tL^2_x}\lesssim \gamma^{\frac{N+2}{4}}.$$
Furthermore,
\begin{multline*}
\left\|\indic_{\Gamma(T)} W^{\frac{1}{N-2}}_{(\lambda_j)}|h|^{\frac{N+1}{N-2}}\right\|_{L^1_tL^2}\\
\leq \left\|\indic_{\Gamma(T)} W^{\frac{1}{N-2}}_{(\lambda_j)}\right\|_{L^{2}_tL^{\infty}_x}\left\|\indic_{\Gamma(T)} |h|^{\frac{N+1}{N-2}}\right\|_{L^2_{t,x}}\lesssim \|\tlh\|^{\frac{N+1}{N-2}}_{\Ssp(\Gamma(T)}, 
\end{multline*}
where we have used that since
$$W^{\frac{1}{N-2}}\lesssim \frac{1}{1+|x|},$$
we have $W^{\frac{1}{N-2}}\indic_{\{|x|\geq t\}}\in L^{2}_t\left(\RR,L^{\infty}_x(\RR^N)\right)$.
We let $h_{L}(t)$ be the solution of 
\begin{equation*}
 \partial_t^2h_L+L_{\ldabf}h_L=0,\quad \vec{h}_{L\restriction t=0}=(u_0,u_1)-(M,0).
\end{equation*} 
In view of the estimate \eqref{G100} on the potential $\sum W_{(\lambda_j)}^{\frac{N+2}{N-2}}$, we can use the approximation Lemma (Lemma \ref{L:linear_approx}). Thanks to the estimates above, we obtain
\begin{equation*}
\left\|\tlh-h_L\right\|_{\Ssp(\Gamma_T)}\lesssim \gamma^{\frac{N+2}{4}}+\|\tlh\|_{\Ssp(\Gamma_T)}^{\frac{N+1}{N-2}}+\|\tlh\|_{\Ssp(\Gamma_T)}^{\frac{4}{N-2}}\|\tlh\|_{\Wsp((0,T)}.
\end{equation*} 
Using again Strichartz estimates, we deduce
\begin{multline}
\label{MNL3}
\sup_{-T\leq t\leq T}\|\vec{\tlh}(t)-\vec{h}_L(t)\|_{\HHH}+
\left\|\tlh-h_L\right\|_{\Wsp((0,T))\cap \Ssp((0,T))}\\
\lesssim \gamma^{\frac{N+2}{4}}+\|\tlh\|_{\Ssp(\Gamma_T)}^{\frac{N+1}{N-2}}+\|\tlh\|_{\Ssp(\Gamma_T)}^{\frac{4}{N-2}}\|\tlh\|_{\Wsp((0,T)}.
\end{multline} 
and thus, since $\|h_L\|_{\Wsp((0,T))\cap \Ssp((0,T))}\lesssim \delta$,
$$\|\tlh\|_{\Wsp((0,T))\cap \Ssp((0,T))}\lesssim \gamma^{\frac{N+2}{4}}+\delta.$$
Going back to \eqref{MNL3} we obtain
\begin{equation*}
\sup_{-T\leq t\leq T} \|\vec{\tlh}(t)-\vec{h}_L(t)\|_{\HHH}\lesssim \gamma^{\frac{N+2}{4}}+\delta^{\frac{N+1}{N-2}}
\end{equation*} 
This estimate is uniform in $T$. Hence
\begin{equation*}
\sup_{t\in \RR} \|\vec{\tlh}(t)-\vec{h}_L(t)\|_{\HHH}\lesssim \gamma^{\frac{N+2}{4}}+\delta^{\frac{N+1}{N-2}}.
\end{equation*} 
Using that $u$ is non-radiative, we deduce
$$\sum_{\pm}\left(\lim_{t\to\pm\infty} \int_{\{|x|>|t|\}} |\nabla_{t,x}h_L(t,x)|^2\,dx\right)^{1/2}\lesssim \gamma^{\frac{N+2}{4}}+\delta^{\frac{N+1}{N-2}}.$$
By Corollary \ref{Cor:G10}, 
$$\left\|\pi_{Z_{\lambdabf}^{\bot}}(h_0,h_1)\right\|_{\HHH} \lesssim \gamma^{\theta_N}\delta+\gamma^{\frac{N+2}{4}}+\delta^{\frac{N+1}{N-2}},$$
where $\theta_7=3/2$ and $\theta_N=2$ if $N\geq 9$. The conclusion of the proposition follows, noting that $\gamma^{\theta_N}\delta\lesssim \gamma^{\frac{N+2}{4}}+\delta^{\frac{N+1}{N-2}}$ if $N\geq 7$. Note that in this case the bound is slightly stronger, but we will not need this in the sequel.

The proof is almost the same when $N=5$, but we must replace the inequality \eqref{esti_linearisation} by 
\begin{equation*}
\left|F\left( \sum_{j=1}^J y_j+h \right)-\sum_{j=1}^JF(y_j)-\frac{7}{3}\sum_{j=1}^J |y_j|^{\frac{4}{3}}h\right|\
\lesssim
\sum_{j\neq k} |y_j|^{\frac{4}{3}}|y_k|+\sum_{j=1}^J |y_j|^{\frac 13} |h|^2+F(h),
\end{equation*}
and use that by Claim \ref{C:G30}, if $j\neq k$,
$$
\left\|\indxt W_{(\lambda_j)}^{\frac{4}{3}}W_{(\lambda_k)} \right\|_{L^1(\RR,L^2)} \lesssim 
\gamma^{\frac 32}.$$
We omit the details.
\end{proof}

\begin{proof}[Proof of the proposition]
   According to Lemma \ref{L:radial_multisoliton_NL},
 \begin{equation}
 \label{Ngq7_1}
 \left\|\pi_{Z_{\lambdabf}^{\bot}}(h_0,u_1)\right\|_{\HHH}\lesssim \delta^{\frac{N}{N-2}}+\gamma^{\frac{N}{4}}.
 \end{equation} 
 In view of the orthogonality condition \eqref{F162} and the expansion \eqref{F164} of $u_1$, we deduce \eqref{F180}.
 Since 
 $$\delta^{2}=\|h_0\|^2_{L^2}+\|g_1\|^2_{L^2} +\left\| \sum_{j=1}^J \alpha_j(\Lambda W)_{[\lambda_j]}\right\|^2_{L^2},$$
 and, by Claim \ref{Cl:estimates1},
$
  \int \left|(\Lambda W)_{[\lambda_j]}(\Lambda W)_{[\lambda_k]}\right|\lesssim \gamma^{\frac{N}{2}-2},
 $
 we obtain
 \begin{equation}
 \label{Ngq7_2}
 \left| \delta^2-\sum_{j=1}^J \alpha_j^2\left\|\Lambda W\right\|^2_{L^2}\right|\lesssim \gamma^{\frac{N}{2}-2}\sum_{j=1}^J \alpha_j^2+\delta^{\frac{2N}{N-2}}+\gamma^{\frac{N}{2}}.
 \end{equation} 
 Noting that the previous inequality implies easily $\sum_{j=1}^J \alpha_j^2\lesssim \delta^2+\gamma^{\frac{N}{2}}$,
 and thus
 $$\gamma^{\frac{N}{2}-2}\sum_{j=1}^J \alpha_j^2\lesssim \gamma^{N-2}+\delta^2\gamma^{\frac N2-2}\lesssim \gamma^{N-2}+\delta^{\frac{2(N-1)}{N-2}}+\gamma^{\frac{(N-4)(N-1)}{2}},$$
 we deduce \eqref{F181}. 
\end{proof}
\subsection{Lower bound for the exterior scaling parameter}
\label{SS:lwrbnd}
Let $u$ be as in Subsection \ref{SS:close_multi}, and denote by $\ell$ and $k_0$ the parameters defined by Theorem \ref{T:asymptotic_NR}. Assume that $u$ is nonstationary so that by Theorem \ref{T:uniqueness_NR}, $k_0\leq m-1$.
\begin{prop}
\label{P:ext_scaling}
There exists a constant $C>0$ such that, if $u$ is as above we have:
$$ |\ell|\leq C \delta^{\frac{2}{N}}\lambda_1^{k_0-\frac 12}.$$
\end{prop}
\begin{proof}
\setcounter{step}{0}
\begin{step}
\label{St:ext1}
 We note that for $R\geq \lambda_1$, we have
 $$\|\vec{u}(0)\|_{\HHH(R)}\lesssim \delta+\left( \frac{\lambda_1}{R} \right)^{m-\frac 12}.$$
 Indeed, $\|\vec{u}(0)-(M,0)\|_{\HHH(R)}\lesssim \delta$, and
 $$\|W_{(\lambda_j)}\|_{\hdot(R)}=\|W\|_{\hdot\left( \frac{R}{\lambda_j} \right)}\lesssim \left( \frac{\lambda_j}{R} \right)^{m-\frac 12},$$
 which yields the announced estimate.
\end{step}
\begin{step}
 Let $\eps_0$ be as in Theorem \ref{T:asymptotic_NR}. Fixing $B>0$ large enough, and using the smallness assumption \eqref{F160} on $\delta$, we see that 
 $$\|\vec{u}(0)\|_{\HHH(B\lambda_1)}\leq \eps_0.$$
 In view of Theorem \ref{T:asymptotic_NR}, we see that for all $R\geq B\lambda_1$,
 $$\left|\left\|\vec{u}(0)\right\|_{\HHH(R)}-\frac{c_{k_0}\ell}{R^{k_0-\frac 12}}\right|\lesssim \max\left\{\left( \frac{B\lambda_1}{R} \right)^{k_0+\frac 12},\left( \frac{B\lambda_1}{R} \right)^{\left( k_0-\frac 12 \right)\frac{N+2}{N-2}} \right\}.$$
 Combining with the estimate of Step \ref{St:ext1}, and using that $k_0\leq m-1$, we deduce that for all $R\geq B\lambda_1$,
 \begin{equation*}
  \frac{|\ell|}{R^{k_0-\frac 12}}\lesssim \delta+\left( \frac{\lambda_1}{R} \right)^{a_{k_0}},\quad a_{k_0}:=\min\left\{ k_0+\frac 12,\left(k_0-\frac 12\right)\frac{N+2}{N-2} \right\}.
 \end{equation*}
 Choosing $R$ such that $\left( \frac{\lambda_1}{R} \right)^{a_{k_0}}=\delta$, that is $R=\lambda_1\delta^{-\frac{1}{a_{k_0}}}$, we obtain 
 $$|\ell|\lesssim \lambda_1^{k_0-\frac 12}\delta^{1-\frac{k_0-\frac 12}{a_{k_0}}},$$
 which yields the conclusion of the proposition, since 
 $$ \min_{1\leq k_0\leq m-1}{1-\frac{k_0-\frac 12}{a_{k_0}}}=\frac{2}{N}.$$
\end{step}
\end{proof}
\section{Reduction to a system of differential inequalities}
\label{S:reduction}
The proof of Theorem \ref{T:resolution} is by contradiction. Consider a global solution that does not satisfy the soliton resolution conjecture. Then by the work of C. Rodriguez \cite{Rodriguez16} it is close, for a sequence  of times $\{t_n\}_n$ going to infinity, to a sum of rescaled solitary waves.

Using the study on non-radiative solutions carried out in Section \ref{S:exterior} and the equation \eqref{NLW}, linearized around a sum of solitary waves, we will obtain an approximate differential system satisfied around $t_n$ by the scaling parameters modulating the stationary solutions. We then deduce a contradiction from this system and differential inequality arguments. We divide the proof into two sections. In this section we will set up the contradiction argument and obtain some differential inequalities. In the next section we will restrict the time interval to prove a crucial lower bound (consequence of Proposition \ref{P:ext_scaling} above) on one of the scaling parameters, and prove that the differential system, together with this lower bound leads to a contradiction.

This section and Section \ref{S:end_of_proof} concern the case of global solution. We omit the very close proof for finite time blow-up solutions (see e.g. Section 4 of \cite{DuKeMe13}).

\subsection{Setting of the proof of the soliton resolution}
\label{SS:setting}
Let $u$ be a solution of \eqref{NLW} such that $T_+(u)=+\infty$ and
\begin{equation}
 \label{R10}
 \limsup_{t\to+\infty} \|\vec{u}(t)\|_{\HHH}<\infty.
\end{equation} 
Let $v_L$ be the unique solution of the free wave equation \eqref{FW} such that 
\begin{equation}
 \label{R11}
 \forall A\in \RR,\quad \lim_{t\to +\infty} \int_{|x|\geq A+|t|} |\nabla_{t,x}(u-v_L)(t,x)|^2\,dx=0
\end{equation}
(see \cite[Proposition 4.1]{Rodriguez16}).
For $J\geq 1$, $\iotabf \in\{\pm 1\}^J$, $(f,g)\in \HHH$, we denote
\begin{equation}
\label{R12}
d_{J,\iotabf}(f,g)
=\inf_{\lambdabf\in G_J} \left\{\Big\|(f,g)-\sum_{j=1}^J \iota_j(W_{(\lambda_j)},0)\Big\|_{\HHH}+\gamma(\lambdabf)\right\},
\end{equation} 
where as before 
\begin{equation*}
G_J=\Big\{(\lambda_j)_{j}\in (0,\infty)^J,\;:\; 0<\lambda_J<\ldots<\lambda_2<\lambda_1\Big\},\quad
\gamma(\lambdabf)=\max_{2\leq j\leq J} \frac{\lambda_j}{\lambda_{j-1}} \in (0,1).
\end{equation*}
Assume that $u$ does not scatter forward in time.
By \cite{Rodriguez16}, we know that there exists $J\geq 1$, $\iotabf \in \{\pm 1\}^J$ and a sequence $\{t_n\}_n\to+\infty$ such that
\begin{equation}
 \label{R13}
 \lim_{n\to\infty} d_{J,\iotabf}(\vec{u}(t_n)-\vec{v}_L(t_n))=0.
\end{equation} 
We will prove by contradiction that $\lim_{t\to\infty}d_{J,\iotabf}(\vec{u}(t)-\vec{v}_L(t))=0$. We thus assume that there exists a small $\eps_0>0$ and a sequence $\{\tilde{t}_n\}_n\to+\infty$ such that 
\begin{gather}
 \label{R14}
 \forall n,\quad \tilde{t}_n<t_n\\
 \label{R15}
 \forall n,\quad \forall t\in (\tilde{t}_n,t_n],\quad d_{J,\iotabf}(\vec{u}(t)-\vec{v}_L(t))<\eps_0\\
 \label{R16}
 d_{J,\iotabf}(\vec{u}(\tilde{t}_n)-\vec{v}_L(\tilde{t}_n))=\eps_0.
\end{gather}
We will denote $U=u-v_L$ and use notations that are analogs to the ones of Subsection \ref{SS:close_multi}, although the setting is a bit different.

The implicit function theorem (see Lemma \ref{L:ortho_scaling} in the appendix) implies that for all $t\in [\tilde{t}_n,t_n]$, we can choose $\lambdabf(t)=(\lambda_1(t),\ldots,\lambda_J(t))\in G_J$ such that 
\begin{equation}
 \label{R17}
 \forall j\in \llbracket 1,J\rrbracket, \quad \int \nabla (u(t)-v_{L}(t)-M(t))\cdot\nabla (\Lambda W)_{(\lambda_j(t))}=0,
 \end{equation}
 where $M(t)=\sum_{j=1}^J \iota_j W_{(\lambda_j(t))}$ and, in view of Remark \ref{R:ortho_scaling}
 \begin{equation}
 \label{R18}
 \Big\| \vec{u}(t)-\vec{v}_L(t)-\left(M(t),0\right)\Big\|_{\HHH}+\gamma(\lambdabf) \approx d_{J,\iotabf}(\vec{u}(t)-v_L(t)).
\end{equation} 
In the sequel, we will denote $h(t)=u(t)-v_L(t)-M(t)=U(t)-M(t)$, 
$$ \gamma(t)=\gamma(\lambdabf(t)), \quad \delta(t)=\sqrt{\|h(t)\|^2_{\hdot}+\|\partial_t(u-v_L)(t)\|^2_{L^2}}.$$

We will expand $\partial_tU=\partial_tu-\partial_tv_L$ as follows:
\begin{equation}
 \label{exp_dt}
 \partial_tU(t)=\sum_{j=1}^J \alpha_j(t)\iota_j\Lambda W_{[\lambda_j(t)]}+g_1(t),
\end{equation}
where
\begin{equation}
 \label{exp_dt_ortho}
 \forall j\in \llbracket 1,J\rrbracket,\quad \int g_1(t)\Lambda W_{[\lambda_j(t)]}=0.
\end{equation} 
We also define:
\begin{equation}
 \label{R155}
 \beta_j(t)=-\iota_j\int (\Lambda W)_{[\lambda_j(t)]}\partial_tU(t)\,dx.
\end{equation} 
In this section we prove:
\begin{prop}
\label{P:EDO}
For all large $n$, for all $t\in [\tilde{t}_n,t_n]$,
\begin{align}
\label{bound_delta_gamma}
\delta&\lesssim \gamma^{\frac{N-2}{4}}+o_n(1)\\
\label{EDO2a}
 \forall j\in \llbracket 1,J\rrbracket,\quad
 \left|\beta_j-\|\Lambda W\|^2_{L^2}\lambda'_j\right|&\leq C \gamma^{\frac N4}+o_n(1)\\
 \label{EDO3a}
 \left|\frac{1}{2}\sum_{j=1}^{J}\beta_j^2-\kappa_1\sum_{1\leq j\leq {J}-1} \iota_j\iota_{j+1}\left( \frac{\lambda_{j+1}}{\lambda_j} \right)^{\frac{N-2}{2}}\right|&\leq C\gamma^{\frac{N-1}{2}}+o_n(1)\\
 \label{EDO4a}
 \forall j\in \llbracket 1,J\rrbracket,\quad \Bigg|\lambda_j\beta'_j+\kappa_0\Bigg( \iota_j\iota_{j+1} \bigg( \frac{\lambda_{j+1}}{\lambda_j}\bigg)^{\frac{N-2}{2}}-\iota_j&\iota_{j-1}\bigg( \frac{\lambda_j}{\lambda_{j-1}} \bigg)^{\frac{N-2}{2}} \Bigg) \Bigg|&\\
 \notag
 &\leq C\gamma^{\frac{N-1}{2}}+o_n(1),
 \end{align}
where $o_n(1)$ goes to $0$ as $n\to\infty$, uniformly with respect to $n$ and $t\in [\tilde{t}_n,t_n]$, and 
\begin{align*}
 \kappa_0&=\frac{N^{\frac N2-1}(N-2)^{\frac{N}{2}}}{2}\int \frac{1}{|x|^{N-2}}W^{\frac{N+2}{N-2}}dx,\\
\kappa_1&=\|\Lambda W\|^2_{L^2}\int \frac{\left(N(N-2)\right)^{\frac N2-1}}{|x|^{N-2}}W^{\frac{N+2}{N-2}}\,dx.
\end{align*}
\end{prop}
Let us mention that the constants $\kappa_0$ and $\kappa_1$ can be computed explicitely. However we will not need their exact values in the sequel.

The proof of Proposition \ref{P:EDO} is based on the results of Subsection \ref{SS:close_multi}  on non-radiative solutions. We will first consider, in Subsection \ref{SS:expansion},  a
profile decomposition for a sequence of $\vec{U}(s_n)=\vec{u}(s_n)-\vec{v}_L(s_n)$, 
(where $s_n\to\infty$) 
observing that any nonlinear profile in this decomposition is non-radiative. In Subsection \ref{SS:lambda_beta}, we will use this observation and an expansion of the energy to deduce estimates on $\lambda_j$, $\beta_j$, $\gamma$ and $\delta$.
In Subsection \ref{SS:derivatives}, we will obtain estimates on $\lambda_j'$ and $\beta_j'$ using the equation \eqref{NLW} and the orthogonality conditions, and conclude the proof of \eqref{bound_delta_gamma},\ldots, \eqref{EDO4a}.

We refer to \cite[Proposition 3.8]{JendrejLawrie19} for modulation equations similar to the ones of Proposition \ref{P:EDO}, in the context of equivariant wave maps, when $J=2$, at the threshold energy (so that in this work the analog of the radiation term $v_L$ is $0$). One important novelty here compared to Proposition 3.8 in \cite{JendrejLawrie19} and its proof, is the proof that the nonlinear profiles associated to a sequence $\{\vec{U}(s_n)\}$ are nonradiative solutions of \eqref{NLW} (see Lemma \ref{L:expansion} below), yielding a crucial additional information.
\subsection{Expansion along a sequence of times and renormalization}
\label{SS:expansion}
Consider a sequence of times $\{s_n\}_n$ with $s_n\in [\tilde{t}_n,t_n]$ for all $n$. Extracting subsequences, we define a
partition of the interval $\llbracket 1,J\rrbracket$ as follows. We let $1=j_1<j_2<\ldots<j_{K+1}=J+1$, so that $\llbracket 1,J\rrbracket=\cup_{k=1}^K \llbracket j_k,j_{k+1}-1\rrbracket$, with
\begin{equation}
\label{RR3}
\forall k\in \llbracket 1,K-1\rrbracket,\quad
\lim_{n\to\infty} \frac{\lambda_{j_{k+1}}(s_n)}{\lambda_{j_k}(s_n)}=0.
\end{equation} 
and, 
\begin{equation}
 \label{defnuj}
 \forall k\in \llbracket 1,K\rrbracket, \; \forall j\in \llbracket j_k,j_{k+1}-1\rrbracket,\quad 
 \nu_j=\lim_{n\to\infty} \frac{\lambda_j(s_n)}{\lambda_{j_k}(s_n)}>0.
\end{equation} 
We note that $\nu_{j_k}=1$. 
In this subsection, we prove
\begin{lemma}
\label{L:expansion}
Under the above assumptions, for all $k\in \llbracket 1, K\rrbracket$, there exists $(V_0^k,V_1^k)$ in $\HHH$ such that, denoting by $V^k$ the solution of \eqref{NLW} with initial data $(V_0^k,V_1^k)$, then $V^k$ is defined on $\{|x|>|t|\}$ and is non-radiative. Furthermore, letting $J^k=j_{k+1}-j_k$, $\iotabf^k=(\iota_{j_k},\ldots,\iota_{j_{k+1}-1})$ and
$$ V_n^k(t,x)=\frac{1}{\lambda_{j_k}^{\frac{N-2}2}(s_n)} V^k\left( \frac{t}{\lambda_{j_k}(s_n)},\frac{x}{\lambda_{j_k}(s_n)} \right),$$
we have (extracting subsequences if necessary),
\begin{equation}
 \label{LRR1}
 \lim_{n\to\infty} \left\|\vec{u}(s_n)-\vv_L(s_n)-\sum_{k=1}^{K} \vec{V}_n^k(0)\right\|_{\HHH}=0
\end{equation} 
and
\begin{equation}
 \label{LRR2}
 d_{J^k,\iotabf^k}\left( V_0^k,V_1^k \right)\leq C\eps_0.
\end{equation} 
More precisely, after extraction,
\begin{equation}
 \label{LRR3}
 \left\{
\begin{aligned}
 V^k_0&=\sum_{j=j_k}^{j_{k+1}-1} \iota_jW_{(\nu_j)} +\check{h}_0^k\\
 V^k_1&=\sum_{j=j_k}^{j_{k+1}-1} \iota_j\check{\alpha}_j(\Lambda W)_{[\nu_j]} +\check{g}_1^k,
\end{aligned}
 \right.
\end{equation}
where 
\begin{align}
\label{LRR4}
 \check{h}_0^k&=\underset{n\to\infty}{\wlim}\,\lambda_{j_k}^{\frac{N-2}{2}}(s_n)h\left(s_n,
 \lambda_{j_k}(s_n) \cdot\right)\\
 \label{LRR5}
 \check{\alpha}_j&=\lim_{n\to\infty}\alpha_j(s_n)\\
\label{LRR6}
 \check{g}_1^k&=\underset{n\to\infty}{\wlim}\, \lambda_{j_k}^{N/2}(s_n)g_1\left(s_n,\lambda_{j_k}(s_n) \cdot\right).
\end{align}
Furthermore, we have 
\begin{equation}
 \label{LRR7}
 JE(W,0)=\sum_{k=1}^K E\left( \vec{V}^k(0) \right)
\end{equation} 
\end{lemma}
Note that the limits \eqref{defnuj}, \eqref{LRR4}, \eqref{LRR5} and \eqref{LRR6} imply the orthogonality conditions
\begin{align}
 \label{RR4}
 \forall j\in \llbracket j_k,j_{k+1}-1\rrbracket
 \quad \int \nabla \check{h}_0^{k}\cdot \nabla (\Lambda W)_{(\nu_j)}=
 \int \check{g}_1^{k}\cdot \nabla (\Lambda W)_{[\nu_j]}=
 0
\end{align}
\begin{proof}[Proof of Lemma \ref{L:expansion}]
\setcounter{step}{0}
In all the proof we will denote 
$$\mu_{k,n}=\lambda_{j_k}(s_n).$$
By \eqref{defnuj}, 
$$\lim_{n\to\infty} \frac{\lambda_j(s_n)}{\mu_{k,n}}=\nu_j,\quad j\in \llbracket j_k,j_{k+1}-1\rrbracket.$$
\begin{step}
 Denoting by 
 $$M(t)=\sum_{j=1}^J \iota_jW_{(\lambda_j(t))},$$
 we see that for $k\in \llbracket 1,K\rrbracket$, 
 $$\mu_{k,n}^{\frac{N}{2}-1}M(s_n,\mu_{k,n}\cdot)\xrightharpoonup[n\to\infty]{} \sum_{j=j_k}^{j_{k+1}-1} \iota_jW_{(\nu_j)}.$$
 Extracting subsequences, so that the limits \eqref{LRR4}, \eqref{LRR5}, \eqref{LRR6} exist, we obtain 
 \begin{equation}
  \label{R40}
  \left( \mu_{k,n}^{\frac N2-1} \left( (u-v_L)(s_n,\mu_{k,n}\cdot),\mu_{k,n}^{\frac N2}\partial_t(u-v_L)(s_n,\mu_{k,n}\cdot ) \right)\right)\xrightharpoonup[n\to\infty]{} (V_0^k,V_1^k),
 \end{equation}
where $(V_0^k,V_1^k)$ is defined by \eqref{LRR3}. Note that by \eqref{R15},
\begin{equation}
 \label{R41}
 \left\| (V_0^k,V_1^k)-\sum_{j=j_k}^{j_{k+1}-1} \iota_j(W_{(\nu_j)},0)\right\|_{\HHH}\lesssim \eps_0,
\end{equation} 
and that for $j\in \llbracket j_k,j_{k+1}-2\rrbracket$, 
\begin{equation}
 \label{R42}
 \frac{\nu_{j+1}}{\nu_j} \leq \lim_{n\to\infty}\gamma(\lambda(s_n))\leq \eps_0.
\end{equation}
In particular, \eqref{LRR2} is satisfied. The bounds \eqref{R41} and \eqref{R42} also imply, denoting by $\nubf^k=(\nu_{j_k},\ldots,\nu_{j_{k+1}-1})$
\begin{equation}
 \label{R43}
 \left|\left\|(V_0^k,V_1^k)\right\|^2_{\HHH} -J_k \|\nabla W\|^2_{L^2}\right|
 \lesssim \eps_0 +\gamma(\nubf^k)^{\frac N2-2}=o_{\eps_0}(1),
\end{equation} 
where $o_{\eps_0}(1)$ goes to $0$ as $\eps_0$ goes to $0$. We have used the bound
\begin{equation*}
 \left|\int \nabla W_{(\nu_j)}\cdot\nabla W_{(\nu_{\ell})}\right|\lesssim \gamma(\nubf^k)^{\frac N2-2}.
\end{equation*} 
(see Claim \ref{Cl:estimates1} in the Appendix).
Notice also
\begin{equation}
 \label{R44}
 \left| \|M(s_n)\|^2_{\hdot} -J\|\nabla W\|^2_{L^2}\right|=o_{\eps_0}(1).
\end{equation} 
As a consequence of the weak limit \eqref{R40}, we see that the sequence $\{\vec{u}(s_n)-\vec{v}_L(s_n)\}_n$ has (after extraction of subsequences) a profile
decomposition with profiles $(V^k_F)_{k\geq 1}$ and parameters $\{\mu_{k,n},s_{k,n}\}_n$, $k\geq 1$, where for $k\in \llbracket 1,K\rrbracket$, $\mu_{k,n}=\lambda_{j_k}(s_n)$, $s_{k,n}=0$ and $V^k_F$ is the solution of the free linear wave equation with initial data $(V^k_0,V^k_1)$ (defined by \eqref{LRR3}). Combining \eqref{R40}, \eqref{R41}, \eqref{R43}, \eqref{R44} and the Pythagorean expansion of the profile decomposition, we obtain
\begin{equation}
 \label{R50}
 \sum_{k\geq K+1} \left\|\vec{V}^k_F(0)\right\|_{\HHH}^2=o_{\eps_0}(1).
\end{equation} 
As usual we will denote by $V^k$ the nonlinear profile associated to $V^k_F$, $\{s_{k,n}\}_n$ and $\{\mu_{k,n}\}_n$.
\end{step}
\begin{step}[Approximation for $\{|x|\geq |t|\}$ and lack of radiation]
 We claim that for all $k\in \llbracket 1,K\rrbracket$, the nonlinear profile $V^k$ is defined on $\{|x|\geq |t|\}$ and 
\begin{equation}
\label{Vk_sympa}
 V^k\in \Ssp(\{|x|>|t|\}). 
\end{equation} 
This follows from
\eqref{R41} and long-time perturbation theory. By Minkowski's inequality and scaling arguments,
$$\left\| \sum_{j=j_k}^{j_{k+1}-1} \iota_jW_{(\nu_j)}   \right\|_{\Wsp(\{|x|\geq |t|\})}\leq J\|W\|_{\Wsp(\{|x|\geq |t|\})}.$$
One can check easily that the right-hand side of the preceding inequality is finite. This can be done directly. One can also use that $W$ coincides for $|x|>|t|$ with the solution $\tW$ of 
$$(\partial_t^2-\Delta)\tW=W^{\frac{N+2}{N-2}}\indxt,$$
with 
$W^{\frac{N+2}{N-2}}\indic_{\{|x|>|t|\}}\in L^1L^2,$
so that, by the Strichartz inequality \eqref{Strichartz}, $\widetilde{W}\in \Wsp(\{|x|>|t|\})$.

By \eqref{R42} and the Claim \ref{Cl:pointwise2} (with $h=0$) in the appendix, denoting 
$$r_k=(\partial_t^2-\Delta)\left(\sum_{j=j_k}^{j_{k+1}-1} \iota_jW_{(\nu_j)}\right)-F\left(  \sum_{j=j_k}^{j_{k+1}-1} \iota_jW_{(\nu_j)}\right),$$
where $F(z)=|z|^{\frac{4}{N-2}}z$, we have
$$\|\indic_{\{|x|\geq |t|\}} r_k\|_{L^1L^2}=o_{\eps_0}(1).$$
Taking $\eps_0>0$ small enough, we deduce by long time perturbation theory (Proposition \ref{P:LTPT}) that $V^k$ is defined on $\{|x|\geq |t|\}$ and satisfies \eqref{Vk_sympa}.

In view of \eqref{R50} all the nonlinear profiles $V^k$, $k\geq K+1$ are globally defined and scatter. From the profile approximation property (Proposition \ref{P:NL_profile}), we obtain that for large $n$, the solution $(\tau,x)\mapsto u(s_n+\tau,x)$ is defined on $\{|x|\geq |\tau|\}$, and that for all $\ell\gg 1$
\begin{equation}
 \label{R60}
u(s_n+\tau)=v_L(s_n+\tau)+\sum_{k=1}^{\ell} V_n^k(\tau)+w_n^{\ell}(\tau)+r_n^{\ell}(\tau),
\end{equation} 
where $w_n^{\ell}$ is the solution of the free wave equation \eqref{FW} with initial data $\vec{u}(s_n)-\vec{v}_L(s_n)-\sum_{k=1}^{\ell} \vec{V}_n^k(0)$, and $r_n^{\ell}$ satisfies
\begin{equation}
 \label{R61}
\lim_{\ell\to\infty} \limsup_{n\to\infty}\, \sup_{\tau\in \RR} \left\| \indic_{\{|x|\geq |\tau|\}} \nabla_{\tau,x}r_n^{\ell}(\tau)\right\|_{L^2}=0.
 \end{equation} 
We next prove that for all $k\geq 1$,
\begin{equation}
 \label{R62}
\lim_{n\to\infty} \sum_{\pm}\lim_{\tau\to\pm\infty}  \int_{|x|\geq |\tau|} |\nabla_{\tau,x}V_n^k(\tau,x)|^2\,dx=0.
\end{equation} 
and that for all $\ell\geq 1$,
\begin{equation}
 \label{R62'}
\lim_{n\to\infty} \sum_{\pm}\lim_{\tau\to\pm\infty}  \int_{|x|\geq |\tau|} |\nabla_{\tau,x}w_n^{\ell}(\tau,x)|^2\,dx=0.
\end{equation}
The proof is similar to that of the analoguous properties \eqref{G124}, \eqref{G125} in the proof of the exterior energy bound for the linearized equation close to a soliton. Using \eqref{R60} and fixing $1\leq k<\ell$, we see that for all $\tau$,
\begin{multline}
 \label{R63}
\int_{|x|\geq |\tau|} \nabla_{\tau,x}\left(u(s_n+\tau,x)-v_L(s_n+\tau,x)\right)\cdot \nabla_{\tau,y} V^k_n(\tau,x)\,dx\\
= \int_{|x|\geq |\tau|} |\nabla_{\tau,x}V_n^k(\tau,x)|^2\,dx+\sum_{\substack{0\leq j\leq \ell\\ j\neq k}} \int_{|x|\geq |\tau|} \nabla_{\tau,x} V_n^j(\tau,x)\cdot \nabla_{\tau,x}V_n^k(\tau,x)\,dx\\
+\int_{|x|\geq |\tau|} \nabla_{\tau,x} w_n^{\ell}(\tau,x)\cdot \nabla_{\tau,x} V_n^k(\tau,x)\,dx + \int_{|x|\geq |\tau|} \nabla_{\tau,x}^{\ell} r_n^{\ell}(\tau,x)\nabla_{\tau,x}V_n^k(\tau,x)\,dx.
\end{multline}
Since for all $j$, $V_n^j$ scatters in both time directions in $\{|x|\geq |\tau|\}$ (in the sense that it satisfies \eqref{exterior_scattering}), we have, using the pseudo orthogonality of the parameters as in the proof mentioned above :
\begin{equation}
 \label{R70}
\forall j\neq k,\quad \lim_{n\to\infty} \sum_{\pm}\lim_{\tau \to \pm\infty}\left|\int_{|x|\geq |\tau|} \nabla_{\tau,x} V_n^j(\tau,x)\cdot \nabla_{\tau,x}V_n^k(\tau,x)\,dx\right|=0
\end{equation} 
and also
\begin{equation*}
k\leq \ell\Longrightarrow
\lim_{n\to\infty} \sum_{\pm}\lim_{\tau \to \pm\infty}
\left|\int_{|x|\geq |\tau|} \nabla_{\tau,x} w_n^{\ell}(\tau,x)\cdot \nabla_{\tau,x} V_n^k(\tau,x)\,dx\right|=0.
\end{equation*}
Furthermore by \eqref{R61},
\begin{equation*}
\lim_{\ell\to\infty}\limsup_{n\to\infty} \sum_{\pm}\lim_{\tau \to \pm\infty}
\left|\int_{|x|\geq |\tau|} \nabla_{\tau,x} r_n^{\ell}(\tau,x)\cdot \nabla_{\tau,x} V_n^k(\tau,x)\,dx\right|=0.
\end{equation*}
By the definition of $v_L$, we have that for all $n$,
\begin{multline}
 \label{R73}
\lim_{\tau \to +\infty}
\int_{|x|\geq \tau} \left|\nabla_{t,x}(u(s_n+\tau,x)-v_L(s_n+\tau,x))\right|^2\,dx\\
=
\lim_{\sigma\to +\infty}
\int_{|x|\geq \sigma-s_n} \left|\nabla_{t,x}(u(\sigma,x)-v_L(\sigma,x))\right|^2\,dx=0.
\end{multline} 
On the other hand,
\begin{multline*}
 \lim_{\tau \to -\infty}
\int_{|x|\geq |\tau|} \left|\nabla_{t,x}(u(s_n+\tau,x)-v_L(s_n+\tau,x))\right|^2\,dx\\
=
\lim_{\sigma\to -\infty}
\int_{|x|\geq s_n-\sigma} \left|\nabla_{t,x}(u(\sigma,x)-v_L(\sigma,x))\right|^2\,dx.
\end{multline*} 
By the small data theory, there exists a solution $u_F$ of the free linear equation such that if $A\gg 1$,
$$ \lim_{t\to-\infty} \int_{|x|>A-t} |\nabla_{t,x}(u-u_F)(t,x)|^2\,dx=0.$$
Combining with the large time asymptotics for linear wave equation, we deduce that there exists $g\in L^2([A,+\infty)$ (for a fixed $A\gg 1$), such that 
\begin{equation}
 \label{R80}
\lim_{\tau\to -\infty} \int_{|x|\geq |\tau|}\left|\nabla_{t,x}(u(s_n+\tau,x)-v_L(s_n+\tau,x))\right|^2\,dx=\int_{\eta\geq s_n} |g(\eta)|^2\,d\eta.
\end{equation} 
Note that the right-hand side of \eqref{R80} goes to $0$ as $n$ goes to infinity.
Combining \eqref{R63},\ldots,\eqref{R80} we obtain the desired estimate \eqref{R62}. A similar proof yields \eqref{R62'}.
\end{step}
\begin{step}[Consequence of the equirepartition of the energy]
 Let $k\geq K+1$. Then by \eqref{R50},
$$\sup_n \|V_n^k(0)\|^2_{\HHH}=o_{\eps_0}(1).$$
Using the small data theory and the equirepartition of the energy outside the wave cone for the free wave equation proved in \cite{DuKeMe12} (see Theorem \ref{T:equirepartition} above), we deduce, if $\eps_0$ is small enough,
$$ \sum_{\pm} \lim_{t\to \pm\infty} \int_{|x|\geq |t|} |\nabla_{t,x}V_n^k(t,x)|^2\,dx\geq \frac{1}{2}\int |\nabla_{t,x}V^k_n(0,x)|^2\,dx\geq \frac{1}{4}\|V^k(0)\|^2_{\HHH}.$$
From \eqref{R62}, we deduce that $V^k\equiv 0$ for $k\geq K+1$. As a consequence $w_n^{\ell}$ does not depend on $\ell$ if $\ell\geq K$. Denoting by $w_n=w_n^{\ell}$, we have by \eqref{R62'}
$$\lim_{n\to\infty} \sum_{\pm}\lim_{\tau\to\pm \infty} \int_{|x|\geq |\tau|} |\nabla_{\tau,x}w_n(\tau,x)|^2\,dx=0.$$
Since $w_n$ is a solution of the free wave equation, we deduce (using Theorem \ref{T:equirepartition} again), that 
$$\lim_{n\to\infty} \int |\nabla_{t,x}w_n(0,x)|^2\,dx=0,$$
and hence \eqref{LRR1}.
It remains to observe that if $1\leq k\leq K$, the property \eqref{R62} implies, since the time parameter $s_{k,n}$ is identically $0$, that 
$$\sum_{\pm}\lim_{t\to \pm\infty}\int_{|x|\geq |t|} |\nabla_{t,x}V^k(t,x)|^2\,dx=0,$$
i.e. that $V^k$ is non-radiative at $t=0$.
\end{step}
\end{proof}
For further use, We state the following important consequence of the proof of Lemma \ref{L:expansion} (see \eqref{R60})
\begin{claim}[Exterior expansion for all time]
\label{Cl:expansion}
We have, for $|x|\geq |\tau|$
\begin{equation}
\label{R60bis}
u(s_n+\tau)=v_L(s_n+\tau)+\sum_{k=1}^K V_n^k(\tau)+r_n(\tau),
\end{equation} 
where
\begin{equation*}
 \lim_{n\to\infty} \sup_{\tau}\int_{|x|\geq |\tau|} |\nabla_{\tau,x}r_n|^2\,dx=0.
\end{equation*} 
\end{claim}
\subsection{Estimates on $\lambda_j$ and $\beta_j$}
\label{SS:lambda_beta}
Recall from the introduction of this section the definitions of $h(t)$, $g_1(t)$, $\alpha_j(t)$, $\beta_j(t)$, $\gamma(t)$, $\delta(t)$.
\begin{lemma}
\label{L:R26}
There exists a constant $C>0$, depending only on $J$ and $N$, such that under the preceding assumptions, 
\begin{gather}
\label{R2601} 
\forall t\in [\tilde{t}_n,t_n], \quad \left\|(h,g_1)\right\|_{\HHH}\leq o_n(1)+C\big(\gamma^{\frac{N}{4}}+\delta^{\frac{N}{N-2}}\big)\\
 \label{R260}
\forall t\in [\tilde{t}_n,t_n], \quad \left|\delta^2-\sum_{j=1}^J \alpha_j^2\|\Lambda W\|^2_{L^2}\right| \leq 
o_n(1)+C\big(\gamma^{\frac{N-1}{2}}+\delta^{\frac{2(N-1)}{N-2}}\big)
\\
  \label{F182} \forall t\in [\tilde{t}_n,t_n], \quad \left|\beta_j+\alpha_j\|\Lambda W\|_{L^2}^2\right|\leq o_n(1)+C\big(\gamma^{\frac{N}{4}}+\delta^{\frac{N}{N-2}}\big)
\end{gather} 
where in all inequalities, $o_n(1)$ goes to $0$ as $n$ goes to infinity \emph{uniformly with respect to $t\in [\tilde{t}_n,t_n]$}.
\end{lemma}
\begin{proof}
Note that \eqref{R2601} and \eqref{R260} are time-dependent version of the estimates \eqref{F180} and \eqref{F181} for non-radiative solution. We will prove \eqref{R260} as a consequence of \eqref{F181}. The proof of \eqref{R2601} using \eqref{F180} is very similar and we omit it. We argue by contradiction. If \eqref{R260} does not hold, there exists, after extraction, a sequence of times $\{s_n\}_n$ with $s_n\in [\tilde{t}_n,t_n]$, and $\eps_1>0$ such that for all $n$,
\begin{equation}
 \label{R262}
\left|\delta^2(s_n)-\sum_{j=1}^J \alpha_j^2(s_n)\|\Lambda W\|^2_{L^2}\right|\geq C\left(\delta^{\frac{2(N-1)}{N-2}}(s_n)+\gamma^{\frac{N-1}{2}}(s_n)\right)+\eps_1.
\end{equation} 
Using Lemma \ref{L:expansion}, we have
\begin{equation}
 \label{R270}
\lim_{n\to\infty}\left\|\vec{u}(s_n)-\vec{v}_L(s_n)-\sum_{k=1}^K \vec{V}_n^k(0)\right\|_{\HHH}=0,
\end{equation} 
where the rescaled profiles $V_n^k$ are defined as in Lemma \ref{L:expansion}:
$$ V_n^k(t,x)=\frac{1}{\lambda_{j_k}^{\frac N2-1}(s_n)} V^k\left( \frac{t}{\lambda_{j_k(s_n)}},\frac{x}{\lambda_{j_k(s_n)}} \right),$$
and $V^k$ is a \emph{non-radiative} solution to the nonlinear wave equation with initial data $(V_0^k,V_1^k)$, and
\begin{equation*}
 \left\{
\begin{aligned}
 V^k_0&=\sum_{j=j_k}^{j_{k+1}-1} \iota_jW_{(\nu_j)} +\check{h}_0^k\\
 V^k_1&=\sum_{j=j_k}^{j_{k+1}-1} \iota_j \check{\alpha}_j\Lambda W_{[\nu_j]} +\check{g}_1^k,
\end{aligned}
 \right.
\end{equation*}
where $\check{h}_0^k$, $\check{g}_1^k$ are defined as weak limits of $h(s_n)$ and $g_1(s_n)$ after an appropriate rescaling (see \eqref{LRR4}, \eqref{LRR6}) and
$$\check{\alpha}_j=\lim_{n\to\infty}\alpha_j(s_n).$$
Since for all $k\in \llbracket 1,K\rrbracket$, $V^k$ is non-radiative, we can use the estimate \eqref{F181}, which writes
\begin{equation}
\label{R280}
\left|\delta_k^2-\sum_{j=j_k}^{j_{k+1}-1}\check{\alpha}_j^2\|\Lambda W\|^2_{L^2}\right|\leq 
C\left(\delta_k^{\frac{2(N-1)}{N-2}}+\gamma_k^{\frac{N-1}{2}}\right),
\end{equation} 
where 
$$ \delta_k^2=\|\partial_tV^k(0)\|^2_{L^2}+\|h^k_0\|_{\hdot}^2,$$
and 
$$\gamma_k=\max_{j_k\leq j\leq j_{k+1}-2} \frac{\nu_j}{\nu_{j+1}}$$
(as usual, if $j_{k+1}=1+j_k$, we let $\gamma_k=0$).

Observe that 
$$\lim_{n\to\infty} \gamma(s_n)=\max_{1\leq k\leq K} \gamma_k,$$
and, by the expansion \eqref{R270},
$$\lim_{n\to\infty}\delta^2(s_n)=\sum_{1\leq k\leq K} \delta_k^2.$$
Summing up \eqref{R280}, we deduce
\begin{equation*}
\left|\delta^2(s_n)-\sum_{j=1}^J \alpha^2_j(s_n)\|\Lambda W\|^2_{L^2}\right|\leq C_J
\left(\delta^{\frac{2(N-1)}{N-2}}(s_n)+\gamma^{\frac{N-1}{2}}(s_n)+o_n(1) \right),
\end{equation*}
where the constant $C_J$ depends only on $J$. This contradicts \eqref{R262} for large $n$. The proof is complete.

We next compare $\alpha_j$ and $\beta_j$ and prove \eqref{F182}. We have, expanding $\partial_t U$ by \eqref{exp_dt},
 \begin{multline*}
  \beta_j(t)=-\iota_j\int (\Lambda W)_{[\lambda_j]}\partial_tU\\
  =-\iota_j \underbrace{\int (\Lambda W)_{[\lambda_j]} g_1}_{0}-\alpha_j \|\Lambda W\|^2_{L^2}-\iota_j \iota_k\sum_{k\neq j}\alpha_k \int (\Lambda W)_{[\lambda_j]}(\Lambda W)_{[\lambda_k]}.
 \end{multline*}
By \eqref{R260} and Claim \ref{Cl:estimates1} in the appendix,
\begin{multline*}
\left|\alpha_k \int (\Lambda W)_{[\lambda_j]}(\Lambda W)_{[\lambda_k]}\right|\lesssim \left( \delta+\gamma^{\frac{N-1}{2}} +o_n(1)\right)\gamma^{\frac N2-2}\\ \lesssim \delta^{\frac{N}{N-2}}+\gamma^{\frac{N(N-4)}{4}}+\gamma^{N-\frac{5}{2}}+o_n(1),
\end{multline*}
which yields \eqref{F182}.
\end{proof}
We next prove:
\begin{lemma}[Expansion of the energy]
\label{St:energy}
\begin{gather}
 \label{energy_bound}
 \left|\frac{1}{2}\delta^2-\kappa_1'\sum_{1\leq j\leq J-1} \iota_j\iota_{j+1}\left( \frac{\lambda_{j+1}}{\lambda_j} \right)^{\frac{N-2}{2}}\right|\lesssim 
o_n(1)+\gamma^{\frac{N-1}{2}}\\ 
\label{bound_delta}
\delta\lesssim \gamma^{\frac{N-2}{4}}+o_n(1)\\
\label{bound_g1_h}
  \left\|(h(t),g_1(t))\right\|_{\HHH}\lesssim o_n(1)+ \gamma^{\frac N4},
 \end{gather} 
 where $o_n(1)\to 0$ as $n\to\infty$, uniformly with respect to $t\in [\tilde{t}_n,t_n]$, and
 $$\kappa_1'=\int \frac{\left(N(N-2)\right)^{\frac N2-1}}{|x|^{N-2}}W^{\frac{N+2}{N-2}}\,dx=\frac{1}{\|\Lambda W\|^2_{L^2}}\kappa_1.$$
 \end{lemma}
 \begin{proof}
 Note that \eqref{bound_g1_h} follows from \eqref{bound_delta} and \eqref{R2601}.
 
 We are thus left with proving \eqref{bound_delta} and \eqref{energy_bound}. Recall that 
 $$\lim_{t\to\infty} E(\vec{u}(t)-\vec{v}_L(t))=JE(W,0).$$
 Expanding the energy
 $$E(\vec{u}-\vec{v}_L)=E\left(\sum_{j=1}^J \iota_j W_{(\lambda_j)}+h,\sum_{j=1}^J\alpha_j(\Lambda W)_{[\lambda_j]}+g_1\right),$$
 we obtain, in view of the inequality,
 \begin{multline*}
 \left|  
 \frac{N-2}{2N}\left|\sum_{j=1}^J y_j+h\right|^{\frac{2N}{N-2}}-\frac{N-2}{2N}\sum_{j=1}^J |y_j|^{\frac{2N}{N-2}} -\sum_{j=1}^J |y_j|^{\frac{4}{N-2}} y_jh
-\sum_{\substack{1\leq j,k\leq J\\ j\neq k}} |y_j|^{\frac{4}{N-2}} y_jy_k
 \right| \\
 \lesssim |h|^{\frac{2N}{N-2}}+\sum_{j=1}^J |y_j|^{\frac{4}{N-2}} |h|^2\\
 +\sum_{1\leq j<k\leq J}\left( \min\left\{ |y_j|^{\frac{4}{N-2}} y_k^2,|y_k|^{\frac{4}{N-2}} y_j^2\right\}+\min\left\{ |y_j|^{\frac{N+2}{N-2}} |y_k|,|y_k|^{\frac{N+2}{N-2}} |y_j|\right\} \right),
\end{multline*}
proved in appendix \ref{A:pointwise},
 and the estimate $\int \left|\Lambda W_{[\lambda_j]}\Lambda W_{[\lambda_k]}\right|\lesssim \gamma^{\frac{N}{2}-2}$ (see Claim \ref{Cl:estimates1} in the appendix),
 \begin{multline*}
  \bigg|\frac{J}{2}\|\nabla W\|^2_{L^2}+\sum_{1\leq j<k\leq J} \iota_j\iota_k\int \nabla W_{(\lambda_j)}\cdot\nabla W_{(\lambda_k)}
+\sum_{1\leq j\leq J} \iota_j\int \nabla W_{(\lambda_j)}\cdot\nabla h\\
-\frac{N-2}{2N}J\int W^{\frac{2N}{N-2}}
-\sum_{j=1}^J \iota_j\int W_{(\lambda_j)}^{\frac{N+2}{N-2}} h\\
-\sum_{\substack{1\leq j,k\leq J\\k\neq j}} \iota_j\iota_k\int W_{(\lambda_j)}^{\frac{N+2}{N-2}}W_{(\lambda_k)}
  +\frac{1}{2}\sum_{j=1}^J\alpha_j^2\|\Lambda W\|^2_{L^2}-JE(W,0)\bigg|
 \\ \lesssim \sum_{1\leq j<k\leq J} \int \min\left(W_{(\lambda_j)}^{\frac{4}{N-2}}W_{(\lambda_k)}^2,W_{(\lambda_k)}^{\frac{4}{N-2}}W_{(\lambda_j)}^2\right)
 +\min\left(W_{(\lambda_j)}^{\frac{N+2}{N-2}}W_{(\lambda_k)}, W_{(\lambda_k)}^{\frac{N+2}{N-2}}W_{(\lambda_j)}\right)\\
 +\|g_1\|^2_{L^2}
 +\|\nabla h\|^2_{L^2}+\gamma^{\frac{N}{2}-2}\delta^2+o_n(1).
 \end{multline*}
We note that for all $j,k$,
$$ \int \nabla W_{(\lambda_j)}\cdot\nabla W_{(\lambda_k)}=\int W_{(\lambda_j)}^{\frac{N+2}{N-2}}W_{(\lambda_k)},\quad \int \nabla W_{(\lambda_j)}\cdot\nabla h=\int W_{(\lambda_j)}^{\frac{N+2}{N-2}}h$$
Furthermore we have, 
\begin{multline*}
 \int \min\left\{W_{(\lambda_j)}^{\frac{N+2}{N-2}}W_{(\lambda_k)},W_{(\lambda_k)}^{\frac{N+2}{N-2}}W_{(\lambda_j)}\right\} +\int \min\left\{W_{(\lambda_j)}^{\frac{4}{N-2}}W_{(\lambda_k)}^2,W_{(\lambda_k)}^{\frac{4}{N-2}}W_{(\lambda_j)}^2\right\}\\
 \lesssim 
 \int  W_{(\lambda_j)}^{\frac{N}{N-2}}W_{(\lambda_k)}^{\frac{N}{N-2}}\lesssim \gamma^{\frac N2},
\end{multline*}
by Claim \ref{Cl:estimates1}.

As a consequence, using also the estimate \eqref{R2601} on $h$ and $g_1$, we have, 
\begin{multline}
\label{prsq_fn}
\left|\frac{1}{2}\sum_{j=1}^J\alpha_j^2\|\Lambda W\|^2_{L^2}-\sum_{1\leq j<k\leq J} \iota_j\iota_k\int W_{(\lambda_j)}^{\frac{N+2}{N-2}}W_{(\lambda_k)}\right|\\ \lesssim  \gamma^{\frac{N}{2}}+\gamma^{\frac{N}{2}-2}\delta^2+\delta^{\frac{2N}{N-2}}+o_n(1).
\end{multline}
We next estimate, for $j<k$,
\begin{multline*}
\int W_{(\lambda_j)}^{\frac{N+2}{N-2}}W_{(\lambda_k)}
=\left( \frac{\lambda_j}{\lambda_k} \right)^{ \frac{N-2}{2}} \int W^{\frac {N+2}{N-2}}(x) W\left( \frac{\lambda_j x}{\lambda_k} \right)\,dx\\
=\left(\frac{\lambda_k}{\lambda_j}\right)^{\frac{N-2}{2}}\int \frac{\left(N(N-2)\right)^{\frac{N-2}{2}}}{|x|^{N-2}}W^{\frac{N+2}{N-2}}\,dx\\
+\OOO\left( \left(\frac{\lambda_k}{\lambda_j}\right)^{\frac{N+1}{2}}\int W^{\frac{N+2}{N-2}}\frac{1}{|x|^{N-\frac 12}}\,dx \right),
\end{multline*}
where we have used 
$$\left|W(x)-\frac{\left((N-2)N\right)^{\frac{N-2}{2}}}{|x|^{N-2}}\right|\lesssim \frac{1}{|x|^{N-\frac 12}}.$$
In particular, if $j<k-1$, we see that
$$\int W_{(\lambda_j)}^{\frac{N+2}{N-2}} W_{(\lambda_k)}\lesssim \gamma^{N-2}.$$
Combining with \eqref{R260} and \eqref{prsq_fn}, we obtain
$$ \delta^2\lesssim \gamma^{\frac{N-2}{2}}+\gamma^{\frac{N}{2}-2}\delta^2+\delta^{\frac{2N}{N-2}}+o_n(1),$$
which yields $\delta\lesssim \gamma^{\frac{N-2}{4}}+o_n(1)$, i.e.  \eqref{bound_delta}. Going back to \eqref{prsq_fn}, we deduce \eqref{energy_bound}.
\end{proof}

\subsection{System of equations and estimates on the derivatives}
\label{SS:derivatives}
Under the above assumptions, we let as before $U(t)=u(t)-v_L(t)$, so that
$$ h(t)=U(t)-\sum_{j=1}^J\iota_jW_{(\lambda_j)}=U(t)-M(t).$$
Expanding the nonlinear wave equation \eqref{NLW}, we see that $(h(t),\partial_tU(t))$ satisfy the following system of equations for $t\in [\tilde{t}_n,t_n]$,
 \begin{equation}
  \label{R130}
  \left\{
  \begin{aligned}
   \frac{\partial h}{\partial t}&=\frac{\partial U}{\partial t}+\sum_{j=1}^J \iota_j\lambda_j'(t)\left( \Lambda W\right)_{[\lambda_j(t)]}\\
  \frac{\partial}{\partial t}\left( \frac{\partial U}{\partial t} \right)-\Delta h&=F(U)-\sum_{j=1}^JF\left(\iota_j W_{(\lambda_j)}\right)+\sigma(h,v_L),
  \end{aligned}\right.
 \end{equation}
where 
\begin{equation}
 \label{R150}
 \sigma(h,v_L):=F\left( M+h+v_L \right)-F\left( M+h \right),
\end{equation} 
satisfies
\begin{equation}
 \label{R133}
 |\sigma(h,v_L)|\lesssim |v_L(t)|^{\frac{N+2}{N-2}}+\sum_{j=1}^J\left(W_{(\lambda_j)}^{\frac{4}{N-2}}+|h(t)|^{\frac{4}{N-2}}  \right)|v_L(t)|.
\end{equation} 
We next estimate $\lambda_j'(t)$, using the orthogonality condition \eqref{R17} and the first equation in \eqref{R130}. More precisely, we will prove:
\begin{lemma}[Derivative of the scaling parameters]
\label{L:derivative}
 \begin{equation} 
  \label{F200}
  \left|\lambda'_j+\alpha_j\right|\lesssim \gamma^{\frac N4}+o_n(1),
 \end{equation} 
 where $o_n(1)$ goes to $0$ as $n\to\infty$, uniformly with respect to $t\in [\tilde{t}_n,t_n]$. \end{lemma}
 \begin{proof}
 According to \eqref{R17},
 $$ \forall t\in I,\quad \int h(t)\frac{1}{\lambda_j^{\frac N2}}\left(\Delta \Lambda W\right)\left( \frac{x}{\lambda_j(t)} \right)\,dx=0.$$
 Differentiating with respect to $t$ we obtain
 \begin{equation*}
  0=   \int \frac{\partial h}{\partial t} \frac{1}{\lambda_j^{\frac N2}} (\Delta \Lambda W)\left( \frac{x}{\lambda_j}\right)\,dx-\frac N2 \lambda_j'\int h \frac{1}{\lambda_j^{1+\frac{N}{2}}}\left( \Lambda_0 \Delta \Lambda W\right)\left( \frac{x}{\lambda_j} \right)\,dx,
  \end{equation*}
  where $\Lambda_0=\frac{N}{2}+x\cdot \nabla$. By the first equation in \eqref{R130}, 
  \begin{multline*}
  0= \int \frac{\partial U}{\partial t}\frac{1}{\lambda_j^{\frac N2}}\left( \Delta \Lambda W \right)\left( \frac{x}{\lambda_j} \right)\,dx\\+\sum_{k=1}^J \iota_k \lambda_k'\int \frac{1}{\lambda_k^{\frac N2}}(\Lambda W)\left( \frac{x}{\lambda_k} \right)\frac{1}{\lambda_j^{\frac N2}}(\Delta \Lambda W)\left( \frac{x}{\lambda_j} \right)\,dx
  \\
  -\frac N2 \lambda_j'\int h \frac{1}{\lambda_j^{1+\frac{N}{2}}} \left( \Lambda_0\Delta \Lambda W\right)\left( \frac{x}{\lambda_j} \right)\,dx.
 \end{multline*}
In view of the definition \eqref{exp_dt} of $g_1$, we have $\partial_tU=\sum_k \iota_k\alpha_k(\Lambda W)_{[\lambda_k]}+g_1$. By the estimate \eqref{bound_g1_h} on $g_1$, the bound \eqref{bound_delta} on $\delta$ and the estimate
$$   \left|\int (\Lambda W)_{[\lambda_j]}\left(\Delta \Lambda W\right)_{[\lambda_k]}\right|\lesssim \gamma^{\frac{N}{2}-2},\quad j\neq k$$
(see \eqref{est1.5} in the appendix),
we obtain
\begin{multline*}
 \int \frac{\partial U}{\partial t} \frac{1}{\lambda_j^{\frac N2}} \left(\Delta \Lambda W\right)\left( \frac{x}{\lambda_j} \right)
 \\
 = \int g_1\frac{1}{\lambda_j^{\frac N2}} \left(\Delta \Lambda W\right)\left( \frac{x}{\lambda_j} \right)-\alpha_j \iota_j\|\Lambda W\|^2_{\hdot}+\sum_{k\neq j}\iota_k\alpha_k \int (\Lambda W)_{[\lambda_j]}\left(\Delta \Lambda W\right)_{[\lambda_k]}\\
 =-\alpha_j\iota_j\|\Lambda W\|^2_{\hdot} +O\left( \gamma^{\frac{N}{4}} \right).
\end{multline*}
By the estimate \eqref{bound_g1_h} on $h$,
$$\left|\lambda_j'\int h \frac{1}{\lambda_j^{1+\frac{N}{2}}} \left( \Lambda_0\Delta \Lambda W\right)\left( \frac{x}{\lambda_j} \right)\,dx\right|\lesssim |\lambda'_j|\|\nabla h\|_{L^2}
\lesssim \left(\gamma^{\frac{N}{4}}+o_n(1)\right)\,|\lambda_j'|.$$
Combining, we obtain
$$ \forall j,\quad \Big|\alpha_j\|\Lambda W\|^2_{\hdot}+\lambda'_j\|\Lambda W\|^2_{\hdot}\Big|\lesssim \gamma^{\frac N4}\left(|\lambda'_j|+1\right)+\gamma^{\frac 12} \sum_{k\neq j} \left|\lambda_k'\right|+o_n(1),$$
and thus, letting $\alphabf=(\alpha_1,\ldots,\alpha_J)$,
$$\left| \lambdabf'+\alphabf\right|\lesssim |\lambdabf'|\gamma^{\frac 12}+\gamma^{\frac N4}+o_n(1).$$
This implies, recalling that by \eqref{bound_delta}, $\delta\lesssim \gamma^{\frac{N-2}{4}}$,
$$ |\lambdabf'|\lesssim |\alphabf|+\gamma^{\frac N4}\lesssim \delta+\gamma^{\frac N4}+o_n(1)\lesssim \gamma^{\frac{N-2}{4}}+o_n(1).$$
The desired estimate \eqref{F200} follows immediately from the two bounds above.
\end{proof}

\begin{lemma}[Second derivative of the scaling parameter] 
\label{L:second_derivative}
 For all $j\in \llbracket 1,J\rrbracket$,
 \begin{equation}
  \label{F270}
  \left|\lambda_j\beta'_j +\kappa_0\left(\iota_j\iota_{j+1}\left( \frac{\lambda_{j+1}}{\lambda_j} \right)^{\frac{N}{2}-1}-\iota_j\iota_{j-1}\left( \frac{\lambda_j}{\lambda_{j-1}}\right)^{\frac N2-1}\right)\right|\lesssim \gamma^{\frac{N-1}{2}}+o_n(1),
 \end{equation} 
where $\kappa_0$ is defined in Proposition \ref{P:EDO},
and by definition $\iota_0=\iota_{J+1}=0$.
\end{lemma}
Note that $\beta_j$ is, according to \eqref{F182}, \eqref{bound_delta}, \eqref{F200}, proportionate to $\lambda'_j$ up to lower order terms, so that \eqref{F270} can be interpreted as an estimate on the second derivative of $\lambda_j$. 
\begin{proof}
Differentiating the definition \eqref{R155} of $\beta_j$, we obtain 
\begin{equation*}
\lambda_j\beta_j'(t)=\iota_j \lambda_j' \int \left(\Lambda_0\Lambda W\right)_{[\lambda_j]}\partial_tU-\iota_j \lambda_j \int \left( \Lambda W \right)_{[\lambda_j]}\partial_t^2U
\end{equation*} 
We first prove that the first term of this sum is small. Using the expansion \eqref{exp_dt} of $\partial_tU$, we obtain
\begin{multline*}
\int \left(\Lambda_0\Lambda W\right)_{[\lambda_j]}\partial_tU
=\int (\Lambda_0\Lambda W)_{[\lambda_j]} g_1\\+\iota_j\alpha_j \underbrace{\int (\Lambda_0\Lambda W)_{[\lambda_j]}\left( \Lambda W \right)_{[\lambda_j]}}_{0}+\sum_{k\neq j} \int \iota_k\alpha_k(\Lambda_0\Lambda W)_{[\lambda_j]} (\Lambda W)_{[\lambda_k]}.
\end{multline*}
Hence, by \eqref{bound_delta}, \eqref{bound_g1_h}, \eqref{R260}, \eqref{F200} and the estimate
$\left|\int (\Lambda_0\Lambda W)_{[\lambda_j]} (\Lambda W)_{[\lambda_k]}\right|\lesssim \gamma^{\frac N2-2}$ (see \eqref{est1.2} in the appendix), we obtain
\begin{equation}
 \label{F211} \left|\lambda_j'\int (\Lambda_0\Lambda W)_{[\lambda_j]}\partial_tU\right|\lesssim \gamma^{\frac{N-1}{2}}+o_n(1).
\end{equation} 
By the second equation in \eqref{R130}, we have
\begin{gather}
\label{main_beta'1}
 \lambda_j\int \left( \Lambda W\right)_{[\lambda_j]}\partial_t^2U
=-\int\left( \Lambda W \right)_{(\lambda_j)}L_{W_{(\lambda_j)}}h\\
\label{main_beta'2}
+\int (\Lambda W)_{(\lambda_j)} \sigma(h,v_L)\\
 \label{main_beta'3}
 +\int (\Lambda W)_{(\lambda_j)}\left( F(\iota_j W_{(\lambda_j)}+h)-F(\iota_jW_{(\lambda_j)})-\frac{N+2}{N-2}W_{(\lambda_j)}^{\frac{4}{N-2}}h \right)\\
 \label{main_beta'4}
 +\int (\Lambda W)_{\lambda_j}\Big( F(M+h)-F(M)+F(\iota_jW_{(\lambda_j)})-F(\iota_j W_{(\lambda_j)}+h)\Big)\\
 \label{main_beta'5}
 +\int (\Lambda W)_{(\lambda_j)} \left( F(M)-\sum_{k=1}^J F(\iota_k W_{(\lambda_k)}) \right),
 \end{gather}
 where $L_{W_{(\lambda_j)}}=-\Delta-\frac{N+2}{N-2}W_{(\lambda_j)}^{\frac{4}{N-2}}$.
The leading term in this equality is \eqref{main_beta'5}. We first prove that the other terms are of order $\OOO(\gamma^{\frac N2})+o_n(1)$.

\medskip

\noindent\emph{Estimates on lower order terms}.

We first note that by integration by parts
$$\int\left( \Lambda W \right)_{(\lambda_j)}L_{W_{(\lambda_j)}}h=\int L_{W_{(\lambda_j)}}\left( \Lambda W \right)_{(\lambda_j)}h=0,$$
so that the term on the right-hand side of \eqref{main_beta'1} is $0$. 

By H\"older inequality, and the estimate \eqref{R133} on $\sigma(h,v_L)$, we have
$$|\eqref{main_beta'2}|\lesssim \|v_L\|_{L^{\frac{2N}{N-2}}}+\|v_L\|_{L^{\frac{2N}{N-2}}}^{\frac{N+2}{N-2}}.$$
Since $v_L$ is a solution to the linear wave equation, we have
$$\lim_{n\to\infty}\|v_L(t)\|_{L^{\frac{2N}{N-2}}}=0,$$
which proves that the term \eqref{main_beta'2} is $o_n(1)$. 

To bound \eqref{main_beta'3}, we use the inequality
$$|F(a+b)-F(a)-F'(a)b|\lesssim b^2\indic_{\{|b|\leq |a|\}}a^{\frac{6-N}{N-2}}+b^{\frac{N+2}{N-2}}\indic_{\{|b|\geq |a|\}}$$
proved in the appendix (see Claim \ref{Cl:pointwise3}). We obtain
$$ |\eqref{main_beta'3}|\lesssim \int \left|(\Lambda W)_{(\lambda_j)}\right|\,h^{2}\indic_{\{|h|\leq W_{(\lambda_j)}\}}W^{\frac{6-N}{N-2}}_{(\lambda_j)}+\int \left|(\Lambda W)_{(\lambda_j)}\right|\,h^{\frac{N+2}{N-2}}\indic_{\{|h|\geq W_{(\lambda_j)}\}}.$$
Since $|(\Lambda W)_{(\lambda_j)}|\lesssim W_{(\lambda_j)}$, we deduce
$$|\eqref{main_beta'3}|\lesssim \int h^2W_{(\lambda_j)}^{\frac{4}{N-2}}+\int |h|^{\frac{2N}{N-2}}\lesssim \|h\|^2_{L^{\frac{2N}{N-2}}}+\|h\|^{\frac{2N}{N-2}}_{L^{\frac{2N}{N-2}}}\lesssim \gamma^{\frac{N}{2}}+o_n(1),$$
where we have used the estimate \eqref{bound_g1_h} on $h$.

To bound \eqref{main_beta'4}, we distinguish between the case $N\geq 7$ and the case $N=5$. If $N\geq 7$, we use the inequality 
$$|F(a+b+c)-F(a+b)-F(a+c)+F(a)|\lesssim |c||b|^{\frac{N+2}{2(N-2)}}|a|^{\frac{6-N}{2(N-2)}}$$ (see again Claim \ref{Cl:pointwise3}), with $a=\iota_jW_{(\lambda_j)}$, $b=\sum_{k\neq j}\iota_kW_{(\lambda_k)}$ and $c=h$. We obtain
$$|\eqref{main_beta'4}|\lesssim \int \left|\Lambda W_{(\lambda_j)}\right|\,W_{(\lambda_j)}^{\frac{6-N}{2(N-2)}}\,|h|\,\left|\sum_{j\neq k} \iota_kW_{(\lambda_k)}\right|^{\frac{N+2}{2(N-2}}.$$
Since $|(\Lambda W)_{(\lambda_j)}|\lesssim W_{(\lambda_j)}$, we deduce
$$|\eqref{main_beta'4}|\lesssim \sum_{k\neq j} \int |h| W_{(\lambda_k)}^{\frac{N+2}{2(N-2)}} W_{(\lambda_j)}^{\frac{N+2}{2(N-2)}}\lesssim \|h\|_{L^{\frac{2N}{N-2}}} \sum_{k\neq j} \left( \int W_{(\lambda_k)}^{\frac{N}{N-2}} W_{(\lambda_j)}^{\frac{N}{N-2}} \right)^{\frac{N+2}{2N}}.$$
By the estimate \eqref{bound_g1_h} on $h$ and the bound $\int W_{(\lambda_k)}^{\frac{N}{N-2}} W_{(\lambda_j)}^{\frac{N}{N-2}}\lesssim \gamma^{\frac{N}{2}}$ (see Claim \ref{Cl:estimates1}), we deduce that if $N\geq 7$,
$$ |\eqref{main_beta'4}|\lesssim \gamma^{\frac{N+1}{2}}+o_n(1).$$
If $N=5$, the inequality 
$$|F(a+b+c)-F(a+b)-F(a+c)+F(a)|\lesssim |c|\,|b|\,(|a|+|b|+|c|)^{\frac 13}$$ proved in Claim \ref{Cl:pointwise3}, with $a=\iota_jW_{(\lambda_j)}$, $b=\sum_{k\neq j}\iota_kW_{(\lambda_k)}$ and $c=h$ yields
$$ |\eqref{main_beta'4}|\lesssim \sum_{k\neq j} \left( \int W_{(\lambda_k)}^{\frac 43}W_{(\lambda_j)}|h| +\int W_{(\lambda_k)} W_{(\lambda_j)}^{\frac 43} |h|+\int |h|^{\frac 43}W_{(\lambda_j)}W_{(\lambda_k)}\right).$$
By H\"older's inequality, we deduce
$$|\eqref{main_beta'4}|\lesssim \|h\|_{L^{\frac{10}{3}}}\sum_{\substack{1\leq k,\ell\leq J\\ k\neq \ell}}\left\|W_{(\lambda_{\ell})}W_{(\lambda_k)}^{\frac{4}{3}}\right\|_{L^{\frac{10}{7}}}+\|h\|_{L^{\frac{10}{3}}}^{\frac 43}\sum_{\substack{1\leq k,\ell\leq J\\ k\neq \ell}}\left( \int W_{(\lambda_{\ell})}^{\frac{5}{3}}W_{(\lambda_k)}^{\frac{5}{3}} \right)^{\frac{3}{5}}.$$
Together with the estimates \eqref{est1.3} and \eqref{est1.4} of Claim \ref{Cl:estimates1} in the appendix, and the bound \eqref{bound_g1_h} of $h$, we deduce that when $N=5$,
$$|\eqref{main_beta'4}|\lesssim \gamma^{\frac 54}\gamma^{\frac{3}{2}}+\gamma^{\frac{5}{3}}\gamma^{\frac 32}+o_n(1)\lesssim \gamma^{\frac{11}{4}}+o_n(1).$$
As a conclusion
$$|\eqref{main_beta'1}|+|\eqref{main_beta'2}|+|\eqref{main_beta'3}|+|\eqref{main_beta'4}|\lesssim \gamma^{\frac N2}+o_n(1).$$

\medskip

\noindent\emph{Estimate on the leading term}.
To conclude the proof, we will prove:
\begin{multline}
 \label{LT160}
 \bigg|\int_{\RR^N} \left( F(M)-\sum_{k=1}^J F\left( \iota_kW_{(\lambda_k)} \right) \right)(\Lambda W)_{(\lambda_j)}\, dx\\
 -\kappa_0\left( \iota_{j+1}\left( \frac{\lambda_{j+1}}{\lambda_j} \right)^{\frac N2-1}-\iota_{j-1}\left( \frac{\lambda_j}{\lambda_{j-1}} \right)^{\frac N2-1} \right)\bigg|\lesssim \gamma^{\frac N2}.
\end{multline} 
We will prove \eqref{LT160} as a consequence of the following inequalities:
\begin{multline}
 \label{LT10}
 \bigg|\int_{\RR^N} \left( F(M)-\sum_{k=1}^J F\left( \iota_kW_{(\lambda_k)} \right) \right)(\Lambda W)_{(\lambda_j)}\, dx\\
 -\frac{N+2}{N-2}\int_{\RR^N} \left(\iota_{j+1}W_{(\lambda_j)}^{\frac{4}{N-2}}W_{(\lambda_{j+1})}+\iota_{j-1}W_{(\lambda_j)}^{\frac{4}{N-2}}W_{(\lambda_{j-1})}\right)(\Lambda W)_{(\lambda_j)}dx\bigg|\lesssim \gamma^{\frac N2}, 
\end{multline} 
(where by convention $\iota_0=\iota_{J+1}=0$),
\begin{multline}
 \label{LT120} 1\leq j\leq J-1\Longrightarrow \bigg|
 \int_{\RR^N} W_{(\lambda_j)}^{\frac{4}{N-2}}W_{(\lambda_{j+1})}(\Lambda W)_{(\lambda_j)}\\
 -\left( \frac{\lambda_{j+1}}{\lambda_j} \right)^{\frac N2-1}\frac{N^{\frac{N}{2}-1}(N-2)^{\frac{N}{2}+1}}{2(N+2)} \int \frac{1}{|x|^{N-2}}W^{\frac{N+2}{N-2}}dx\bigg|\lesssim \gamma^{\frac N2}
\end{multline}
and
\begin{multline}
 \label{LT150} 2\leq j\leq J\Longrightarrow \bigg|
 \int_{\RR^N} W_{(\lambda_j)}^{\frac{4}{N-2}}W_{(\lambda_{j-1})}(\Lambda W)_{(\lambda_j)}\\
 +\left( \frac{\lambda_{j}}{\lambda_{j-1}} \right)^{\frac N2-1}\frac{N^{\frac{N}{2}-1}(N-2)^{\frac{N}{2}+1}}{2(N+2)} \int \frac{1}{|x|^{N-2}}W^{\frac{N+2}{N-2}}dx\bigg|\lesssim \gamma^{\frac N2}.
\end{multline}

\medskip

\noindent\emph{Proof of \eqref{LT10}}
We adopt the convention $\lambda_0=+\infty$, $\lambda_{J+1}=0$. 
We first notice that there exists a constant $C$ such that for any $k,\ell\in \llbracket 1,J\rrbracket$, we have
\begin{equation}
\label{LT11}
\sqrt{\lambda_{k+1}\lambda_k}\leq |x|\leq \sqrt{\lambda_{k-1}\lambda_k}\Longrightarrow W_{(\lambda_{\ell})}\lesssim W_{(\lambda_k)} 
\end{equation} 
(this follows easily from the facts that $W$ is positive and $W(x)\approx \frac{1}{|x|^{N-2}}$ for large $|x|$).  To prove \eqref{LT10}, we write
\begin{equation}
 \label{LT30}
\int_{\RR^N} P_j(x)\,dx=\sum_{k=1}^J \int_{\sqrt{\lambda_{k}\lambda_{k+1}}\leq |x|\leq \sqrt{\lambda_{k}{\lambda_{k-1}}}}P_j(x)\,dx,
 \end{equation} 
where 
\begin{multline*}
P_j(x):=(\Lambda W)_{(\lambda_j)}\times 
\bigg( F(M)-\sum_{\ell=1}^J F\left( \iota_{\ell}W_{(\lambda_{\ell})} \right) \\
 -\frac{N+2}{N-2}\iota_{j+1}W_{(\lambda_j)}^{\frac{4}{N-2}}W_{(\lambda_{j+1})}-\frac{N+2}{N-2}\iota_{j-1}W_{(\lambda_j)}^{\frac{4}{N-2}}W_{(\lambda_{j-1})}\bigg).
\end{multline*}
If $j\neq k$, using that if $\sqrt{\lambda_k\lambda_{k+1}}\leq |x|\leq \sqrt{\lambda_k\lambda_{k-1}}$, 
$$|F(M)-F(\iota_kW_{(\lambda_k)})|\lesssim W_{(\lambda_k)}^{\frac{4}{N-2}}\sum_{\ell \neq k}W_{(\lambda_{\ell})},$$
and that $\left|(\Lambda W)_{(\lambda_j)}\right|\lesssim W_{(\lambda_j)}$,
we obtain
\begin{multline*}
\int_{\sqrt{\lambda_{k}\lambda_{k+1}}\leq |x|\leq \sqrt{\lambda_{k}{\lambda_{k-1}}}}|P_j(x)|\,dx\lesssim  \sum_{\ell \neq k}\int_{\sqrt{\lambda_{k}\lambda_{k+1}}\leq |x|\leq \sqrt{\lambda_{k}{\lambda_{k-1}}}}W_{(\lambda_k)}^{\frac{4}{N-2}}W_{(\lambda_\ell)}W_{(\lambda_j)}\\+\sum_{\ell\neq k}\int_{\sqrt{\lambda_{k}\lambda_{k+1}}\leq |x|\leq \sqrt{\lambda_{k}{\lambda_{k-1}}}}W_{(\lambda_{\ell})}^{\frac{N+2}{N-2}}W_{(\lambda_j)}\\
+ \int_{\sqrt{\lambda_{k}\lambda_{k+1}}\leq |x|\leq \sqrt{\lambda_{k}{\lambda_{k-1}}}} W_{\lambda_j}^{\frac{N+2}{N-2}}(W_{(\lambda_{j+1})}+W_{(\lambda_{j-1})})\,dx
\end{multline*}
Since $j\neq k$, and, by \eqref{LT11} $\ell\neq k\Rightarrow W_{(\lambda_{\ell})}\lesssim W_{(\lambda_k)}$ on the domain of integration, we can bound all the terms of the right-hand side of the preceding inequality by
\begin{multline*}
\sum_{\ell\neq k}   \int_{\sqrt{\lambda_{k}\lambda_{k+1}}\leq |x|\leq \sqrt{\lambda_{k}{\lambda_{k-1}}}}W_{(\lambda_k)}W_{(\lambda_{\ell})}^{\frac{N+2}{N-2}}
\\
\lesssim 
\sum_{\ell\neq k}   \int_{\sqrt{\lambda_{k}\lambda_{k+1}}\leq |x|\leq \sqrt{\lambda_{k}{\lambda_{k-1}}}}W_{(\lambda_k)}^{\frac{N}{N-2}}W_{(\lambda_{\ell})}^{\frac{N}{N-2}}
\lesssim \gamma^{\frac N2},
\end{multline*}
by Claim \ref{Cl:estimates1} in the appendix.
Thus we have proved
\begin{equation}
\label{LT50}
k\neq j\Longrightarrow
\int_{\sqrt{\lambda_{k}\lambda_{k+1}}\leq |x|\leq \sqrt{\lambda_{k}{\lambda_{k-1}}}}|P_j(x)|\,dx\lesssim \gamma^{\frac{N}{2}}
\end{equation} 
Next, by Claim \ref{Cl:pointwise3}, we observe, denoting by 
$$E_j:=\left\{x\in \RR^N\;:\;\left|\sum_{\ell\neq j}\iota_{\ell}W_{(\lambda_{\ell})}(x)\right|\leq W_{(\lambda_j)}(x)\right\},$$
that for all $x$ such that $\sqrt{\lambda_j\lambda_{j+1}}\leq |x|\leq \sqrt{\lambda_j\lambda_{j-1}}$ 
\begin{multline*}
\left|F(M)-F\left( \iota_jW_{(\lambda_j)} \right)-\frac{N+2}{N-2}W_{(\lambda_j)}^{\frac{4}{N-2}} \sum_{\ell\neq j} \iota_{\ell}W_{(\lambda_{\ell})}\right|\\
\lesssim \indic_{E_j}\sum_{\ell \neq j}  W_{(\lambda_{\ell})}^2W_{(\lambda_j)}^{\frac{6-N}{N-2}}+\indic_{\RR^N\setminus E_j}\left(\sum_{\ell\neq j} \iota_{\ell}W_{(\lambda_\ell)}\right)^{\frac{N+2}{N-2}}\lesssim \sum_{\ell\neq j} W_{(\lambda_{\ell})}^2W_{(\lambda_j)}^{\frac{6-N}{N-2}},
\end{multline*}
where 
\begin{itemize}
 \item If $N\geq 7$, 
we have used that $\frac{N+2}{N-2}<2$ and that on $\RR^N\setminus E_j$, $W_{(\lambda_j)}^2\lesssim \sum_{\ell\neq j}W_{(\lambda_{\ell})}^2$.
\item If $N=5$, $\frac{N+2}{N-2}=\frac 73>2$. However the preceding inequality holds for $\sqrt{\lambda_j\lambda_{j+1}}\leq |x|\leq \sqrt{\lambda_j\lambda_{j-1}}$ since in this set $\left|\sum_{\ell\neq j}\iota_{\ell}W_{(\lambda_{\ell})}(x)\right|\lesssim W_{(\lambda_j)}(x)$ and $6-N>0$.
\end{itemize}
As a consequence, using also $|(\Lambda W)_{(\lambda_j)}|\lesssim W_{(\lambda_j)}$,
\begin{multline*}
\int_{\sqrt{\lambda_{j}\lambda_{j+1}}\leq |x|\leq\sqrt{\lambda_{j}\lambda_{j-1}}}|P_j(x)|\,dx\lesssim \sum_{\ell \neq j}  \int_{\sqrt{\lambda_{j}\lambda_{j+1}}\leq |x|\leq\sqrt{\lambda_{j}\lambda_{j-1}}} W_{(\lambda_{\ell})}^2 W_{(\lambda_j)}^{\frac{4}{N-2}}\\
+\sum_{\ell\notin \{j-1,j,j+1\}} \int_{\RR^N} W_{(\lambda_{\ell})}W_{(\lambda_j)}^{\frac{N+2}{N-2}}. 
\end{multline*}
Using again \eqref{LT11}, we obtain, by \eqref{est1.4} in the appendix 
$$\sum_{\ell \neq j}  \int_{\sqrt{\lambda_{j}\lambda_{j+1}}\leq |x|\leq\sqrt{\lambda_{j}\lambda_{j-1}}} W_{(\lambda_{\ell})}^2 W_{(\lambda_j)}^{\frac{4}{N-2}}\lesssim \sum_{\ell\neq j}\int W_{(\lambda_{\ell})}^{\frac{N}{N-2}}W_{(\lambda_{j})}^{\frac{N}{N-2}}\lesssim \gamma^{\frac N2}.$$
Furthermore, if $\ell\notin\{j-1,j,j+1\}$, by estimate \eqref{est1.1} in the appendix,
$$ \int W_{(\lambda_{\ell})}W_{(\lambda_j)}^{\frac{N+2}{N-2}}=\int \nabla W_{(\lambda_{\ell})}\cdot \nabla W_{(\lambda_j)} \lesssim \min\left\{\left( \frac{\lambda_j}{\lambda_{\ell}}\right)^{\frac{N-2}{2}}, \left(\frac{\lambda_{\ell}}{\lambda_{j}}\right)^{\frac{N-2}{2}}\right\}\lesssim \gamma^{N-2}.$$
Combining, we obtain
\begin{equation}
\label{LT50bis}
\int_{\sqrt{\lambda_{j}\lambda_{j+1}}\leq |x|\leq\sqrt{\lambda_{j}\lambda_{j-1}}}|P_j(x)|\,dx\lesssim  \gamma^{\frac N2},
\end{equation}  
which, together with \eqref{LT50}, yields the desired inequality \eqref{LT10}.

\medskip

\noindent\emph{Proof of \eqref{LT120}}
Recall that
$$W(x)=\left( 1+\frac{|x|^2}{N(N-2)} \right)^\frac{2-N}{2},\quad \Lambda W(x)=x\cdot\nabla W(x)+\frac{N-2}{2}W.$$
Thus,
\begin{align}
 \label{LT100} |W(x)-1|+\left|\Lambda W(x)-\frac{N-2}{2}\right|&\lesssim |x|\\
 \label{LT101} \left|W(x)-\frac{\big(N(N-2)\big)^{\frac{N}{2}-1}}{|x|^{N-2}}\right|+\left|\Lambda W(x)+\frac{\big(N(N-2)\big)^{\frac{N}{2}}}{2N|x|^{N-2}}\right|&\lesssim \frac{1}{|x|^{N-1}}.
\end{align}
By \eqref{LT101},
\begin{multline*}
\int W_{(\lambda_j)}^{\frac{4}{N-2}}W_{(\lambda_{j+1})} (\Lambda W)_{(\lambda_j)}=\left( \frac{\lambda_j}{\lambda_{j+1}}\right)^{\frac{N-2}{2}}\int W^{\frac{4}{N-2}}(x)\Lambda W(x) W\left( \frac{\lambda_jx}{\lambda_{j+1}} \right)\,dx
\\
=\left( \frac{\lambda_{j+1}}{\lambda_{j}}\right)^{\frac{N-2}{2}}\int W^{\frac{4}{N-2}}(x)\Lambda W(x)\frac{\big(N(N-2)\big)^{\frac{N}{2}-1}}{|x|^{N-2}}\,dx +\OOO\left( \gamma^{\frac N2} \right),
\end{multline*}
where we have used that the integral $\int W^{\frac{4}{N-2}}\Lambda W\frac{1}{|x|^{N-1}}\,dx$ converges. Furthemore, by an easy integration by parts,
$$ \int W^{\frac{4}{N-2}}x\cdot \nabla W\frac{1}{|x|^{N-2}}\,dx=\frac{-2(N-2)}{N+2}\int \frac{1}{|x|^{N-2}}W^{\frac{N+2}{N-2}}\,dx,$$
and thus;
$$\int \frac{1}{|x|^{N-2}} W^{\frac{4}{N-2}}\Lambda W\,dx=\frac{(N-2)^2}{2(N+2)}\int \frac{1}{|x|^{N-2}}W^{\frac{N+2}{N-2}}\,dx.$$
Combining, we obtain \eqref{LT120}.

\medskip

\noindent\emph{Proof of \eqref{LT150}.}
By \eqref{LT100},
\begin{multline*}
 \int W_{(\lambda_j)}^{\frac{4}{N-2}}W_{(\lambda_{j-1})} (\Lambda W)_{(\lambda_j)}=\left( \frac{\lambda_j}{\lambda_{j-1}} \right)^{\frac{N-2}{2}} \int W^{\frac{4}{N-2}}\Lambda W(x) W\left( \frac{\lambda_j x}{\lambda_{j-1}} \right)\,dx\\
 =\left( \frac{\lambda_j}{\lambda_{j-1}} \right)^{\frac{N-2}{2}} \int W^{\frac{4}{N-2}}\Lambda W(x)\left( 1+\OOO\left( \frac{\lambda_j|x|}{\lambda_{j-1}} \right) \right)dx\\
 =\left( \frac{\lambda_j}{\lambda_{j-1}} \right)^{\frac{N-2}{2}} \int W^{\frac{4}{N-2}}\Lambda W(x)\,dx +\OOO\left( \gamma^{\frac N2} \right).
\end{multline*}
By a straightforward integration by parts, we obtain
$$\int W^{\frac{4}{N-2}}\Lambda W\,dx=-\frac{(N-2)^2}{2(N+2)}\int W^{\frac{N+2}{N-2}}\,dx,$$
and thus
\begin{equation}
 \label{LT130}
\int W_{(\lambda_j)}^{\frac{4}{N-2}}W_{(\lambda_{j-1})} (\Lambda W)_{(\lambda_j)}=-\frac{(N-2)^2}{2(N+2)} \left( \frac{\lambda_j}{\lambda_{j-1}} \right)^{\frac{N-2}{2}} \int W^{\frac{N+2}{N-2}}+\OOO\left( \gamma^{\frac{N}{2}} \right).
 \end{equation} 
Finally, we observe that, for $\sigma>0$,
$$W^{\frac{N+2}{N-2}}\left( \frac{N(N-2)}{\sigma} \right)=\frac{\sigma^{N+2}}{\big(N(N-2)\big)^{\frac{N+2}{2}}}W^{\frac{N+2}{N-2}}\left( N(N-2)\sigma \right),$$
and thus, by the change of variable $r=\frac{N(N-2)}{\sigma}$,
$$\int_0^{\infty} \frac{1}{r^{N-2}}W^{\frac{N+2}{N-2}}(r)r^{N-1}\,dr=\frac{1}{\big(N(N-2)\big)^{\frac{N-2}{2}}}\int_0^{\infty} W^{\frac{N+2}{N-2}}(\sigma)\sigma^{N-1}\,d\sigma.$$
Combining with \eqref{LT130}, we obtain \eqref{LT150}. The proof of \eqref{LT160} is complete.
\end{proof}
\begin{proof}[End of the proof of Proposition \ref{P:EDO}]
 We next gather the results of the preceding Lemmas to conclude the proof of Proposition \ref{P:EDO}.
 
 The estimate \eqref{bound_delta_gamma} is exactly \eqref{bound_delta} in Lemma \ref{St:energy}.
 
 By \eqref{F182}, \eqref{F200} and \eqref{bound_delta},
 \begin{equation*}
  \left|\beta_j-\lambda_j'\|\Lambda W\|^2_{L^2}\right|\lesssim \left|\beta_j+\alpha_j \|\Lambda W\|^2_{L^2}\right|+\left|\lambda_j'+\alpha_j\|\Lambda W\|^2_{L^2}\right|\lesssim \gamma^{\frac{N}{4}}+o_n(1),
 \end{equation*}
hence \eqref{EDO2a}.

Combining \eqref{R260}, \eqref{F182} and \eqref{energy_bound}, we have
\begin{multline*}
 \left|\frac{1}{2}\sum_{j=1}^J\beta_j^2-\kappa_1\sum_{j=1}^{J-1}\iota_{j}\iota_{j+1}\left( \frac{\lambda_{j+1}}{\lambda_j} \right)^{\frac{N-2}{2}}\right| \lesssim \frac{1}{2}\left|\|\Lambda W\|^4_{L^2}\sum_{j=1}^J\alpha_j^2-\sum_{j=1}^J \beta_j^2\right|\\+\frac{1}{2}\left|\delta^2-\|\Lambda W\|^2_{L^2}\sum_{j=1}^J\alpha_j^2\right|\,\|\Lambda W\|^2_{L^2}+\left|\frac 12\delta^2-\kappa_1'\sum_{j=1}^J \iota_j\iota_{j+1}\left( \frac{\lambda_{j+1}}{\lambda_j} \right)^{\frac{N-2}{2}}\right|\,\|\Lambda W\|^2_{L^2}\\
 \lesssim \gamma^{\frac{N-1}{2}}+o_n(1).
\end{multline*}
Hence \eqref{EDO3a}. We have used that \eqref{F182} and the estimates $|\beta_j|+|\alpha_j|\lesssim \gamma^{\frac{N-2}{4}}$ (consequence of \eqref{R260}, \eqref{F182} and \eqref{bound_delta}) implies:
$$\left|\beta_j^2-\|\Lambda W\|_{L^2}^4\alpha_j^2\right|=\left|(\beta_j-\|\Lambda W\|_{L^2}^2\alpha_j)(\beta_j+\|\Lambda W\|_{L^2}^2\alpha_j)\right|\lesssim \gamma^{\frac{N}{4}}\gamma^{\frac{N-2}{4}}+o_n(1).$$

Finally \eqref{EDO4a} is exactly \eqref{F270} in Lemma \ref{L:second_derivative}.
\end{proof}
\section{End of the proof}
\label{S:end_of_proof}
\subsection{Exit time for a system of differential inequalities}
Using Proposition \ref{P:EDO} and a lower bound on one of the scaling parameters $\lambda_j$ (consequence of Proposition \ref{P:ext_scaling}), we will reduce the proof to the following proposition:
\begin{prop}
\label{P:rigidity}
Let $C>0$, $J_0\geq 2$, $a>0$. There exists $\eps=\eps(C,J_0,a)>0$, such that for all $L>0$, there exists $T^*=T^*(L,C,J_0,a)$ with the following property. For all $T>0$, for all $C^1$ functions 
\begin{align*}
\lambdabf=(\lambda_j)_j: [0,T]&\to G_{J_0}\\
\betabf=(\beta_j)_{j}:[0,T]&\to \RR^{J_0}
\end{align*}
satisfying, for all $t\in [0,T]$,
\begin{gather}
\label{EDO1}
\gamma(\lambdabf)=:\gamma\leq \eps\\
\label{EDO2}
 \forall j,\quad
 \left|\beta_j-\|\Lambda W\|^2_{L^2}\lambda'_j\right|\leq C \gamma^{\frac N4}\\
 \label{EDO3}
 \left|\frac{1}{2}\sum_{j=1}^{J_0}\beta_j^2-\kappa_1\sum_{1\leq j\leq {J_0}-1} \iota_j\iota_{j+1}\left( \frac{\lambda_{j+1}}{\lambda_j} \right)^{\frac{N-2}{2}}\right|\leq C\gamma^{\frac{N-1}{2}}\\
 \label{EDO4}
 \forall j,\quad \left|\lambda_j\beta'_j+\kappa_0\left( \iota_j\iota_{j+1} \left( \frac{\lambda_{j+1}}{\lambda_j}\right)^{\frac{N-2}{2}}-\iota_j\iota_{j-1}\left( \frac{\lambda_j}{\lambda_{j-1}} \right)^{\frac{N-2}{2}} \right) \right|\leq C\gamma^{\frac{N-1}{2}}\\
 \label{EDO5}
 L\leq C\gamma^{\frac{N-2}{2}}\left(\frac{\lambda_1}{\lambda_1(0)}\right)^a
 \end{gather}
we have 
$$T\leq T^*\lambda_1(0).$$
\end{prop}
\begin{remark}
 Let us emphasize that $T^*$ is independent of $\eps>0$ if it is chosen small enough, and that $\eps$ does not depend on $L$.
\end{remark}
We postpone the proof of Proposition \ref{P:rigidity} to Subsection \ref{SS:rigidity} and conclude the proof of Theorem \ref{T:resolution} in the two next subsections.
In view of Proposition \ref{P:EDO}, $(\lambda_j)_{1\leq j\leq J}$, and $(\beta_j)_{1\leq j\leq J}$ satisfy the assumptions of Proposition \ref{P:rigidity},  except for the lower bound \eqref{EDO5} and up to terms that are $o_n(1)$. In Subsection \ref{SS:interval}, we will eliminate the $o_n(1)$ terms. In order to do this, we will ignore all the exterior profiles that are equal to $\pm W$, restricting to indices $j\in \llbracket \widetilde{J},J\rrbracket$ for an appropriate index $\widetilde{J}$ and to a time interval  $[\tilde{t}_n,t_n']$ strictly included in $[\tilde{t}_n,t_n]$. In Subsection \ref{SS:exterior} we will show that the new exterior scaling parameter $\lambda_{\widetilde{J}}$ satisfies the lower bound \eqref{EDO5} and conclude the proof of Theorem \ref{T:resolution} assuming Proposition \ref{P:rigidity}. Finally in Subsection \ref{SS:rigidity} we prove Proposition \ref{P:rigidity}.
\subsection{Restriction of the indices and of the time interval}
\label{SS:interval}
Recall from Subsection \ref{SS:setting} the definitions of $t_n$, $\tilde{t}_n$, $J$, and for $j\in \llbracket 1,J\rrbracket$, $\iota_j$, $\alpha_j(t)$, $\beta_j(t)$, $\lambda_j(t)$.

After extraction of subsequences, the following weak limits exits in $\HHH$:
\begin{equation}
 \label{D10}
 \left( \widetilde{U}_0^{j}, \widetilde{U}_1^{j}\right)=\underset{n\to \infty}{\wlim} \left( \lambda_j(\tilde{t}_n)^{\frac N2 -1} U(\tilde{t}_n,\lambda_j(\tilde{t}_n)\cdot), \lambda_j(\tilde{t}_n)^{\frac N2 } \partial_tU(\tilde{t}_n,\lambda_j(\tilde{t}_n)\cdot)\right),
\end{equation} 
where as before $U=u-v_L$. We note that there exists $j\in \llbracket 1, J\rrbracket$ such that $\left( \widetilde{U}_0^{j}, \widetilde{U}_1^{j}\right)\neq \left( \iota_jW,0 \right)$. If not, for all $k\in \llbracket 1, K-1\rrbracket$,  by \eqref{R40}, \eqref{R41} and \eqref{R42}, $j_{k+1}=j_k+1$ and by the definition \eqref{RR3} of $j_k$, we see that for all $j\in \llbracket 1,J\rrbracket$
\begin{equation*}
\lim_{n\to+\infty} \frac{\lambda_{j+1}(\tilde{t}_n)}{\lambda_{j}(\tilde{t}_n)}=0.
 \end{equation*} 
This implies $\lim_{n\to\infty} \gamma\left( \tilde{t}_n \right)=0$ yielding, by Proposition \ref{P:EDO}, $\lim_{n\to\infty} \delta\left( \tilde{t}_n \right)=0$, a contradiction with the definition of $\tilde{t}_n$. 
 We define $\tJ$ as the unique index in $\llbracket 1, J\rrbracket$ such that:
 \begin{gather}
  \label{D24}
   \forall j\in \llbracket 1,\tJ-1\rrbracket, \quad \left( \tU_0^j,\tU_1^j \right)=\left( \iota_jW,0 \right)\\
   \label{D25}
   \left( \tU_0^{\tJ},\tU_1^{\tJ} \right)\neq \left( \iota_{\tJ}W,0 \right).
 \end{gather}
If $\left( \tU_0^1,\tU_1^1 \right)\neq (\iota_1W,0)$, we let $\tJ=1$. By the argument above,
\begin{equation}
 \label{D11}
\forall j\in \llbracket 1,\tJ-1\rrbracket, \quad \lim_{n\to+\infty} \frac{\lambda_{j+1}(\tilde{t}_n)}{\lambda_{j}(\tilde{t}_n)}=0.
 \end{equation} 
We denote
\begin{equation}
 \label{D12}
 \lambda_{\tJ,n}:=\lambda_{\tJ}(\tilde{t}_n), \quad \tgamma(t):= \gamma \big( (\lambda_{j}(t))_{\tJ\leq j\leq J}\big)=
\max_{\widetilde{J}\leq j\leq J-1} \frac{\lambda_{j+1}(t)}{\lambda_j(t)}.
\end{equation} 
In this subsection, we prove the following lemmas
\begin{lemma}
\label{L:D20}
\begin{equation}
 \label{D20}
 \lim_{n\to\infty} \frac{t_n-\tilde{t}_n}{\lambda_{\tJ,n}}=+\infty.
\end{equation} 
\end{lemma}
\begin{lemma}
 \label{L:D21}
 Let $T>0$ and 
$$t_n'=\tilde{t}_n+T\lambda_{\tJ,n}.$$
 Then for large $n$, for all $t\in \left[\tilde{t}_n,t_n'\right]$ and for all $j\in \llbracket \widetilde{J},J\rrbracket$,
\begin{gather}
\label{D21}
 \left|\beta_j-\|\Lambda W\|^2_{L^2}\lambda'_j\right|\leq C \tgamma^{\frac N4}\\
 \label{D23}
\Bigg|\lambda_j\beta'_j+\kappa_0\Bigg( \iota_j\iota_{j+1} \bigg( \frac{\lambda_{j+1}}{\lambda_j}\bigg)^{\frac{N-2}{2}}-\iota_j\iota_{j-1}\bigg( \frac{\lambda_j}{\lambda_{j-1}} \bigg)^{\frac{N-2}{2}} \Bigg) \Bigg|
 \leq C\tgamma^{\frac{N-1}{2}},
 \end{gather}
 and for all $t\in [\tilde{t}_n,t_n']$,
 \begin{equation}
 \label{D22}
 \left|\frac{1}{2}\sum_{j=\widetilde{J}}^{J}\beta_j^2-\kappa_1\sum_{\widetilde{J}\leq j\leq {J}-1} \iota_j\iota_{j+1}\left( \frac{\lambda_{j+1}}{\lambda_j} \right)^{\frac{N-2}{2}}\right|\leq C\tgamma^{\frac{N-1}{2}}.
 \end{equation}
\end{lemma}
\begin{proof}[Proof of Lemma \ref{L:D20}]

 \setcounter{step}{0}
 \begin{step}[Expansion of the solution along the sequence $\{\tilde{t}_n\}_n$]
  Extracting subsequences if necessary, we introduce, as in the beginning of Subsection \ref{SS:expansion}, the following partition of $\llbracket 1,J\rrbracket$:
$$\llbracket 1,J\rrbracket=\bigcup_{k=1}^{\widetilde{K}}\llbracket \tilde{\jmath}_k,\tilde{\jmath}_{k+1}-1\rrbracket,$$
with
\begin{equation*}
1=\tilde{\jmath}_1<\tilde{\jmath}_2<\ldots<\tilde{\jmath}_{\widetilde{K}+1}=J+1
\end{equation*}
and, letting for all $k\in \llbracket 1,\tK\rrbracket$,
$$ \lambda_{k,n}=\lambda_{\tj_k}(\tilde{t}_n),$$
we have, 
\begin{align}
\label{D30}
\forall k&\in \llbracket 1,\tK\rrbracket,\quad 
\forall j\in \llbracket \tj_k,\tj_{k+1}-1\rrbracket,\quad \lim_{n\to\infty}\frac{\lambda_j(\tilde{t}_n)}{\lambda_{k,n}}>0\\
\label{D30'}
\forall k&\in \llbracket 1,\tK-1\rrbracket,\quad \lim_{n\to\infty} \frac{\lambda_{k+1,n}}{\lambda_{k,n}}=0.
\end{align}
We denote
\begin{equation}
 \label{defU01k}
 \left( U_0^k,U_1^k \right):=\underset{n\to \infty}{\wlim} \left( \lambda_{k,n}^{\frac{N-2}{2}} U\left(\tilde{t}_n,\lambda_{k,n}\cdot \right),\lambda_{k,n}^{\frac{N}{2}} \partial_t U\left(\tilde{t}_n,\lambda_{k,n}\cdot \right)\right).
\end{equation} 
Note that by \eqref{D11},
\begin{equation}
\label{jk=k}
 \forall k\in \llbracket 1,\tJ-1\rrbracket,\; \left( U_0^k,U_1^k \right)=(\iota_k W,0),\quad\forall k\in \llbracket 1,\tJ\rrbracket,\; \tj_k=k.
\end{equation}
We let $U^k$ be the solution of the nonlinear wave equation \eqref{NLW} with initial data $\left( U_0^k,U_1^k \right)$. 
According to Lemma \ref{L:expansion}, $U^k$ is defined for $|x|>|t|$, non-radiative and, denoting
\begin{equation}
\label{D32}
U_n^k(t,x)=\frac{1}{\lambda_{k,n}^{\frac{N-2}{2}}} U^k\left( \frac{t}{\lambda_{k,n}},\frac{x}{\lambda_{k,n}} \right),
\end{equation} 
we can expand $\vec{u}(\tilde{t}_n)$ as follows:
\begin{equation} 
 \label{R160}
 \lim_{n\to\infty} \left\|\vec{u}(\tilde{t}_n)-\vv_L(\tilde{t}_n)-\sum_{k=1}^{\widetilde{K}} \vec{U}_n^k(0)\right\|_{\HHH}=0.
\end{equation} 
We now make a crucial observation on $\left( U_0^{\tJ},U_1^{\tJ} \right)$. Since by the definition of $\tJ$, $\left(U_0^{\tJ},U_1^{\tJ}  \right)\neq \left( \iota_{\tJ}W,0 \right)$, we see by the analog of the expansion \eqref{LRR3} (where $k=\tJ$ and $(V_0^k,V_1^k)$ has to be replaced by $(U_0^{\tJ},U_1^{\tJ})$), the orthogonality relations \eqref{RR4} and the uniqueness in Lemma \ref{L:ortho_scaling}, that $\left( U^{\tJ}_0,U^{\tJ}_1 \right)$ is not the initial data of a stationary solution. Since the corresponding solution is non-radiative, we deduce from Theorems \ref{T:asymptotic_NR} and \ref{T:uniqueness_NR} that there exists $p_0\in \left\llbracket 1,\frac{N-1}{2}\right\rrbracket$ and $\ell\in \RR\setminus\{0\}$ such that for all $t\in \RR$ and for all $R$ large (depending on $t$),
\begin{equation}
\label{R170}
\left\|\vec{U}^{\widetilde{J}}(t)-\ell \,\Xi_{p_0}\right\|_{\HHH(R)}\lesssim \max\left(\frac{1}{R^{(p_0-\frac{1}2)\frac{N+2}{N-2}}},\frac{1}{R^{p_0+\frac 12}}\right),
 \end{equation} 
 where the implicit constant might depend on $R$ (but of course not on $t$).
 \end{step}
\begin{step}[Contradiction argument]
Assuming that \eqref{D20} does not hold we have, extracting subsequences if necessary, 
\begin{equation}
 \label{R190}
 \lim_{n\to\infty} \frac{t_n-\tilde{t}_n}{\lambda_{\tJ,n}}=T\in [0,\infty).
\end{equation} 
By the expansion \eqref{R60bis} with $\tau=t_n-\tilde{t}_n$, $s_n=\tilde{t}_n$,
\begin{equation}
 \label{R200}
 \vec{u}(t_n)=\vec{v}_L(t_n)+\sum_{k=1}^{\widetilde{K}} \vec{U}_n^k(t_n-\tilde{t}_n)+\vec{r}_n(t_n-\tilde{t}_n),\; |x|>|t_n-\tilde{t}_n|
\end{equation} 
where 
\begin{equation}
 \label{R201}
 \lim_{n\to\infty} \int_{|x|\geq |t_n-\tilde{t}_n|} |\nabla_{t,x}r_n(t_n-\tilde{t}_n)|^2\,dx=0. 
\end{equation}
Let
\begin{equation}
 \label{R202}
 \left( A_0^{\widetilde{J}},A_1^{\widetilde{J}} \right)
 =\underset{n\to\infty}{\wlim}\left( \lambda_{\widetilde{J},n}^{\frac{N-2}{2}}u(t_n,\lambda_{\widetilde{J},n}\cdot),\lambda_{\widetilde{J},n}^{\frac N2}\partial_tu(t_n,\lambda_{\widetilde{J},n}\cdot )\right).
\end{equation} 
We claim
\begin{equation}
 \label{R203}
 \left( A_0^{\widetilde{J}},A_1^{\widetilde{J}} \right)(x)=\vec{U}^{\widetilde{J}}(T,x),\quad |x|> |T|,
\end{equation} 
where $T$ is defined by \eqref{R190}. Indeed, let $\varphi\in C^{\infty}_0(\{x\in \RR^N,\; |x|> T\})$. Then by the definition \eqref{D32} of $U^{\tJ}_n$,
\begin{multline*}
 \int \lambda_{\tJ,n}^{\frac{N-2}{2}} U_n^{\tJ}\left( t_n-\tilde{t}_n,\lambda_{\tJ,n}x \right)\varphi(x)dx=\int U^{\tJ}\left( \frac{t_n-\tilde{t}_n}{\lambda_{\tJ,n}},x \right)\varphi(x)dx\\
 \underset{n\to\infty}{\longrightarrow} \int U^{\tJ}(T,x)\varphi(x)dx,
\end{multline*}
where we have used \eqref{R190}, and the fact that $U^{\tJ}_{\restriction \{|x|>|t|\}}$ is the restriction to $\{|x|>|t|\}$ of an element of $C^0\left( \RR,\hdot \right)$ (see Definition \ref{D:sol_cone}). Furthermore, if $k\in \llbracket 1,\tK\rrbracket \setminus \{\tJ\}$,
\begin{align}
 \label{D50}
 \int \lambda_{\tJ,n}^{\frac{N-2}{2}}&U_n^k(t_n-\tilde{t}_n,\lambda_{\tJ,n}x)\varphi(x)dx\\
 \label{D52}
 &= \int \left( \frac{\lambda_{k,n}}{\lambda_{\tJ,n}} \right)^{\frac{N+2}{2}} U^k\left( \frac{t_n-\tilde{t}_n}{\lambda_{k,n}},y\right)\varphi\left( \frac{\lambda_{k,n}}{\lambda_{\tJ,n}}y \right)dy.
\end{align}
Using that there exists $T'>T$ such that $|x|\geq T'$ in the support of $\varphi$, we see by \eqref{R190} that $|y|>\frac{t_n-\tilde{t}_n}{\lambda_{k,n}}$ for large $n$ in the support of the integrand in \eqref{D52}. Thus for large $n$, \eqref{D52} (or equivalently \eqref{D50}) does not depend on the values of $U^k(t,x)$ for $|x|\leq |t|$. 
Recall that, after extraction,
\begin{equation}
\label{ortho_k_tJ}
\lim_{n\to \infty}\frac{\lambda_{k,n}}{\lambda_{\tJ,n}}\in \{0,+\infty\}. 
\end{equation} 
If (after extraction) $\lim_{n}\frac{t_n-\tilde{t}_n}{\lambda_{k,n}}=\sigma\in [0,\infty)$, then, using that $U^{k}_{\restriction \{|x|>|t|\}}$ is the restriction to $\{|x|>|t|\}$ of an element of $C^0\left( \RR,\hdot \right)$, 
$$\lim_{n\to \infty} U^k\left( \frac{t_n-\tilde{t}_n}{\lambda_{k,n}}\right)=U^k(\sigma)\text{ in }\hdot\left( \{|x|>\sigma\}\right)$$ 
and, since $\left( \frac{\lambda_{k,n}}{\lambda_{\tJ,n}} \right)^{\frac{N+2}{2}}\varphi\left( \frac{\lambda_{k,n}}{\lambda_{\tJ,n}} \cdot\right)$ converges weakly to $0$ in $\dot{H}^{-1}$, we have that \eqref{D52} (and hence \eqref{D50}) goes to $0$ as $n$ goes to infinity. On the other hand, if $\lim_{n}\frac{t_n-\tilde{t}_n}{\lambda_{k,n}}=+\infty$, using that by Lemma \ref{L:expansion}, $U^k$ is nonradiative, we obtain again that \eqref{D50} goes to $0$ as $n$ goes to infinity.

Finally, we have 
$$\lim_{n\to\infty}\int \lambda_{\tJ,n}^{\frac{N-2}{2}}r_n\left( t_n-\tilde{t}_n,\lambda_{\tJ,n}x\right)\varphi(x)dx=0$$
by \eqref{R201} and \eqref{R190}, and
$$\lim_{n\to\infty}\int \lambda_{\tJ,n}^{\frac{N-2}{2}}v_L\left( t_n,\lambda_{\tJ,n}x\right)\varphi(x)dx=0$$
by the standard asymptotics of the linear wave equation. This yields
$$\forall \varphi \in  C_0^{\infty}\left( \{|x|>T\} \right),\quad \int A_0^{\tJ}(x)\varphi(x)dx=\int U_0^{\tJ}(T,x)\varphi(x)dx$$
and thus, arguing similarly on $A_1^{\tJ}$ and $U_1^{\tJ}$, the desired equality \eqref{R203} follows. 

Since $\lim_{n\to\infty} d_{J,\iotabf}=0$, we obtain that $(A_0^{\tJ},A_1^{\tJ})=(0,0)$ or  $(A_0^{\tJ},A_1^{\tJ})=(\pm W_{(\mu)},0)$ for some sign $\pm$ and scaling parameter $\mu>0$, contradicting \eqref{R170}, and concluding the proof of Lemma \ref{L:D20}.
\end{step}
\end{proof}
\begin{remark}
\label{R:important}
 The same proof yields the following: let $\{s_n\}$ be a sequence of times with $s_n\in [\tilde{t}_n,t_n]$ such that 
 $$\lim_{n\to\infty} d_{J,\iotabf}(s_n)=0.$$
 Then
 $$\lim_{n\to\infty}\frac{s_n-\tilde{t}_n}{\lambda_{\widetilde{J},n}}=+\infty.$$
\end{remark}
\begin{proof}[Proof of Lemma \ref{L:D21}]
 We will show
\begin{gather}
 \label{D90}
\forall j \in \llbracket 1,\tJ-1\rrbracket,\quad \lim_{n\to\infty} \max_{\tilde{t}_n\leq t\leq t_n'} |\beta_j(t)|+\frac{\lambda_{j+1}(t)}{\lambda_{j}(t)}=0\\
\label{D91}
\liminf_{n\to\infty} \min_{\tilde{t}_n\leq t\leq t_n'} \tgamma(t)>0.
\end{gather}
Assuming \eqref{D90} and \eqref{D91}, the conclusion of Lemma \ref{L:D21} follows easily from Proposition \ref{P:EDO}. Indeed by \eqref{D90} and \eqref{D91}, restricting $t$ to $[\tilde{t}_n,t_n']$, we can replace $\gamma(t)+o_n(1)$ by $\tilde{\gamma}(t)$ in the right-hand side of \eqref{EDO2a}, \eqref{EDO3a} and \eqref{EDO4a}. Si\-mi\-larly, using again \eqref{D90} and \eqref{D91}, we can restrict the indices in the sums in the left-hand side of \eqref{EDO3a} to $\tJ\leq j$.
This yields that \eqref{D21}, \eqref{D23} (for $j\in [\tJ,J]$) and \eqref{D22} hold for all large $n$ and $t\in[\tilde{t}_n,t_n']$.

\medskip
\noindent\emph{Proof of \eqref{D90}.}
We first claim
\begin{equation}
  \label{R224}
  \forall j\in \llbracket 1,\widetilde{J}-1\rrbracket, \quad \lim_{n\to\infty} \sup_{\tilde{t}_n\leq t\leq t_n'} \left|\frac{\lambda_j(t)}{\lambda_j(\tilde{t}_n)}-1\right|=0
 \end{equation} 
Indeed, by the estimates \eqref{EDO2a} and \eqref{EDO3a}, which imply $|\lambda'|\leq C\gamma^{\frac{N-2}{4}}+o_n(1)$, we see that
$$ \forall t\in [\tilde{t}_n,t_n'] ,\quad \left|\lambda_j(t)-\lambda_j(\tilde{t}_n)\right|\lesssim \sup_{t\in [\tilde{t}_n,t_n']} \left(\gamma^{\frac{N-2}{4}}(t)+o_n(1)\right)(t_n'-\tilde{t}_n),$$
and thus, for large $n$,
\begin{equation}
 \label{R191}
 \left|1-\frac{\lambda_j(t)}{\lambda_j(\tilde{t}_n)}\right|\lesssim \frac{t_n'-\tilde{t}_n}{\lambda_j(\tilde{t}_n)}=T\frac{\lambda_{\tJ,n}}{\lambda_j(\tilde{t}_n)}=T\frac{\lambda_{\tJ}(\tilde{t}_n)}{\lambda_j(\tilde{t}_n)},
\end{equation} 
and \eqref{R224} follows in view of the fact that, by \eqref{D11} and the definition of $\tJ$, 
\begin{equation}
\label{D11'}
\forall j\in \llbracket 1,\tJ-1\rrbracket,\quad
\lim_{n\to\infty} \frac{\lambda_{j}(\tilde{t}_n)}{\lambda_{j+1}(\tilde{t}_n)}=+\infty. 
\end{equation} 
Combining \eqref{R224} with \eqref{D11'}, we obtain 
\begin{equation}
 \label{D110} 
\forall j\in \llbracket 1,\tJ-1\rrbracket,\quad
\sup_{\tilde{t}_n\leq t \leq t_n'} \frac{\lambda_{j+1}(t)}{\lambda_{j}(t)}=0. 
\end{equation} 
By the definition \eqref{R155} of $\beta_j$ and the definition of $\tJ$, we also have
\begin{equation} 
 \label{D111}
\forall j\in \llbracket 1,\tJ-1\rrbracket,\quad
\lim_{n\to \infty} \left|\beta_j(\tilde{t}_n)\right|=0
\end{equation} 
By Lemma \ref{L:second_derivative},
$$\forall t\in \left[\tilde{t}_n,t_n'\right],\; \forall j\in \llbracket 1,\tJ-1\rrbracket,\quad \left|\beta'_j(t)\right|\lesssim \frac{\gamma^{\frac N2 -1}(t)+o_n(1)}{\lambda_j(t)}\lesssim \frac{1}{\lambda_j(t)}.$$
Integrating in time, we obtain 
$$\left|\beta_j(t)-\beta_j(\tilde{t}_n)\right|\lesssim \frac{\lambda_{\tilde{J},n}}{\lambda_j(t)}\lesssim \sup_{s\in [\tilde{t}_n,t_n']} \frac{\lambda_{\tJ}(\tilde{t}_n)}{\lambda_j(s)},$$
where the implicit constants depend on $T$.
By \eqref{R224} and \eqref{D110}, the right-hand side of the preceding inequality goes to $0$ as $n$ tends to $\infty$. Combining with \eqref{D110} and \eqref{D111}, we obtain \eqref{D90}.
\medskip

\noindent\emph{Proof of \eqref{D91}}

We argue by contradiction, assuming, after extraction of a subsequence, that there exists a sequence of times $\{s_n\}_n$ with $s_n\in [\tilde{t}_n,t_n']$ such that $\lim_{n\to\infty}\tgamma(s_n)=0$. By \eqref{D90}, $\lim_{n\to\infty}\gamma(s_n)=0$. By Lemma \ref{St:energy}, $\lim_{n\to\infty}\delta(s_n)=0$. Thus $\lim_{n\to\infty} d_{J,\iotabf}(s_n)=0$, a contradiction with the conclusion of Remark \ref{R:important} since $\tilde{t}_n\leq s_n\leq \tilde{t}_n+T\lambda_{\tJ,n}$.
\end{proof}

\subsection{Lower bound for the exterior scaling parameter and end of the proof}
\label{SS:exterior}
In this subsection we conclude the contradiction argument started in Section \ref{S:reduction}, using the same notations as in Section \ref{S:reduction} and Subsection \ref{SS:interval}. We recall in particular that $t_n'=\tilde{t}_n+T\lambda_{\tilde{J},n}$, where the sequence $\{\tilde{t}_n\}_n$ is defined by \eqref{R14}, \eqref{R15} and \eqref{R16}, $\lambda_{\tilde{J},n}=\lambda_{\tJ}(\tilde{t}_n)$, and $T$ is a large positive parameter that will be chosen at the end of the proof in order to obtain a contradiction with Proposition \ref{P:rigidity}. In all the argument $T$ is chosen first, and the small parameter $\eps_0$  appearing in \eqref{R16} might depend on $T$. The constants are independent of $\eps_0$.

\begin{lemma}
\label{L:exterior}
Let $\ell$ and $p_0$ be defined by \eqref{R170}. Then if $\eps_0=\eps_0(T)$ is chosen small enough, there exists a constant $C>0$ such that for large $n$,
\begin{equation*}
 \forall t\in [\tilde{t}_n,t_n'],\quad 
|\ell|\leq C\left(
\frac{\lambda_{\widetilde{J}}(t)}{\lambda_{\widetilde{J}}(\tilde{t}_n)}
\right)^{p_0-\frac{1}{2}}\delta(t)^{\frac{2}{N}}.
\end{equation*} 
\end{lemma}
\begin{proof}[End of the proof of Theorem \ref{T:resolution}]
We first assume Lemma \ref{L:exterior} and conclude the proof of Theorem \ref{T:resolution}. By \eqref{bound_delta_gamma}, for large $n$
$$ \forall t\in [\tilde{t}_n,t_n'],\quad \delta(t)^{\frac{2}{N}} \lesssim \gamma(t)^{\frac{N-2}{2N}}+o_n(1)\lesssim \gamma(t)^{\frac{N-2}{2}}+o_n(1).$$
Combining with \eqref{D90} and \eqref{D91}, we see that
$$ \forall t\in [\tilde{t}_n,t_n'],\quad \delta(t)^{\frac{2}{N}} \lesssim \tgamma(t)^{\frac{N-2}{2}}.$$
Thus Lemma \ref{L:exterior} implies that there exists a constant $C>0$ such that 
\begin{equation}
 \label{D120} 
 \forall t\in [\tilde{t}_n,t_n'],\quad 
|\ell|\leq C\left(
\frac{\lambda_{\widetilde{J}}(t)}{\lambda_{\widetilde{J}}(\tilde{t}_n)}
\right)^{p_0-\frac{1}{2}}\tgamma(t)^{\frac{N-2}{2}}.
\end{equation} 
By \eqref{D120} and the estimates \eqref{D21}, \eqref{D23}, \eqref{D22} of the preceding subsection, the parameters $\left(\beta_j\right)_{\tJ\leq j\leq J}$ and $\left(\lambda_j\right)_{\tJ\leq j\leq J}$ satisfy the assumptions of Proposition \ref{P:rigidity} for times $t\in \left[\tilde{t}_n,t_n'\right]$. The conclusion of the Proposition yields $t_n'-\tilde{t}_n\leq T_*\lambda_{\tJ}(\tilde{t}_n)$ for large $n$, and thus $T\leq T_*$, for a constant $T_*$ depending only on the solution $u$ and the parameters $\ell$, $p_0$. Since we can take $T$ arbitrarily large we obtain a contradiction, concluding the proof of the Theorem except for the fact that 
$$\lim_{t\to\infty}\frac{\lambda_1(t)}{t}=0,$$
which follows from finite speed of propagation and the small data theory (see e.g. the proof of (3.53) in \cite{DuKeMe13}). 
\end{proof}
\begin{proof}[Proof of Lemma \ref{L:exterior}]
We argue by contradiction, assuming (after extraction) that there exists a large constant $M$ and a sequence $\{\tilde{s}_n\}_n$ with $\tilde{s}_n\in [\tilde{t}_n,t_n']$ and
\begin{equation}
 \label{D121}
 \forall n,\quad  |\ell|\geq M \left( \frac{\lambda_{\tJ}(\tilde{s}_n)}{\lambda_{\tJ}(\tilde{t}_n)} \right)^{p_0-1/2}\delta(\tilde{s}_n)^{\frac 2N}.
\end{equation} 
We use the expansion \eqref{R60bis}, as in the proof of Lemma \ref{L:D20} (see \eqref{R200}) at $s_n=\tilde{t}_n$, $\tau=\tilde{s}_n-\tilde{t}_n$. This yields
\begin{equation*}
 \vec{u}(\tilde{s}_n)=\vec{v}_L(\tilde{s}_n)+\sum_{k=1}^{\widetilde{K}} \vec{U}_n^k(\tilde{s}_n-\tilde{t}_n)+\vec{r}_n(\tilde{s}_n-\tilde{t}_n),\quad |x|>\tilde{s}_n-\tilde{t}_n,
\end{equation*}
where the $U_n^k$ are defined in \eqref{defU01k}, \eqref{D32} and 
$$\lim_{n\to\infty} \int_{\{|x|>|\tilde{s}_n-\tilde{t}_n|\}} \left|\nabla_{t,x}r_n(\tilde{s}_n-\tilde{t}_n)\right|^2dx=0.$$
Let 
\begin{gather}
 \label{D130}
 \left( B_0^{\tJ},B_1^{\tJ} \right)=\underset{n\to \infty}{\wlim}  \left( \lambda_{\tJ,n}^{\frac{N-2}{2}}U(\tilde{s}_n,\lambda_{\tJ,n}\cdot),\lambda_{\tJ,n}^{\frac{N}{2}} \partial_t U(\tilde{s}_n,\lambda_{\tJ,n}\cdot )\right)\\
 \label{D132}
 \sigma=\lim_{n\to\infty} \frac{\tilde{s}_n-\tilde{t}_n}{\lambda_{\tJ,n}}.
\end{gather} 
Note that $\sigma\in [0,\infty)$ since $\frac{t_n'-\tilde{t}_n}{\lambda_{\tJ,n}}=T$.
As in the proof of Lemma \ref{L:D20} (see \eqref{R203}) we obtain, 
\begin{equation} 
 \label{D131}
\left( B_0^{\tJ}(x),B_1^{\tJ}(x) \right)=\vec{U}^{\tJ}(\sigma,x),\quad |x|>\sigma.
 \end{equation} 
 We will next use Lemma \ref{L:expansion} along the sequence $\{\tilde{s}_n\}$. For this, we recall (see \eqref{R224}) 
 \begin{equation}
  \label{D140}
  \forall j\in \llbracket 1,\tJ-1\rrbracket,\quad \lim_{n\to\infty} \frac{\lambda_j(\tilde{s}_n)}{\lambda_j(\tilde{t}_n)}=1.
 \end{equation} 
 On the other hand, after extraction,
 \begin{equation}
  \label{D141}
  \lim_{n\to\infty} \frac{\lambda_{\tJ}(\tilde{t}_n)}{\lambda_{\tJ}(\tilde{s}_n)}=:\tlambda\in (0,\infty).
 \end{equation} 
 Indeed, by the estimates \eqref{EDO2a}, \eqref{EDO3a}, which imply $|\lambda'|\leq C\gamma^{\frac{N-2}{4}}$, we have
 $$ \left|\lambda_{\tJ}\left( \tilde{t}_n \right)-\lambda_{\tJ}\left( \tilde{s}_n \right)\right|\lesssim |\tilde{s}_n-\tilde{t}_n| \sup_{t\in [\tilde{t}_n,\tilde{s}_n]} \gamma^{\frac{N-2}{4}}(t),$$
 and thus for large $n$ (using \eqref{R15})
 $$\left| 1-\frac{\lambda_{\tJ}(\tilde{s}_n)}{\lambda_{\tJ}(\tilde{t}_n)}\right| \leq C\frac{|\tilde{s}_n-\tilde{t}_n|}{\lambda_{\tJ}\left( \tilde{t}_n \right)} \left(\eps_0^{\frac{N-2}{4}}+o_n(1)\right)\leq 2C T \eps_0^{\frac{N-2}{4}}.$$
 Taking $\eps_0$ small enough, so that $T\eps_0^{\frac{N-2}{4}}\leq 1/(4C)$, we obtain $\frac 12\leq \frac{\lambda_{\tJ}(\tilde{s}_n)}{\lambda_{\tJ}(\tilde{t}_n)}\leq \frac{3}{2}$, which yields \eqref{D141}.
 
 By Lemma \ref{L:expansion}, letting 
 \begin{equation}
\label{D142}
  \left( V_0^{\tJ},V_1^{\tJ} \right)=\underset{n\to \infty}{\wlim}  \left( \lambda_{\tJ}^{\frac{N-2}{2}}(\tilde{s}_n)U\left( \tilde{s}_n,\lambda_{\tJ}(\tilde{s}_n)\cdot \right),\lambda_{\tJ}^{\frac N2}(\tilde{s}_n)\partial_tU\left(\tilde{s}_n,\lambda_{\tJ}(\tilde{s}_n)\cdot\right) \right),
 \end{equation} 
 we have 
 \begin{equation}
  \label{D151}
  V_0^{\tJ}=\sum_{j=\tJ}^{j_{\tJ+1}-1} \iota_j W_{(\nu_j)}+\tilde{h}_0^{\tJ},\quad V_1^{\tJ}=\sum_{j=\tJ}^{j_{\tJ+1}-1} \tilde{\alpha}_j(\Lambda W)_{[\nu_j]}+\tilde{g}_1^{\tJ},
 \end{equation} 
 where 
\begin{equation}
\label{defnujbis} 
\nu_j=\lim_{n\to\infty} \frac{\lambda_j(\tilde{s}_n)}{\lambda_{\tJ}(\tilde{s}_n)}>0,\quad j\in \llbracket \tJ,j_{\tJ+1}-1\rrbracket 
\end{equation} 
and $j_{\tJ+1}$ is the first index $j>\tJ$ such that
\begin{equation}
 \label{D150}
 \lim_{n\to\infty} \frac{\lambda_{j}(\tilde{s}_n) }{\lambda_{\tJ}(\tilde{s}_n)}=0
\end{equation}
(as usual, if \eqref{D150} is not satisfied for any $j>\tJ$, we take $j_{\tJ+1}=J+1$), and $\tilde{h}_0^{\tJ}$, $\tilde{g}_1^{\tJ}$ satisfy the orthogonality conditions 
\begin{equation}
 \label{D152}
  \int \nabla \tilde{h}_0^k \nabla (\Lambda W)_{(\tilde{\nu}_j)}=\int \tilde{g}_1^k(\Lambda W)_{[\tilde{\nu}_j]}=0,\quad j\in \llbracket \tj_{k},\tj_{k+1}-1\rrbracket,
\end{equation} 
By \eqref{D141}, \eqref{D142}, and the definition \eqref{D130} of $\left( B_0^{\tJ},B_1^{\tJ} \right)$, we see that 
$$ B_0^{\tJ}= \tlambda^{\frac{N-2}{2}} V_0^{\tJ}\big( \tlambda\cdot \big),  \quad B_1^{\tJ}=\tlambda^{\frac N2}V_1^{\tJ}\big( \tlambda \cdot \big). $$
Using \eqref{R170} and \eqref{D131}, we deduce that for large $R$
$$\left\| \left( \tlambda^{\frac{N-2}{2}} V_0^{\tJ}\big(\tlambda \cdot\big),\tlambda^{\frac N2} V_1^{\tJ}\big(\tlambda\cdot\big) \right) -\ell \Xi_{p_0}\right\|_{\HHH(R)} \lesssim \max \left\{\frac{1}{R^{(p_0-\frac 12)^{\frac{N+2}{N-2}}}},\frac{1}{R^{p_0+\frac 12}}\right\}$$
and after rescaling, for large $R$,
\begin{equation}
\label{D153}
\left\| \left( V_0^{\tJ},V_1^{\tJ} \right)-\tlambda^{p_0-\frac 12} \ell \Xi_{p_0}\right\|_{\HHH(R)} \lesssim \max \left\{\frac{1}{R^{(p_0-\frac 12)^{\frac{N+2}{N-2}}}},\frac{1}{R^{p_0+\frac 12}}\right\}.
\end{equation} 
We next note that $(V_0^{\tJ},V_1^{\tJ})$ satisfies the assumptions of Subsections \ref{SS:close_multi} and \ref{SS:lwrbnd}. Indeed, it is close to the multi-soliton
$$ \widetilde{M}=\sum_{j=\tJ}^{j_{\tJ+1}-1}\iota_jW_{\nu_j},$$
(see \eqref{R41})
and the solution with initial data $(V_0^{\tJ},V_1^{\tJ})$  is, by Lemma \ref{L:expansion}, non-radiative. We use Proposition \ref{P:ext_scaling} on $\left( V_0^{\tJ},V_1^{\tJ} \right)$. In view of \eqref{D151} and the orthogonality conditions \eqref{D152} satisfied by $\tilde{h}_0^{\tJ}$ and $\tilde{h}_1^{\tJ}$, the exterior scaling parameter (denoted by $\lambda_1$ in Proposition \ref{P:ext_scaling}) is $\nu_{\tJ}$ who is, by definition, equal to $1$ (see \eqref{defnujbis}). By \eqref{D153}, we must replace, in the conclusion of Proposition \ref{P:ext_scaling}, $|\ell|$ by $\tlambda^{p_0-\frac 12}|\ell|$. This yields
\begin{equation}
 \label{D160}
 \tlambda^{p_0-\frac 12} |\ell| \leq C\delta(\tilde{s}_n)^{\frac{2}{N}}
\end{equation} 
for some constant $C$. Using \eqref{D141}, we obtain
$$ \left( \frac{\lambda_{\tJ}(\tilde{t}_n)}{\lambda_{\tJ}(\tilde{s}_n)} \right)^{p_0-\frac 12}|\ell|\leq 2C \delta(\tilde{s}_n)^{\frac 2N}$$
for large $n$. By \eqref{D121},
$$ 2C \delta(\tilde{s}_n)^{\frac 2N}\geq M \delta(\tilde{s}_n)^{\frac{2}{N}}$$
for large $n$. Choosing $M=3C$ we obtain that $\delta(\tilde{s}_n)=0$ for large $n$, contradicting \eqref{D160}. 
\end{proof}
\subsection{Study of a system of differential inequalities}
\label{SS:rigidity}
In this subsection, we prove Proposition \ref{P:rigidity}.

\setcounter{step}{0}
We first observe that by the following change of variable of unknwown functions:
$$\tau=\frac{t}{\lambda_1(0)},\quad \check{\lambda}_j(\tau)=\frac{\lambda_j(t)}{\lambda_1(0)},\quad \check{\beta}_j(\tau)=\beta_j(t).$$
we can assume 
$$\lambda_1(0)=1.$$
The proof is by contradiction and relies on a monotonicity formula that follows from the modulation equations \eqref{EDO2},\ldots,\eqref{EDO5}.
 In all the proof, the estimates are uniform in $t\in[0,T]$, and we must follow the dependence of the constants with respect to $L$. The implicit constants implied by the symbols $\lesssim$, $\ll$, \ldots will thus never depend on $L$ and $t\in [0,T]$. We will introduce in the course of the proof two constants $m$ and $M$ depending only on $L$ and the parameters of the system. To simplify notations, we will let 
 $$\kappa_2:=\frac{1}{\|\Lambda W\|^2_{L^2}}.$$
 We note that \eqref{EDO3} and the smallness assumption \eqref{EDO1} imply 
\begin{equation}
\label{EDO6'}
\sum_{j=1}^{J_0}|\beta_j|\lesssim\gamma^{\frac{N-2}{4}}.
\end{equation}
Together with \eqref{EDO2} we obtain
\begin{equation}
 \label{EDO7}
\sum_{j=1}^{J_0} |\lambda_j'|\lesssim C\gamma^{\frac{N-2}{4}}.
\end{equation}  
The idea of the proof is to construct a function $V$ which is of the same order as $\lambda_1^2$ and is convex. We first introduce a positive quantity $B(t)$ which will appear in the computation of $V''(t)$.
\begin{step}[Introduction of a positive quantity]
  \label{St:Vmonotonic}
  We define $\theta_1$,\ldots,$\theta_{J_0}$ as follows: $\theta_1=1$ and
  $$\forall j\in \llbracket 2,{J_0}\rrbracket,\quad \theta_j=\begin{cases}
                     2\theta_{j-1}&\text{ if } \iota_j\iota_{j-1}=1\\
                     \frac{1}{2}\theta_{j-1}&\text{ if }\iota_j\iota_{j-1}=-1.
                    \end{cases}$$
  We let 
  $$B(t)=\sum_{j=1}^{J_0}\theta_j\lambda_j(t)\beta_j'(t).$$
  In this step we prove
  \begin{equation}
   \label{EDO10}
   \forall t\in [0,T],\quad B(t)\geq \frac{\kappa_0}{2^{{J_0}+1}} \gamma(t)^{\frac{N-2}{2}},
  \end{equation} 
  where $\kappa_0$ is the constant appearing in the formula \eqref{EDO4} for $\beta_j'$. Indeed by \eqref{EDO4}, denoting by $\iota_0=\iota_{{J_0}+1}=0$,
  \begin{multline*}
   B(t)=\kappa_0\sum_{j=1}^{J_0}\theta_j\left[\iota_j\iota_{j-1}\left( \frac{\lambda_j}{\lambda_{j-1}} \right)^{\frac{N-2}{2}}-\iota_j\iota_{j+1}\left( \frac{\lambda_{j+1}}{\lambda_j} \right)^{\frac{N-2}{2}}\right)+\OOO\left( \gamma^{\frac{N-1}{2}} \right)\\
  =\kappa_0\sum_{j=2}^{J_0} \iota_j\iota_{j-1}\left( \frac{\lambda_j}{\lambda_{j-1}} \right)^{\frac{N-2}{2}}\left( \theta_j-\theta_{j-1} \right)+\OOO\left( \gamma^{\frac{N-1}{2}} \right). 
  \end{multline*}
  By the definition of $\theta_j$, for $2\leq j\leq {J_0}$, we have
  $$\iota_{j-1}\iota_j\left(\theta_{j}-\theta_{j-1}\right)=
  \begin{cases}
\theta_{j-1}&\text{ if }\iota_{j-1}\iota_j=1\\
\frac{1}{2}\theta_{j-1}&\text{ if }\iota_{j-1}\iota_{j}=-1.                                                           \end{cases}.$$
Hence $\iota_{j-1}\iota_j\left(\theta_{j}-\theta_{j-1}\right)\geq \frac{1}{2^{J_0}}$, and \eqref{EDO10} follows since $\gamma$ is small.
 \end{step}
 \begin{step}[Approximate first derivative of $V$]
 \label{St:EDO2}
 We let
 \begin{equation}
  \label{EDO11}
  A(t)=\sum_{j=1}^{J_0}\theta_j\lambda_j(t)\beta_j(t).
 \end{equation} 
 In this step we prove 
 \begin{equation}
  \label{EDO13}
  A'(t)\gtrsim\gamma^{\frac{N-2}{2}}(t)
 \end{equation} 
 and that there exists $C>0$ such that
 \begin{align}
  \label{EDO12}
  A'(t)&\geq (\kappa_2^{-1}+C^{-1})\sum_{j=1}^{J_0}\theta_j{\lambda_j'}^2(t)\\
  \label{EDO14}
  A'(t)&\geq (\kappa_2+C^{-1})\sum_{j=1}^{J_0} \theta_j\beta_j^2(t).
 \end{align}
Indeed,
$$A'(t)=\sum_{j=1}^{J_0}\theta_j\lambda_j'(t)\beta_j(t)+B(t)=\kappa_2\sum_{j=1}^{J_0}\theta_j\beta_j^2(t)+B(t)+\OOO\left( \gamma^{\frac{N-1}{2}} \right),$$
where we have used \eqref{EDO2} and \eqref{EDO6'}. By \eqref{EDO10}, we deduce:
$$A'(t)\geq \kappa_2\sum_{j=1}^{J_0} \theta_j\beta_j^2(t)+\frac{1}{C}\gamma(t)^{\frac{N-2}{2}},$$
hence \eqref{EDO13}. The estimate \eqref{EDO14} follows also immediately since $\sum \beta_j^2\lesssim \gamma^{\frac{N-2}{2}}$.  Together with \eqref{EDO2}, we also obtain \eqref{EDO12}.
\end{step}
\begin{step}[Choice of an intermediate time]
\label{St:EDO3}
In this step we will show that $A(1/L)$  is bounded from below by a constant depending only on $L$.
 By \eqref{EDO5} and \eqref{EDO13}, recalling that $\lambda_1(0)=1$, we have
\begin{equation*}
 L\lesssim C\,A'(t)\lambda_1^{a}(t).
\end{equation*}
Integrating between $0$ and $\tau\leq T$, we obtain
\begin{equation}
 \label{EDO20}
L\tau\lesssim \int_0^{\tau} A'(t)\lambda_1^a(t)\,dt.
\end{equation} 
Furthermore, by integration by parts, using again that $\lambda_1(0)=1$,
\begin{align*}
 \int_0^\tau A'(t)\lambda_1^a(t)\,dt&=A(\tau)\lambda_1^a(\tau)-A(0)-a\int_0^{\tau}A(t)\lambda_1'(t) \lambda_1^{a-1}(t)\,dt\\
&=A(\tau)\lambda_1^a(\tau)-A(0)-\frac{a}{\kappa_2}\int_0^{\tau}\left( \lambda'_1\right)^2\lambda_1^a+\OOO\left( \int_0^{\tau}\left|\lambda_1'\right|\gamma^{\frac{N}{4}}\lambda_1^{a} \right).
\end{align*} 
For this, we have used that by \eqref{EDO2} and the bound $|\beta_j|\lambda_j\lesssim \gamma^{\frac{N+2}{4}}\lambda_1$ (for $j\geq 2$), which follows from \eqref{EDO6'} and the definition of $\gamma$, we have
\begin{multline}
\label{eq_At}
 A(t)=\sum_{j=1}^J\theta_j \lambda_j\beta_j=\kappa_2^{-1}\lambda_1\lambda_1'+\left( \beta_1-\kappa_2^{-1}\lambda_1' \right)\lambda_1+\sum_{j=2}^{J_0} \theta_j\lambda_j\beta_j\\
 =\kappa_2^{-1}\lambda_1\lambda_1'+\OOO\left( \lambda_1\gamma^{\frac N4} \right).
\end{multline}
Combining with the bound \eqref{EDO7} on $|\lambda_1'|$, we deduce
\begin{equation*}
 \int_0^{\tau} A'(t)\lambda_1^a(t)\,dt=A(\tau)\lambda_1^a(\tau)-A(0) -\frac{a}{\kappa_2}\int_0^\tau (\lambda_1'(t))^2\lambda_1^a(t)\,dt+\OOO\left( \int_0^{\tau}\gamma^{\frac{N-1}{2}}\lambda_1^a \right).
\end{equation*} 
Using \eqref{EDO20} and that by \eqref{EDO13},
\begin{equation*}
 \int_0^\tau A'(t)\lambda_1^a(t)\,dt\gtrsim \int_0^{\tau}\gamma^{\frac{N-2}{2}}\lambda_1^a\gtrsim \frac{1}{\sqrt{\eps}} \int_0^{\tau} \gamma^{\frac{N-1}{2}}\lambda_1^{a},
\end{equation*} 
we deduce, if $\eps$ is small enough, that for all $\tau\in (0,T]$, 
\begin{equation}
\label{EDO30'}
 \frac 12L\tau +\frac{a}{\kappa_2} \int_0^{\tau}(\lambda'_1(t))^2\lambda_1^a(t)\,dt +A(0)\leq A(\tau)\lambda_1^a(\tau).
\end{equation} 
Assume in all the sequel that $T\geq 1/L$.
Using that $|\lambda_1'|\leq 1$ (see \eqref{EDO7}), we obtain
$$ \lambda_1(1/L)\leq 1+1/L.$$
Noting that by \eqref{EDO5} at $t=0$, and since $\eps$ is small, $\frac{1}{L}\geq 1$, we deduce, taking $\eps$ small enough,
\begin{equation}
 \label{bound_lambda1}
\lambda_1(1/L)\leq \frac{2}{L}
\end{equation} 
Going back to \eqref{EDO30'}, using that $|A(0)|\leq \frac{1}{100}$, we obtain
\begin{equation}
\label{EDO32}
 A(1/L)\geq \frac{1}{2} \frac{1}{\lambda^a_1(1/L)}\geq 
\frac{L^a}{2^{a+1}}.
\end{equation}  
Note that since $A'(t)\geq 0$ by Step \ref{St:EDO2}, \eqref{EDO32} implies that $A(t)>0$ for all $t\in [1/L,T]$.
\end{step}
\begin{step}[Bound from above of $\lambda_1$]
\label{St:EDO4} 
We let 
$$V(t)=\sum_{j=1}^{J_0} \theta_j\lambda_j^2,$$
and note that, using that $\theta_1=1$ and that $\lambda_j\leq \eps \lambda_1$ for $j\geq 2$,
$$|V(t)-\lambda_1^2(t)|\lesssim \eps^2\lambda_1^2(t).$$
In this step we show that there exists $C>0$, and $M=M(L)>0$ such that 
\begin{equation}
 \label{EDO42}
\forall t\in [1/L,T], \quad V(t)\leq M\eps^{\frac 1C}.
\end{equation} 
Indeed
$V'(t)=2\sum_{j=1}^{J_0} \theta_j\lambda_j\lambda_j'$, and thus
\begin{equation*}
 V'(t)A(t)=2\sum_{j=1}^{J_0} \theta_j\lambda_j\lambda_j'\sum_{j=1}^{J_0}\theta_j\lambda_j\beta_j
\leq 2V(t)\sqrt{\sum_{j=1}^{J_0} \theta_j{\lambda_j'}^2}\sqrt{\sum_{j=1}^{J_0}\theta_j\beta_j^2}\leq 2^-V(t)A'(t)
\end{equation*}
by Step \ref{St:EDO2}. Here 
$$2^-=\frac{2}{(\kappa_2^{-1}+C^{-1})(\kappa_2+C^{-1})}$$ 
is a fixed positive constant, smaller than $2$, and independent of $L$. Recall that for $t\geq 1/L$, we have $A(t)>0$. For such $t$, we deduce
\begin{equation}
\label{EDO40'}
2^-\frac{A'(t)}{A(t)}\geq \frac{V'(t)}{V(t)}, 
\end{equation} 
and thus, for $t\geq 1/L$, 
\begin{equation}
 \label{EDO40}
\frac{d}{dt}\left( \frac{A^{2^-}(t)}{V(t)} \right)\geq 0.
\end{equation} 
Hence, by \eqref{bound_lambda1} (that implies $V(1/L)\lesssim 1/L^2$) and \eqref{EDO32}, for $t\geq 1/L$,
\begin{equation}
 \label{EDO41}
\frac{A^{2^-}(t)}{V(t)} \geq \frac{A^{2-}(1/L)}{V(1/L)}\gtrsim L^{2+2^-a}.
\end{equation} 
Using the inequality
$$A(t)\lesssim \sqrt{V(t)}\gamma^{\frac{N-2}{4}}(t),$$
we deduce
$$ \forall t\geq 1/L, \quad \left( \gamma^{\frac{N-2}{2}} \right)^{\frac{2^-}{2}}\gtrsim \frac{A(t)^{2^-}}{V(t)^{\frac{2^-}{2}}}\gtrsim V(t)^{\frac{2-2^-}{2}}
\frac{A^{2^-}(t)}{V(t)} \gtrsim 
L^{2+2^-a}
V(t)^{\frac{2-2^-}{2}}.$$
This yields \eqref{EDO42} with $M=\frac{1}{L^{\frac{2(2+2^-a)}{2-2^-}}}$.
\end{step}
\begin{step}[Bound from below of $\lambda_1$]
 \label{St:EDO5}
By \eqref{EDO40'},  we have, for $t\geq 1/L$,
$$\frac{d}{dt}\left( \frac{A(t)}{\sqrt{V(t)}} \right)\geq 0.$$
As a consequence, for $t\geq 1/L$,
$$\frac{A(t)}{\sqrt{V(t)}}\geq \frac{A(1/L)}{\sqrt{V(1/L)}}.$$
By the bound \eqref{bound_lambda1} on $\lambda_1$, the fact that $V(t)\approx \lambda_1^2(t)$ and the bound \eqref{EDO32} on $A(1/L)$, we have
\begin{equation}
 \label{EDO51}
\frac{A(t)}{\sqrt{V(t)}}\geq \frac{A(1/L)}{\sqrt{V(1/L)}}\gtrsim L^{a+1}=:m. 
\end{equation} 
Since $|V'(t)-2\kappa_2A(t)|\lesssim \gamma^{\frac{N}{4}}(t)\lambda_1(t)$ (see \eqref{eq_At}), we deduce from \eqref{EDO51}
$$\forall t\geq 1/L,\quad \frac{V'(t)}{\sqrt{V(t)}} \gtrsim m+\OOO\left( \gamma^{\frac{N}{4}}(t) \right).$$
Integrating, we obtain
$$\int_{1/L}^t \gamma^{\frac{N}{4}}(s)\,ds+\sqrt{V(t)}-\sqrt{V(1/L)}\gtrsim (t-1/L)m,$$
and thus, by the bound \eqref{EDO42} on $V\approx\lambda_1^2$,
\begin{equation}
\label{EDO52}
\forall t\in [1/L,T],\quad \sqrt{M}\eps^{\frac{1}{2C}}+\int_{1/L}^t\gamma^{\frac{N}{4}}(s)\,ds\gtrsim (t-1/L)m.
\end{equation}
Notice that, by Step \ref{St:EDO2}, for $t\in[1/L,T]$, 
\begin{multline*}
\int_{1/L}^t \gamma^{\frac{N}{4}}\leq \sqrt{\eps} \int_{1/L}^t \gamma^{\frac{N-2}{4}}(s)\,ds\lesssim \sqrt{\eps}\sqrt{t}\sqrt{\int_{1/L}^t A'(s)\,ds}\\
\lesssim \sqrt{\eps}\sqrt{t} \sqrt{A(t)}\lesssim \sqrt{\eps}\sqrt{t}\,\eps^{\frac{N-2}{8}} \sqrt{\lambda_1(t)}
\lesssim \eps^{\frac{N+2}{8}} \sqrt{t}M^{1/4},
\end{multline*}
where we have used the bound $|A(t)|\lesssim \lambda_1(t)\gamma^{\frac{N-2}{4}}$ and, to get the last inequality, the bound \eqref{EDO42} on $V$. Going back to \eqref{EDO52}, we obtain
\begin{equation}
 \label{EDO60}
m(t-1/L)\lesssim \eps^{\frac{1}{2C}}\sqrt{M}+\eps^{\frac{N+2}{8}} M^{1/4}\sqrt{t}.
\end{equation} 
Taking $\eps$ small, we deduce
\begin{equation}
 \label{EDO60'}
m(T-1/L)\leq \sqrt{M}+M^{1/4}\sqrt{T},
\end{equation} 
i.e.
$$ \left( \sqrt{T}-\frac{M^{1/4}}{2m} \right)^2\leq \frac 1L+\frac{\sqrt{M}}{m}+\frac{\sqrt{M}}{4m^2},$$
which implies 
$$ T\leq T^*:=\left(\sqrt{\frac{1}{L}+\frac{\sqrt{M}}{m}+\frac{\sqrt{M}}{4m^2}}+\frac{M^{1/4}}{2m}\right)^2,$$
concluding the proof since the constants $m$ and $M$ depend only on $L$ and the parameters of the system.
\end{step}
\section{Inelastic collision}
\label{S:inelastic}
This section is dedicated to the proof of Theorem \ref{T:inelastic}. The proof is almost contained in the proof of Theorem \ref{T:resolution} and we only sketch it. Let $u$ satisfy the assumptions of Theorem \ref{T:inelastic}. If $u$ scatters forward in time, then
$$\lim_{t\to+\infty}\|\vec{u}(t)\|_{\HHH}=0,$$
and by the small data theory, $u$ is identically $0$. Thus $u$ does not scatter as $t\to\infty$, and according to Theorem \ref{T:resolution}, there exists $J\geq 1$, signs $\{\iota_j\}_{1\leq j\leq J}$, parameters $\lambda_{j}(t)$ defined for large $t$ and such that 
$$ 0<\lambda_{J}(t)< \lambda_{J-1}(t)<\ldots< \lambda_{1}(t),\quad \lim_{t\to\infty} \gamma(t)=0,\quad \lim_{t\to\infty}\lambda_1(t)/t=0,$$
(where $\gamma(t)=\gamma(\lambdabf(t))=\max_{j\in \llbracket 1,J-1\rrbracket} \lambda_{j+1}(t)/\lambda_{j}(t)$)
and
$$ \lim_{t\to\infty} \left\|\vec{u}(t)-\sum_{j=1}^{J} \left(\iota_j W_{(\lambda_j(t))},0\right)\right\|_{\HHH}=0.$$
Note that \eqref{no_radiation} implies that $v_L\equiv 0$.
We will use the notations of Section \ref{S:reduction}. Using Lemma \ref{L:ortho_scaling} in the appendix, we can choose the $\lambda_j(t)$, for large $t$ such that 
\begin{equation}
 \label{R17bis}
 \forall j\in \llbracket 1,J\rrbracket, \quad \int \nabla (u(t)-M(t))\cdot\nabla \left(\Lambda W_{(\lambda_j(t))}\right)=0,
 \end{equation}
 where $M(t):=\sum_{j=1}^J \iota_j W_{(\lambda_j(t))}$. For $t\geq T$, $T$ large, we expand 
$$ u(t)=M(t)+h(t),\quad \partial_tu(t)=\sum_{j=1}^J \alpha_j(t)\iota_j (\Lambda W)_{[\lambda_j(t)]}+g_1(t),$$
and denote
$$ \delta(t):=\sqrt{\|h(t)\|^2_{\hdot}+\|\partial_tu(t)\|^2_{L^2}},\quad \beta_j(t) :=-\iota_j \int (\Lambda W)_{[\lambda_j(t)]} \partial_tu(t).$$
Observe that the expansions above are valid for all large times, as opposed to the analoguous expansions in Section \ref{S:reduction} that are made on intervals of the form $[\tilde{t}_n,t_n]$, where $\tilde{t}_n$ and $t_n\to \infty$ as $n\to\infty$.

Then we have the following variant of Proposition \ref{P:EDO}:
\begin{prop}
\label{P:EDObis}
There exists $C>0$ such that for all $t\geq T$,
\begin{align}
\label{bound_delta_gamma_bis}
\delta&\leq C\gamma^{\frac{N-2}{4}}
\\
\label{EDO2a_bis}
 \forall j\in \llbracket 1,J\rrbracket,\quad
 \left|\beta_j-\|\Lambda W\|^2_{L^2}\lambda'_j\right|&\leq C \gamma^{\frac N4}\\
 \label{EDO3a_bis}
 \left|\frac{1}{2}\sum_{j=1}^{J}\beta_j^2-\kappa_1\sum_{1\leq j\leq {J}-1} \iota_j\iota_{j+1}\left( \frac{\lambda_{j+1}}{\lambda_j} \right)^{\frac{N-2}{2}}\right|&\leq C\gamma^{\frac{N-1}{2}}\\
 \label{EDO4a_bis}
 \forall j\in \llbracket 1,J\rrbracket,\quad \Bigg|\lambda_j\beta'_j+\kappa_0\Bigg( \iota_j\iota_{j+1} \bigg( \frac{\lambda_{j+1}}{\lambda_j}\bigg)^{\frac{N-2}{2}}-\iota_j&\iota_{j-1}\bigg( \frac{\lambda_j}{\lambda_{j-1}} \bigg)^{\frac{N-2}{2}} \Bigg) \Bigg|&\\
 \notag
 &\leq C\gamma^{\frac{N-1}{2}},
 \end{align}
where the  positive constants $\kappa_0$, $\kappa_1$ are as in Proposition \ref{P:EDO}.
\end{prop}
\begin{proof}[Sketch of proof]
 The proof is the same as the proof of Proposition \ref{P:EDO} in Subsections \ref{SS:lambda_beta} and \ref{SS:derivatives}, observing that in the context of Proposition \ref{P:EDObis}, we do not need Subsection \ref{SS:expansion} and we can remove all the $o_n(1)$ in the estimates. More precisely:
\begin{itemize}
 \item The estimates of Lemma \ref{L:R26} hold for $t\geq T$, without the $o_n(1)$ terms, as a direct consequence of Proposition \ref{P:F18}.
\item The estimates of Lemma \ref{St:energy} can be proved for all $t\geq T$, without the $o_n(1)$ terms, expanding the equality:
$$ E(\vec{u}(t))=JE(W,0),$$
and with the same proof. In Lemma \ref{St:energy}, the terms $o_n(1)$ came from the fact that the preceding equality was replaced by the weaker statement:
$$ \lim_{t\to\infty} E(\vec{u}(t)-\vec{v}_L(t))=JE(W,0).$$
\item Similarly, one can prove Lemma \ref{L:derivative} and Lemma \ref{L:second_derivative} for $t\geq T$ without $o_n(1)$ with the same proofs, observing that in the proofs of these Lemmas, the $o_n(1)$ terms come either from the $o_n(1)$ terms of Subsection \ref{SS:lambda_beta}, or from the term $\sigma(h,v_L)$ defined in \eqref{R150} and which is $0$ in our setting.
\end{itemize}
Assuming that $u$ is not stationary, it is now easy to obtain a contradiction: by Proposition \ref{P:ext_scaling},
$$ \forall t\geq T,\quad |\ell|\leq C\delta(t)^{\frac{2}{N}}\lambda_1^{k_0-\frac{1}{2}}.$$
Combining with the estimates of Proposition \ref{P:EDObis}, we see that this contradicts Proposition \ref{P:rigidity}, concluding the proof.
\end{proof}

\appendix
\section{Proof of some estimates}
\label{A:calculs}
In this appendix we gather a few purely computational proofs.
\subsection{Estimates on integrals in the space variable}
\begin{claim}
 \label{Cl:estimates1}
Let $0<\lambda<\mu$. Assume $N\geq 5$. Then
\begin{gather}
\label{est1.1}\int_{\RR^N} \left|\nabla (\Lambda W_{(\lambda))}\cdot\nabla (\Lambda W_{(\mu)})\right|+
 \int_{\RR^N} \left|\nabla W_{(\lambda)}\cdot\nabla W_{(\mu)}\right|\lesssim \left( \frac{\lambda}{\mu} \right)^{\frac N2-1},\\
 \label{est1.2}
 \int_{\RR^N}  \left|(\Lambda W)_{[\lambda]}(\Lambda 
 W)_{[\mu]}\right|+\int_{\RR^N}  \left|(\Lambda W)_{[\lambda]}(\Lambda_0\Lambda W)_{[\mu]}\right|\lesssim \left( \frac{\lambda}{\mu} \right)^{\frac N2-2},\\
 \label{est1.3}\left\| W_{(\lambda)}W_{(\mu)}^{\frac{4}{N-2}}\right\|_{L^{\frac{2N}{N+2}}}\lesssim \left( \frac{\lambda}{\mu} \right)^{\frac{N-2}{2}},\quad \left\|  W_{(\mu)}W_{(\lambda)}^{\frac{4}{N-2}}\right\|_{L^{\frac{2N}{N+2}}}\lesssim \left( \frac{\lambda}{\mu} \right)^{2}\\
 \label{est1.4}
  \left|\int_{\RR^N} W_{(\lambda)}^{\frac{N}{N-2}}W_{(\mu)}^{\frac{N}{N-2}}\right|\lesssim \left( \frac{\lambda}{\mu} \right)^{\frac N2}\\
  \label{est1.5}
  \int \left|(\Lambda W)_{[\lambda]}(\Delta \Lambda W)_{[\mu]}\right|\lesssim \left(\frac{\lambda}{\mu}\right)^{\frac{N}{2}-2},\quad \int \left|(\Lambda W)_{[\mu]}(\Delta \Lambda W)_{[\lambda]}\right|\lesssim \left(\frac{\lambda}{\mu}\right)^{\frac{N}{2}}
\end{gather} 
\end{claim}
\begin{proof}
We have
\begin{gather}
\label{bound_W1}
 |\Lambda_0\Lambda W(x)|+|\Lambda W(x)|+|W(x)|\lesssim \min\left\{1,\frac{1}{|x|^{N-2}}\right\}\\
 \label{bound_nabla_W1}
 |\nabla \Lambda W(x)|+|\nabla W(x)|\lesssim \min\left\{1,\frac{1}{|x|^{N-1}}\right\},\quad |\Delta \Lambda W|\lesssim \min\left\{1,\frac{1}{|x|^{N}}\right\}.
\end{gather} 
 In view of these bounds, the estimates \eqref{est1.1}, \eqref{est1.2}, \eqref{est1.3}, \eqref{est1.4} and \eqref{est1.5} are consequences of the following inequality, which holds for any $a,b\in \RR$ with $a+b>N$, and can be proved by integrating separately on $\{|x|<\lambda\}$, $\{\lambda<|x|<\mu\}$ and $\{|x|>\mu\}$:
 \begin{equation}
  \label{crucial_est}
  \int_{\RR^N} \min\left( 1,\left( \frac{\lambda}{|x|} \right)^a \right)\min\left( 1,\left( \frac{\mu}{|x|} \right)^b \right)\,dx\lesssim \lambda^a\mu^{N-a}.
 \end{equation} 
 We will prove \eqref{est1.3}. The proofs of \eqref{est1.1}, \eqref{est1.2}, \eqref{est1.4} and \eqref{est1.5} are very similar.
By \eqref{bound_W1}, we have
\begin{multline*}
\int W_{(\lambda)}^{\frac{2N}{N+2}}W_{(\mu)}^{\frac{8N}{(N-2)(N+2)}}\\ \lesssim\frac{1}{\lambda^{\frac{N(N-2)}{N+2}}\mu^{\frac{4N}{N+2}}}\int \min\left\{ 1,\left( \frac{\lambda}{|x|} \right)^{\frac{2N(N-2)}{N+2}}\right\}\min\left\{ 1,\left( \frac{\mu}{|x|} \right)^{\frac{8N}{N+2}}\right\}, 
\end{multline*}
and the first estimate of \eqref{est1.3} follows from \eqref{crucial_est} with $a=\frac{2N(N-2)}{N+2}$. Similarly
\begin{multline*}
\int W_{(\mu)}^{\frac{2N}{N+2}}W_{(\lambda)}^{\frac{8N}{(N-2)(N+2)}}\\ \lesssim\frac{1}{\mu^{\frac{N(N-2)}{N+2}}\lambda^{\frac{4N}{N+2}}}\int \min\left\{ 1,\left( \frac{\mu}{|x|} \right)^{\frac{2N(N-2)}{N+2}}\right\}\min\left\{ 1,\left( \frac{\lambda}{|x|} \right)^{\frac{8N}{N+2}}\right\}, 
\end{multline*}
and the second estimate of \eqref{est1.3} follows from \eqref{crucial_est} with $a=\frac{8N}{N+2}$.
\end{proof}

\subsection{Estimates on space time norms}
\begin{claim}
\label{C:G30}
 Assume $N\geq 5$. Let $0<\lambda<\mu$. Then 
 \begin{align}
  \label{G31} 
  \left\|\indxt W_{(\lambda)}^{\frac{4}{N-2}}W_{(\mu)} \right\|_{L^1(\RR,L^2)} &\lesssim 
  \begin{cases}
\left( \frac{\lambda}{\mu} \right)^{\frac 32}& \text{ if }N=5\\
\left( \frac{\lambda}{\mu} \right)^{2}& \text{ if }N\geq 7
\end{cases}\\
\label{G31'}
  \left\|\indxt W_{(\mu)}^{\frac{4}{N-2}}W_{(\lambda)} \right\|_{L^1(\RR,L^2)} &\lesssim \begin{cases} \left(\frac{\lambda}{\mu} \right)^{\frac{3}{2}}\text{ if }N=5\\                                                                                                            
  \left(\frac{\lambda}{\mu} \right)^{2}\text{ if }N\geq 7.
  \end{cases}
\\
  \label{G32}
\left\|\indxt  t W_{(\lambda)}^{\frac{4}{N-2}}W_{[\mu]} \right\|_{L^1(\RR,L^2)}& \lesssim \left(\frac{\lambda}{\mu}\right)^2\\
\label{G32'}
\left\|\indxt  t W_{(\mu)}^{\frac{4}{N-2}}W_{[\lambda]} \right\|_{L^1(\RR,L^2)}& \lesssim 
\begin{cases}
\left( \frac{\lambda}{\mu} \right)^{\frac 12}& \text{ if }N=5\\
\left( \frac{\lambda}{\mu} \right)^{\frac{3}{2}}& \text{ if }N=7\\
\left( \frac{\lambda}{\mu} \right)^{2}& \text{ if }N\geq 9.
\end{cases}
\end{align}
The same inequalities remain valid when replacing $W$ by $\Lambda W$ anywhere in the preceding norms.
\end{claim}
\begin{proof}
\setcounter{step}{0}
In all of the proof of the Claim, we will use the bound
\begin{equation}
 \label{G33}
 |W(x)|+|\Lambda W(x)|\lesssim \min(1,|x|^{2-N}).
\end{equation}
By scaling, we can assume $\mu=1$. By symmetry, it is sufficient to bound the integrals for $t\geq 0$. The proofs for all $4$ bounds are the same. 
We divide the domain of integration for $r$ in three parts, $(0,\lambda)$, $(\lambda,1)$ and $(1,\infty)$, writing 
\begin{multline*}
\int_0^{+\infty} \left(\int_{t}^{+\infty}\ldots r^{N-1}\,dr\right)^{\frac 12}dt
\lesssim \int_0^{\lambda} \left( \int_{t}^{\lambda} \ldots r^{N-1}\,dr \right)^{\frac 12}\,dt\\
+\int_0^{1} \left( \int_{\max\{t,\lambda\}}^{1} \ldots r^{N-1}\,dr \right)^{\frac 12}\,dt+\int_0^{+\infty} \left( \int_{\max\{t,1\}}^{+\infty} \ldots r^{N-1}\,dr \right)^{\frac 12}\,dt\\
=(1)+(2)+(3),
\end{multline*}
where $\ldots$ is either  $\left(W_{(\lambda)}^{\frac{4}{N-2}}W\right)^2$, $\left(W^{\frac{4}{N-2}}W_{(\lambda)}\right)^2$, $\left(tW_{(\lambda)}^{\frac{4}{N-2}}W\right)^2$ or $\left(tW^{\frac{4}{N-2}}W_{[\lambda]}\right)^2$. In the integrals $(1)$, we use the bound $W(r/\lambda)+W(r)\lesssim 1$,  in the integrals $(2)$, we use the bounds $W(r/\lambda)\lesssim \left( \frac{\lambda}{r} \right)^{N-2}$ and $W(r)\lesssim 1$, and in the integrals $(3)$, we use the bounds $W(r/\lambda)\lesssim \left( \frac{\lambda}{r} \right)^{N-2}$, $W(r)\lesssim \frac{1}{r^{N-2}}$. 

We will detail the proof of \eqref{G31}
 and sketch the proof of the other estimates.
 
\medskip 
 
\noindent\emph{Proof of \eqref{G31}}.
 We have
 \begin{multline*}
(1)=\int_0^{\lambda} \left( \int_{t}^{\lambda} W^2\frac{1}{\lambda^4}W^{\frac{8}{N-2}}\left( \frac{r}{\lambda} \right)r^{N-1}\,dr \right)^{\frac 12}\,dt\\
\lesssim \frac{1}{\lambda^2}\int_0^{\lambda} \left( \int_{t}^{\lambda} r^{N-1}\,dr \right)^{\frac 12}\,dt \lesssim \lambda^{\frac{N-2}{2}}. 
 \end{multline*}
 If $N\in \{5,7\}$, we have
 \begin{multline*}
(2)= \int_0^{1} \left( \int_{\max\{t,\lambda\}}^{1} W^2\frac{1}{\lambda^4}W^{\frac{8}{N-2}}\left( \frac{r}{\lambda} \right) r^{N-1}\,dr \right)^{\frac 12}\,dt\\
\lesssim \lambda^2\int_0^{\lambda}\left( \int_{\lambda}^1 r^{N-9}\,dr \right)^{\frac{1}{2}}dt+\lambda^2\int_{\lambda}^{1}\left( \int_{t}^1 r^{N-9}\,dr \right)^{\frac{1}{2}}dt\\
\lesssim \lambda^2\int_0^{\lambda}\lambda^{\frac{N-8}{2}}\,dt +\lambda^2\int_{\lambda}^1 t^{\frac{N-8}{2}}\,dt
\lesssim \begin{cases}\lambda^2 &\text{ if }N=7\\ \lambda^{\frac{3}{2}}&\text{ if }N=5\end{cases}
 \end{multline*}
 If $N\geq 9$, we obtain
 $$(2)\lesssim \lambda^2\int_0^{\lambda}\left( \int_{\lambda}^1 r^{N-9}\,dr \right)^{\frac{1}{2}}dt+\lambda^2\int_{\lambda}^{1}\left( \int_{t}^1 r^{N-9}\,dr \right)^{\frac{1}{2}}dt\lesssim \lambda^2.$$
It remains to bound the third integral:
\begin{multline*}
 (3)=\int_0^{+\infty}\left( \int_{\max\{1,|t|\}}^{+\infty}W_{(\lambda)}^{\frac{8}{N-2}}W^2r^{N-1}dr \right)^2dt\\
 \lesssim \int_0^{+\infty}\left( \int_{\max\{1,|t|\}}^{+\infty}\frac{\lambda^4}{r^{8+2(N-2)}}r^{N-1}dr \right)^2dt\lesssim \lambda^2.
\end{multline*}
Combining the preceding bounds, we obtain \eqref{G31}.

\medskip

\noindent\emph{Sketch of proof of \eqref{G31'}}.
By analoguous arguments, we obtain:
\begin{itemize}
 \item For \eqref{G31'},
 $$(1)\lesssim \lambda^2,\quad (2)\lesssim \begin{cases}\lambda^{\frac 32}&\text{ if }N=5\\ \lambda^{2}&\text{ if }N\geq 7 \end{cases},\quad (3)\lesssim \lambda^{\frac{N-2}{2}};$$
\item For \eqref{G32},
$$(1)\lesssim \lambda^{\frac{N}{2}},\quad (2)\lesssim \lambda^2,\quad (3)\lesssim \lambda^2;$$
\item For \eqref{G32'},
$$(1)\lesssim \lambda^2,\quad (2)\lesssim \begin{cases}
\lambda^{\frac 12}&\text{ if }N=5\\
\lambda^{\frac 32}&\text{ if }N=7\\
\lambda^2&\text{ if }N\geq 9
\end{cases},
\quad (3)\lesssim \lambda^{\frac{N-4}{2}}.
$$
\end{itemize}
\end{proof}
\begin{claim}
\label{C:G30'}
 Assume $N\geq 5$. Let $0<\lambda<\mu$. Then 
 \begin{equation}
  \label{G31''} 
  \left\|\indxt \min\left\{W_{(\lambda)}^{\frac{4}{N-2}}W_{(\mu)},W_{(\mu)}^{\frac{4}{N-2}}W_{(\lambda)}\right\} \right\|_{L^1(\RR,L^2)} \lesssim \left( \frac{\lambda}{\mu} \right)^{\frac{N+2}{4}}.
 \end{equation}
\end{claim}
\begin{proof}
As before, we will use continuously the bound $|W(x)|\lesssim \min(1,|x|^{2-N})$. By scaling, we can assume
$\mu=1$ (and thus $\lambda\leq 1$). We note that 
$$\sqrt{\lambda}\lesssim r\Longrightarrow W_{(\lambda)}\lesssim W,\quad r\lesssim \sqrt{\lambda}\Longrightarrow W_{(\lambda)}\gtrsim W.$$
We divide the space into $4$ regions, writing:
\begin{multline*}
\frac{1}{2}\left\|\indxt \min\left\{W_{(\lambda)}^{\frac{4}{N-2}}W,W^{\frac{4}{N-2}}W_{(\lambda)}\right\} \right\|_{L^1(\RR,L^2)}\\
=\int_0^{+\infty}\left(\int_{t}^{+\infty}\min\left\{W_{(\lambda)}^{\frac{4}{N-2}}W,W^{\frac{4}{N-2}}W_{(\lambda)}\right\}r^{N-1}\,dr   \right)^{\frac 12}\,dt\\
\lesssim \int_0^{\lambda}\left( \int_{t}^{\lambda}\ldots \right)^{\frac 12}\,dt+\int_0^{\sqrt{\lambda}}\left( \int_{\max\{t,\lambda\}}^{\sqrt{\lambda}}\ldots \right)^{\frac 12}\,dt\\
+\int_0^{1}\left( \int_{\max\{t,\sqrt{\lambda}\}}^{1}\ldots \right)^{\frac 12}\,dt+\int_0^{+\infty}\left( \int_{\max\{t,1\}}^{+\infty}\ldots \right)^{\frac 12}\,dt\\
=A_1+A_2+A_3+A_4.
\end{multline*}
\setcounter{cas}{0}
\begin{cas}[$N\geq 7$]
 In this case, $\frac{4}{N-2}<1$. 
 We have 
 $$A_1\lesssim \int_0^\lambda \left( \int_{|t|}^{\lambda}W^2\frac{1}{\lambda^4}W^{\frac{8}{N-2}}\left( \frac{r}{\lambda} \right)r^{N-1}\,dr \right)^{\frac 12}\,dt.$$
 Using that $W$ is bounded, we obtain
 \begin{equation}
  \label{bnd_A1}
  A_1\lesssim \frac{1}{\lambda^2} \int_0^{\lambda} \left( \int_0^{\lambda}r^{N-1}\,dr \right)^{\frac 12}\,dt\lesssim \lambda^{\frac{N-2}{2}}.
 \end{equation} 
 We next consider $A_2$:
\begin{equation*}
  A_2\lesssim \int_0^{\sqrt{\lambda}}\left( \int_{\max\{t,\lambda\}}^{\sqrt{\lambda}}W^2\frac{1}{\lambda^4}W^{\frac{8}{N-2}}\left( \frac{r}{\lambda} \right)r^{N-1}\,dr\right)^{\frac 12}\,dt.
 \end{equation*}
Using the bounds $W^2\lesssim 1$ and $W^{\frac{8}{N-2}}\left( \frac{r}{\lambda} \right)\lesssim \frac{\lambda^8}{r^8}$, we obtain
\begin{equation*}
 A_2\lesssim \lambda^2\int_0^{\sqrt{\lambda}} \left( \int_{\max\{t,\lambda\}}^{\sqrt{\lambda}}r^{N-9}\,dr \right)^{\frac 12}\,dt.
\end{equation*}
If $N=7$, this yields
$$A_2\lesssim \lambda^2\int_{0}^{\lambda}\frac{1}{\lambda^{\frac 12}}\,dt+\lambda^2\int_{\lambda}^{\sqrt{\lambda}}\frac{1}{\sqrt{t}}\,dt\lesssim \lambda^{\frac{9}{4}}.$$
If $N\geq 9$, we deduce:
$$
A_2\lesssim \lambda^2\int_0^{\sqrt{\lambda}} \lambda^{\frac{N-8}{4}}dt=\lambda^{\frac{N+2}{4}}.$$
In both cases, we have obtained:
\begin{equation}
 \label{bnd_A2}
 A_2\lesssim \lambda^{\frac{N+2}{4}}.
\end{equation} 
We have:
$$A_3\lesssim \frac{1}{\lambda^{\frac{N-2}{2}}}\int_0^{1}\left( \int_{\max\{t,\sqrt{\lambda}\}}^{1}W^2\left( \frac{r}{\lambda} \right)W^{\frac{8}{N-2}}(r)r^{N-1}\,dr\right)^{\frac 12}\,dt.$$ 
 \end{cas}
Using the bounds $W^2\left( \frac{r}{\lambda} \right)\lesssim \left(\frac{\lambda}{r}\right)^{2(N-2)}$, $W^{\frac{8}{N-2}}\lesssim 1$, we deduce:
$$A_3\lesssim \lambda^{\frac{N-2}{2}}\int_0^{\sqrt{\lambda}} \left( \int_{\sqrt{\lambda}}^1r^{3-N}\,dr \right)^{\frac 12}\,dt+\lambda^{\frac{N-2}{2}}\int_{\sqrt{\lambda}}^1 \left( \int_{t}^1r^{3-N}\,dr \right)^{\frac 12}\,dt,$$
which yields
\begin{equation}
 \label{bnd_A3}
 A_3\lesssim \lambda^{\frac{N+2}{4}}.
\end{equation} 
Finally, we bound $A_4$. We have
$$A_4\lesssim \frac{1}{\lambda^{\frac{N-2}{2}}}\int_0^{\infty} \left( \int_{\max\{1,t\}}^{+\infty}W^2\left( \frac{r}{\lambda} \right)W^{\frac{8}{N-2}}\left( r \right)r^{N-1}dr \right)^{\frac 12}dt.$$
Using the bounds $W^2\left( \frac{r}{\lambda} \right)\lesssim \left( \frac{\lambda}{r} \right)^{2(N-2)}$ and $W^{\frac{8}{N-2}}\lesssim \frac{1}{r^{8}}$, we deduce
\begin{equation}
\label{bnd_A4}
A_4\lesssim \lambda^{\frac{N-2}{2}}.
\end{equation} 
Combining \eqref{bnd_A1}, \eqref{bnd_A2}, \eqref{bnd_A3} and \eqref{bnd_A4}, and noting that $\frac{N+2}{4}\leq \frac{N-2}{2}$ if $N\geq 6$, we deduce the bound \eqref{G31''} when $N\geq 7$. 
\begin{cas}[$N=5$]
If $N=5$, we have $\frac{4}{N-2}=\frac{4}{3}>1$. The proof is the same as in the preceding case, except that 
$$ \min\left\{W^{\frac{4}{3}}_{(\lambda)}W,W^{\frac{4}{3}}W_{(\lambda)}\right\}(r)\approx \begin{cases} 
W^{\frac{4}{3}}(r)W_{(\lambda)}(r)&\text{ if }r\leq \sqrt{\lambda}\\
W^{\frac{4}{3}}_{(\lambda)}(r)W(r)&\text{ if }r\geq \sqrt{\lambda}.
\end{cases}
$$
By explicit computation, one obtains the bounds 
$$A_1\lesssim \lambda^2,\quad A_2\lesssim \lambda^{\frac{7}{4}},\quad A_3\lesssim \lambda^{\frac{7}{4}},\quad A_4\lesssim \lambda^2,$$
which yields the bound \eqref{G31''}. We leave the details for the reader.
 \end{cas}

\end{proof}

\subsection{Pointwise bounds}
\label{A:pointwise}
\begin{claim}
 \label{Cl:pointwise1}
 Assume $N\geq 5$, $J\geq 1$.
 For all $(y_1,\ldots,y_J,h)\in \RR^{J+1}$,
\begin{multline}
 \label{L40}
 \Bigg| 
 \frac{N-2}{2N}\bigg|\sum_{j=1}^J y_j+h\bigg|^{\frac{2N}{N-2}}-\frac{N-2}{2N}\sum_{j=1}^J |y_j|^{\frac{2N}{N-2}}\\
 -\sum_{j=1}^J |y_j|^{\frac{4}{N-2}} y_jh
-\sum_{\substack{1\leq j,k\leq J\\ j\neq k}} |y_j|^{\frac{4}{N-2}} y_jy_k
 \Bigg| \\
 \lesssim |h|^{\frac{2N}{N-2}}+\sum_{j=1}^J |y_j|^{\frac{4}{N-2}} h^2\\
 +\sum_{1\leq j<k\leq J}\left( \min\left\{ |y_j|^{\frac{4}{N-2}} y_k^2,|y_k|^{\frac{4}{N-2}} y_j^2\right\}+\min\left\{ |y_j|^{\frac{N+2}{N-2}} |y_k|,|y_k|^{\frac{N+2}{N-2}} |y_j|\right\} \right),
\end{multline}
 \end{claim}
\begin{proof}
 We fix $(y_1,\ldots,y_J,h)\in \RR^{J+1}$ and distinguish between two cases.
\setcounter{cas}{0}
 \begin{cas}[$|h|\geq \max_{1\leq j\leq J}|y_j|$]
  In this case, the inequality is trivial since all terms of the left-hand side are bounded by $|h|^{\frac{2N}{N-2}}$ up to a constant.
 \end{cas}
 \begin{cas}[$|h|\leq \max_{1\leq j\leq J}|y_j|$]
  We assume without loss of generality $$|y_{1}|=\max_{1\leq j\leq J}|y_j|.$$ We use the inequality
  $$ \left|\frac{N-2}{2N}|1+s|^{\frac{2N}{N-2}}-\frac{N-2}{2N}-s\right|\lesssim s^2+|s|^{\frac{2N}{N-2}}$$
  with 
  $s=\frac{1}{y_{1}}\left( h+\sum_{j=2}^J y_j \right).$.
  Multiplying the resulting inequality by $|y_1|^{\frac{2N}{N-2}}$, we obtain
  \begin{multline}
  \label{L50}
   \left|\frac{N-2}{2N}\bigg|h+\sum_{j=1}^J y_j\bigg|^{\frac{2N}{N-2}}-\frac{N-2}{2N}|y_{1}|^{\frac{2N}{N-2}}-|y_1|^{\frac{4}{N-2}}y_1\left( h+\sum_{j=2}^J y_j \right)\right|\\
   \lesssim |y_1|^{\frac{4}{N-2}}\left( h^2+\sum_{j=2}^J y_j^2\right)+|h|^{\frac{2N}{N-2}}+\sum_{j=2}^J |y_j|^{\frac{2N}{N-2}}
  \end{multline}
  Since $|y_1|=\max_{1\leq j\leq J}|y_j|$ and $\frac{4}{N-2}<2$, the right-hand side of \eqref{L50} is clearly bounded by $$
 |h|^{\frac{2N}{N-2}}+\sum_{j=1}^J |y_j|^{\frac{4}{N-2}} h^2\\
 +\sum_{1\leq j<k\leq J}\min\left\{ |y_j|^{\frac{4}{N-2}} y_k^2,|y_k|^{\frac{4}{N-2}} y_j^2\right\}.$$
 It remains to bound the terms that appear in the left-hand side of \eqref{L40} but not on the left-hand side of \eqref{L50}. Using again that $|y_1|=\max_{1\leq j\leq J}|y_j|$, we have 
 \begin{align*}
\sum_{j=2}^J |y_j|^{\frac{2N}{N-2}}&\lesssim \sum_{j\neq k} \min\left\{ |y_j|^{\frac{4}{N-2}} y_k^2,|y_k|^{\frac{4}{N-2}} y_j^2\right\}\\
\bigg|\sum_{j=2}^{J} |y_j|^{\frac{4}{N-2}} y_jh\bigg|&\lesssim \sum_{j=2}^J |y_j|^{\frac{2N}{N-2}}+\sum_{j=2}^J |y_j|^{\frac{4}{N-2}}h^2\\&\lesssim \sum_{j\neq k} \min\left\{ |y_j|^{\frac{4}{N-2}} y_k^2,|y_k|^{\frac{4}{N-2}} y_j^2\right\}+\sum_{j=2}^J |y_j|^{\frac{4}{N-2}}h^2\\
\bigg|
\sum_{\substack{2\leq j\leq J\\ 1\leq k\leq J\\ j\neq k}}|y_j|^{\frac{4}{N-2}}y_jy_k\bigg|&\lesssim \sum_{2\leq j\leq J}|y_1||y_j|^{\frac{N+2}{N-2}}\lesssim \sum_{j\neq k}\min\left\{ |y_j|^{\frac{N+2}{N-2}}|y_k|,|y_k|^{\frac{N+2}{N-2}}|y_j|\right\},
 \end{align*}
 which concludes the proof.
\end{cas}
\end{proof}
Recall the notation $F(\sigma)=|\sigma|^{\frac{4}{N-2}}\sigma$. 
\begin{claim}
 \label{Cl:pointwise2}
 Assume $N\geq 7$. Let $J\geq 1$. Then for all $(y_1,\ldots,y_J,h)\in \RR^{J+1}$, if $N\geq 7$,
 \begin{multline}
\label{L10}
\bigg|F\bigg(h+\sum_{j=1}^Jy_j\bigg)-\sum_{j=1}^J F(y_j)-\frac{N+2}{N-2}\sum_{j=1}^J |y_j|^{\frac{4}{N-2}}h-F(h)\bigg|\\
\lesssim \sum_{1\leq j<k\leq J} \min\left\{ |y_j|^{\frac{4}{N-2}}|y_k|,|y_k|^{\frac{4}{N-2}}|y_j| \right\}+|h|^{\frac{N+1}{N-2}} \sum_{j=1}^{J} |y_j|^{\frac{1}{N-2}},
 \end{multline}
and if $N=5$,
 \begin{multline}
 \label{L10'}
 \bigg|
F\bigg(\sum_{j=1}^J y_j+h\bigg)-\sum_{j=1}^J F(y_j)-\frac 73\sum_{j=1}^J |y_j|^{\frac 43}h-\frac 73\sum_{\substack{ 1\leq j,k\leq J\\ j\neq k}}|y_j|^{\frac 43} y_k-F(h)\bigg|\\ \lesssim \sum_{j=1}^{J}|y_j|^{\frac 13}h^2+\sum_{1\leq j<k\leq J}\min\left( |y_j|^{\frac 43}|y_k|,|y_k|^{\frac 43}|y_j| \right)
 \end{multline}
 \end{claim}
 \begin{proof}
 \setcounter{cas}{0}
 \begin{cas}[$\max_j|y_j|\leq |h|$]
  We use:
  \begin{equation}
   \label{L11}
   \quad |s|\leq J\Longrightarrow
   \left|F(1+s)-1-\frac{N+2}{N-2}s\right|\lesssim s^2,
  \end{equation} 
  with $s=\frac{1}{h}\sum_{j=1}^J y_j$. Multiplying the resulting inequality by $|h|^{\frac{N+2}{N-2}}$, we obtain
  $$
  \bigg|F\bigg(h+\sum_{j=1}^Jy_j\bigg)-F(h)-\frac{N+2}{N-2} |h|^{\frac{4}{N-2}}\sum_{j=1}^J y_j\bigg|\lesssim |h|^{\frac{N+2}{N-2}-2}\sum_{j=1}^J y_j^2.
  $$
 Using that $\max_j |y_j|\leq |h|$ we deduce \eqref{L10} or \eqref{L10'}.
 \end{cas}
 \begin{cas}[$|h|\leq \max_j |y_j|$]
  We assume without loss of generality $|y_1|=\max_{1\leq j\leq J} |y_j|$. We use \eqref{L11} with $s=\frac{1}{y_1}\left( h+\sum_{j=2}^J y_j \right)$. Multiplying the resulting inequality by $|y_1|^{\frac{N+2}{N-2}}$, we obtain:
  $$\bigg| F\left( h+\sum_{1\leq j\leq J} y_j \right)-F(y_1)-\frac{N+2}{N-2}|y_1|^{\frac{4}{N-2}}\bigg(h+\sum_{j=2}^J y_j\bigg)\bigg|\lesssim |y_1|^{\frac{N+2}{N-2}-2}\bigg(h^2+\sum_{j=2}^J y_j^2\bigg).$$
 Let $j\geq 2$. Since $|y_j|\leq |y_1|$,
 $$F(y_j)\lesssim \min\left\{|y_j|^{\frac{4}{N-2}}|y_1|,|y_1|^{\frac{4}{N-2}}|y_j|\right\}.$$
 Furthermore 
 $$|y_j|^{\frac{4}{N-2}}|h|\lesssim \begin{cases} 
|y_j|^{\frac{1}{N-2}}|h|^{\frac{N+1}{N-2}}&\text{ if }|y_j|<|h|\\
\min\left\{|y_j|^{\frac{4}{N-2}}|y_1|,|y_1|^{\frac{4}{N-2}}|y_j|\right\}&\text{ if }|h|\leq |y_j|,
\end{cases}
 $$
 and also
 \begin{align*}
|y_{1}|^{\frac{N+2}{N-2}-2}h^2+F(h)&\lesssim |y_1|^{\frac{1}{N-2}}|h|^{\frac{N+1}{N-2}}\\
 |y_1|^{\frac{N+2}{N-2}-2}y_j^2&\leq \min\left\{|y_1|^{\frac{4}{N-2}}|y_j|,|y_j|^{\frac{4}{N-2}}|y_1|\right\}.
 \end{align*}
If $N\geq 7$, we have $\frac{4}{N-2}<1$ and thus
 $$|y_1|^{\frac{4}{N-2}}|y_j|=\min\left\{|y_1|^{\frac{4}{N-2}}|y_j|,|y_j|^{\frac{4}{N-2}}|y_1|\right\}.$$
 If $N=5$, we note that if $2\leq j,k$ with $j\neq k$
 $$ |y_j|^{\frac{4}{3}}|y_k|\lesssim \min\left\{|y_1|^{\frac{4}{3}}|y_j|,|y_j|^{\frac{4}{3}}|y_1|\right\}+\min\left\{|y_1|^{\frac{4}{3}}|y_k|,|y_k|^{\frac{4}{3}}|y_1|\right\}$$
 Combining the preceding inequalities, we obtain \eqref{L10} or \eqref{L10'}.
 \end{cas}
\end{proof}
\begin{claim}
\label{Cl:pointwise3}
Let $(a,b,c)\in \RR^3$ with $a\neq 0$.
 \begin{equation}
  \label{BT20}
  |F(a+b)-F(a)-F'(a)b|\lesssim \indic_{\{|b|\leq |a|\}}b^2a^{\frac{6-N}{N-2}}+\indic_{\{|b|\geq |a|\}} b^{\frac{N+2}{N-2}},
  \end{equation}
  and
  \begin{equation}
  \label{BT30}
\left|F(a+b+c)-F(a+b)-F(a+c)+F(a)\right|\lesssim 
\begin{cases}
|a|^{\frac{6-N}{2(N-2)}}|b|^{\frac{N+2}{2(N-2)}}|c|&\text{ if }N\geq 7\\
|b|\,|c|\,\big(|a|+|b|+|c|\big)^{\frac 13}&\text{ if }N=5.
 \end{cases}
 \end{equation}
\end{claim}
\begin{proof}
 \noindent\emph{Proof of \eqref{BT20}}. By scaling, we can assume $a=1$. We are thus reduced to prove:
 \begin{equation}
  \label{BT20'}
  |F(1+b)-F(1)-F'(1)b|\lesssim \indic_{|b|\leq 1}b^2+\indic_{|b|\geq 1} b^{\frac{N+2}{N-2}}, \quad b\in \RR,
 \end{equation}
which follows easily from the fact that $F(z)$ is $C^{2}$ outside $z=0$ and of order $|z|^\frac{N+2}{N-2}$ as $|z|\to\infty$.

\medskip

 \noindent\emph{Proof of \eqref{BT30} in the case $N\geq 7$}.
 Note that $1+\frac{N+2}{2(N-2)}+\frac{6-N}{2(N-2)}=\frac{N+2}{N-2}$. Thus both sides of \eqref{BT30} are homogeneous of degree $\frac{N+2}{N-2}$ and we can assume without loss of generality, $a=1$. We are thus reduced to prove (assuming $N\geq 7$):
 \begin{equation}
 \label{BT30'}
  |G(b,c)|\lesssim |b|^{\frac{N+2}{2(N-2)}}|c|,
  \end{equation}
  where
  \begin{equation*}
  \quad G(b,c):=|F(1+b+c)-F(1+b)-F(1+c)+F(1)|.
 \end{equation*}
We distinguish between two cases.
\setcounter{cas}{0}
\begin{cas}[$|c|\leq |b|$]
 There exists $b_1$ and $d_1$ such that
\begin{align}
\label{exist_b1}
 F(1+b+c)-F(1+b)&=F'(b_1)c,&\quad |b+1-b_1|&\leq |c|\\
 \label{exist_d1}
 F(1+c)-F(1)&= F'(d_1)c,&\quad |d_1-1|&\leq |c|.
\end{align}
In particular,
$$ |G(b,c)|\lesssim (1+|b|+|c|)^{\frac{4}{N-2}}|c|.$$
If $|b|\geq \frac{1}{10}$, this implies \eqref{BT30'}, since $\frac{4}{N-2}\leq \frac{N+2}{2(N-2)}$.

If $|b|\leq \frac{1}{10}$, we use that $F$ is $C^2$ outside the origin. Thus there exists $d_2 \in [b_1,d_1]$ (or $[d_1,b_1]$) such that 
$$ F'(b_1)-F'(d_1)=F''(d_2)(d_1-b_1).$$
Since $\frac{1}{2}\leq d_2\leq 2$, we have $|F''(d_2)|\lesssim 1$ and we obtain by the triangle inequality
$$ |b_1-d_1|\leq |b_1-(1+b)|+|d_1-1|+|b|\leq |c|+|b|,$$
which yields
$$ |G(b,c)|=|F''(d_2)(b_1-d_1)c|\leq |c|(|c|+|b|),$$
yielding \eqref{BT30'} since $|c|\leq |b|\leq \frac{1}{10}$ and $\frac{N+2}{2(N-2)}< 1$.
\end{cas}
\begin{cas}[$|c|\geq |b|$]
The same proof as in the preceding case, inverting  $b$ and $c$, yields
$$ |G(b,c)|\lesssim (1+|b|+|c|)^{\frac{4}{N-2}}|b|,$$
and, if $|c|\leq \frac{1}{10}$,
$$ |G(b,c)|\leq |b|(|c|+|b|).$$
Using that $\frac{4}{N-2}<\frac{N+2}{2(N-2)}<1$, 
we obtain \eqref{BT30'}.
\end{cas}

\noindent\emph{Proof of \eqref{BT30} in the case $N=5$}.
By homogeneity, we can assume $a=1$, and we are thus reduce to prove, with the same notation  $G(b,c)$ as before,
$$|G(b,c)|\lesssim |b|\,|c|\,(1+|b|+|c|)^{\frac 13}.$$
The inequality is symmetric in $(b,c)$ and we can assume $|b|\geq |c|$. Again, we use that there exist $b_1$ and $d_1$ such that \eqref{exist_b1} and \eqref{exist_d1} hold. Since $F$ is of class $C^2$, we also know that there exists $d_2 \in [b_1,d_1]$ (or $[d_1,b_1]$) such that 
$$ F'(b_1)-F'(d_1)=F''(d_2)(d_1-b_1).$$
We have $|d_1-b_1|\lesssim |b|+|c|\lesssim |b|$, and $|F''(d_2)|\lesssim (1+|b|+|c|)^{\frac 13}$, and thus
$$|G(b,c)|=\left|F''(d_2)(d_1-b_1)c\right|\lesssim |c|\,|b|\,(1+|b|+|c|)^{\frac 13}.$$
\end{proof}

\section{Choice of the scaling parameters}
\begin{lemma}
 \label{L:ortho_scaling}
 Let $J\geq 1$. There exists a small constant $\eps_J>0$ and a large constant $C_J>0$, with the following property. For all $\eps\in(0,\eps_J)$, for all $\mubf=(\mu_j)_j\in (0,\infty)^J$ with $\mu_J<\mu_{J-1}<\ldots<\mu_1$ such that $\gamma(\mubf)<\eps$, for all $(\iota_j)_j\in \{\pm 1\}^J$, for all $f\in \hdot$ such that 
 $$\left\| f-\sum_{j=1}^J \iota_jW_{(\mu_j)}\right\|_{\hdot}\leq \eps,$$
 there exists a unique $\lambdabf\in (0,\infty)^{J}$ such that 
 $$\max_{1\leq j\leq J}\left|\frac{\lambda_j}{\mu_j}-1\right|\leq C_J\eps$$
 and
 $$\forall j\in \llbracket 1,J\rrbracket \int \nabla\Big( f-\sum_{j=1}^J \iota_j W_{(\lambda_j)}\Big)\cdot\nabla \left( \Lambda W\right)_{(\lambda_j)}=0.$$
 Furthermore, the map $f\mapsto \lambdabf$ is of class $C^1$.
\end{lemma}
\begin{remark}
 \label{R:ortho_scaling}
 Let us mention that
 $$\left\|\sum_{j=1}^J \iota_j W_{(\lambda_j)}-\sum_{j=1}^J \iota_j W_{(\mu_j)}\right\|_{\hdot}\leq C_J\eps$$
 (see the computation in the proof below) and that $\gamma(\lambdabf)\approx \gamma(\mubf)$.
\end{remark}
 
\begin{proof}[Sketch of proof]
 This is standard and follows from the implicit function theorem. However we have to check that the uniformity of the constant with respect to $\mubf$ stated in the lemma follows from the proof. 
 
We fix $\mubf$ and $f$ such that 
$$\gamma(\mubf)<\eps,\quad \left\|f-\sum_{j=1}^J \iota_j W_{(\mu_j)}\right\|_{\hdot}<\eps.$$ 
We consider 
$$\Phi:(0,\infty)^J \to \RR^J,\quad \Phi=(\phi_{\ell})_{1\leq \ell\leq J},$$
defined by
$$\Phi_{\ell}(\lambdabf)=\lambda_{\ell}-\frac{1}{\int |\nabla \Lambda W|^2}\mu_{\ell}\iota_{\ell} 
\int \nabla\Big( f-\sum_{j=1}^J \iota_j W_{(\lambda_j)}\Big)\cdot\nabla \left(\Lambda W_{(\lambda_\ell)}\right).
$$
We will prove that $\Phi$ is a contraction of the compact set:
$$B_{\eta}=\left\{(\lambda_j)_{1\leq j\leq J}\;:\; \max_{\ell}\left|1-\frac{\lambda_{\ell}}{\mu_{\ell}}\right|\leq \eta\right\},$$
where
$$\eta=M_J\eps$$
for a large positive $M_J$ to be specified.
Choosing $\eps$ small enough, we have that $\lambdabf\in B(\eta)$ implies
$$ \frac{1}{2}\leq \frac{\lambda_j}{\mu_j}\leq \frac 32$$ for all $j$.

If $j\neq \ell$, we have
$$ \frac{\partial \Phi_{\ell}}{\partial \lambda_{j}}(\lambdabf)=-\iota_j\iota_{\ell}\frac{\mu_{\ell}}{\lambda_{j}} \frac{1}{\int |\nabla \Lambda W|^2} \int \nabla (\Lambda W_{(\lambda_j)})\cdot \nabla (\Lambda W_{(\lambda_{\ell})}).$$
Since by Claim \ref{Cl:estimates1}, we have
$$ \left|\int \nabla (\Lambda W_{(\lambda_j)})\cdot \nabla (\Lambda W_{(\lambda_{\ell})})\right|\lesssim \max\left\{\left(\frac{\lambda_{\ell}}{\lambda_j}\right)^{\frac 32},\left(\frac{\lambda_{j}}{\lambda_{\ell}}\right)^{\frac 32}\right\},$$
we deduce
\begin{equation}
 \label{bnd_dphil_j} 
 \left|\frac{\partial \Phi_{\ell}}{\partial \lambda_j}(\lambdabf)\right|\lesssim \frac{\mu_{\ell}}{\mu_j} \eps^{\frac{3}{2}}.
\end{equation} 
Moreover
$$\frac{\partial \Phi_{\ell}}{\partial \lambda_{\ell}}(\lambdabf)=1-\frac{\mu_{\ell}}{\lambda_{\ell}}+\frac{1}{\lambda_{\ell}\int |\nabla \Lambda W|^2} \iota_{\ell}\int \nabla\Big( f-\sum_{j=1}^J \iota_j W_{(\lambda_j)}\Big)\cdot\nabla \left(\Lambda \Lambda W_{(\lambda_\ell)}\right),$$
and thus
\begin{multline}
\label{bnd_dphil_l}
\left|\frac{\partial \Phi_{\ell}}{\partial \lambda_{\ell}}(\lambdabf)\right|\lesssim \left|1-\frac{\mu_{\ell}}{\lambda_{\ell}}\right|+\left\|f-\sum_{j=1}^J\iota_jW_{(\mu_j)}\right\|_{\hdot}+\left\|\sum_{j=1}^J\iota_jW_{(\mu_j)}-\iota_jW_{(\lambda_j)}\right\|_{\hdot}\\
\lesssim \eta+\eps,
\end{multline}
where we have used
\begin{multline*}
 \int |\nabla(W_{(\lambda_j)}-W_{(\mu_j)})|^2=\int \left|\left( \frac{\mu_j}{\lambda_j} \right)^{\frac N2}\nabla W\left( \frac{\mu_j}{\lambda_j} x\right)-\nabla W(x)\right|^2dx\\
 \lesssim \eta^2+ \int \left|\nabla W\left( \frac{\mu_j}{\lambda_j}x \right)-\nabla W(x)\right|^2dx\lesssim \eta^2+\int_0^{+\infty} \left( r-\frac{\mu_j}{\lambda_j}r \right)^2\frac{r^{N-1}}{(1+r^N)^2}dr\lesssim \eta^2,
\end{multline*}
since $|\nabla W(r)-\nabla W(\rho)|\lesssim \frac{|r-\rho|}{1+r^N}$, $r\approx \rho$.

Furthermore,
\begin{equation}
 \label{bnd_Phi_mu}
 \left|\frac{1}{\mu_{\ell}}\Phi_{\ell}(\mubf)-1\right|\lesssim \eps.
\end{equation} 
Combining \eqref{bnd_dphil_j}, \eqref{bnd_dphil_l} and \eqref{bnd_Phi_mu}, we see that if $\lambdabf\in B_{\eta}$, 
$$\left|\frac{1}{\mu_{\ell}}\Phi_{\ell}(\lambdabf)-1\right|\lesssim \eps+(\eta+\eps)\eta.$$
This proves that if $\eta=M_J\eps$ for some large constant $M_J$, and $\eps\leq \eps_J\ll M_J^{-1}$, $\Phi$ maps $B_{\eta}$ into $B_{\eta}$. By \eqref{bnd_dphil_j} and \eqref{bnd_dphil_l}, $\Phi$ is a contraction of $B_{\eta}$. By the Banach fixed point theorem, there exists a unique $\lambdabf\in B_{\eta}$ such that $\Phi(\lambdabf)=\lambdabf$, which exactly means that it satisfies the desired orthogonality conditions. The fact that $f\mapsto \lambdabf$ is $C^1$ is classical and we omit the proof.
\end{proof}

\bibliographystyle{acm}
\bibliography{/home/duyckaerts/ownCloud/Recherche/toto} 

\begin{thebibliography}{10}

\bibitem{BaGe99}
{\sc Bahouri, H., and G{\'e}rard, P.}
\newblock High frequency approximation of solutions to critical nonlinear wave
  equations.
\newblock {\em Amer. J. Math. 121}, 1 (1999), 131--175.

\bibitem{BoJeMcL18}
{\sc Borghese, M., Jenkins, R., and McLaughlin, K. D.-R.}
\newblock Long time asymptotic behavior of the focusing nonlinear
  {S}chr\"odinger equation.
\newblock {\em Ann. Inst. H. Poincar\'e Anal. Non Lin\'eaire 35}, 4 (2018),
  887--920.

\bibitem{Bulut10}
{\sc Bulut, A.}
\newblock Maximizers for the {S}trichartz inequalities for the wave equation.
\newblock {\em Differential Integral Equations 23}, 11/12 (2010), 1035--1072.

\bibitem{BuCzLiPaZh13}
{\sc Bulut, A., Czubak, M., Li, D., Pavlovi{\'c}, N., and Zhang, X.}
\newblock Stability and unconditional uniqueness of solutions for energy
  critical wave equations in high dimensions.
\newblock {\em Comm. Partial Differential Equations 38}, 4 (2013), 575--607.

\bibitem{ChTZ93}
{\sc Christodoulou, D., and Tahvildar-Zadeh, A.~S.}
\newblock On the asymptotic behavior of spherically symmetric wave maps.
\newblock {\em Duke Math. J. 71}, 1 (1993), 31--69.

\bibitem{Cote15}
{\sc C\^{o}te, R.}
\newblock On the soliton resolution for equivariant wave maps to the sphere.
\newblock {\em Comm. Pure Appl. Math. 68}, 11 (2015), 1946--2004.

\bibitem{CoKeLaSc15a}
{\sc C{\^o}te, R., Kenig, C.~E., Lawrie, A., and Schlag, W.}
\newblock Characterization of large energy solutions of the equivariant wave
  map problem: {I}.
\newblock {\em Amer. J. Math. 137}, 1 (2015), 139--207.

\bibitem{CoKeLaSc15b}
{\sc C{\^o}te, R., Kenig, C.~E., Lawrie, A., and Schlag, W.}
\newblock Characterization of large energy solutions of the equivariant wave
  map problem: {II}.
\newblock {\em Amer. J. Math. 137}, 1 (2015), 209--250.

\bibitem{CoKeLaSc18}
{\sc C{\^o}te, R., Kenig, C.~E., Lawrie, A., and Schlag, W.}
\newblock Profiles for the radial focusing 4d energy-critical wave equation.
\newblock {\em Communications in Mathematical Physics 357}, 3 (2018),
  943--1008.

\bibitem{CoKeSc14}
{\sc C{\^o}te, R., Kenig, C.~E., and Schlag, W.}
\newblock Energy partition for the linear radial wave equation.
\newblock {\em Math. Ann. 358}, 3-4 (2014), 573--607.

\bibitem{dPMPP11}
{\sc del Pino, M., Musso, M., Pacard, F., and Pistoia, A.}
\newblock Large energy entire solutions for the {Y}amabe equation.
\newblock {\em J. Differential Equations 251}, 9 (2011), 2568--2597.

\bibitem{dPMPP13}
{\sc del Pino, M., Musso, M., Pacard, F., and Pistoia, A.}
\newblock Torus action on {$S^n$} and sign-changing solutions for conformally
  invariant equations.
\newblock {\em Ann. Sc. Norm. Super. Pisa Cl. Sci. (5) 12}, 1 (2013), 209--237.

\bibitem{Ding86}
{\sc Ding, W.~Y.}
\newblock On a conformally invariant elliptic equation on {${\bf R}^n$}.
\newblock {\em Comm. Math. Phys. 107}, 2 (1986), 331--335.

\bibitem{Dodson15}
{\sc Dodson, B.}
\newblock Global well-posedness and scattering for the mass critical nonlinear
  {S}chr\"odinger equation with mass below the mass of the ground state.
\newblock {\em Advances in Mathematics 285\/} (2015), 1589--1618.

\bibitem{Donninger17}
{\sc Donninger, R.}
\newblock Strichartz estimates in similarity coordinates and stable blowup for
  the critical wave equation.
\newblock {\em Duke Math. J. 166}, 9 (2017), 1627--1683.

\bibitem{DoKr13}
{\sc Donninger, R., and Krieger, J.}
\newblock Nonscattering solutions and blowup at infinity for the critical wave
  equation.
\newblock {\em Math. Ann. 357}, 1 (2013), 89--163.

\bibitem{DuHoRo08}
{\sc Duyckaerts, T., Holmer, J., and Roudenko, S.}
\newblock Scattering for the non-radial 3{D} cubic nonlinear {S}chr\"odinger
  equation.
\newblock {\em Math. Res. Lett. 15}, 6 (2008), 1233--1250.

\bibitem{DuJiKeMe17}
{\sc Duyckaerts, T., Jia, H., Kenig, C., and Merle, F.}
\newblock Soliton resolution along a sequence of times for the focusing energy
  critical wave equation.
\newblock {\em Geometric and Functional Analysis 27}, 4 (2017), 798--862.

\bibitem{DuJiKeMe17b}
{\sc Duyckaerts, T., Jia, H., Kenig, C., and Merle, F.}
\newblock Universality of blow up profile for small blow up solutions to the
  energy critical wave map equation.
\newblock {\em Int. Math. Res. Not. IMRN 2018}, 22 (05 2017), 6961--7025.

\bibitem{DuKeMe11a}
{\sc Duyckaerts, T., Kenig, C., and Merle, F.}
\newblock Universality of blow-up profile for small radial type {II} blow-up
  solutions of the energy-critical wave equation.
\newblock {\em J. Eur. Math. Soc. (JEMS) 13}, 3 (2011), 533--599.

\bibitem{DuKeMe12b}
{\sc Duyckaerts, T., Kenig, C., and Merle, F.}
\newblock Profiles of bounded radial solutions of the focusing, energy-critical
  wave equation.
\newblock {\em Geom. Funct. Anal. 22}, 3 (2012), 639--698.

\bibitem{DuKeMe12}
{\sc Duyckaerts, T., Kenig, C., and Merle, F.}
\newblock Universality of the blow-up profile for small type {II} blow-up
  solutions of the energy-critical wave equation: the nonradial case.
\newblock {\em J. Eur. Math. Soc. (JEMS) 14}, 5 (2012), 1389--1454.

\bibitem{DuKeMe13}
{\sc Duyckaerts, T., Kenig, C., and Merle, F.}
\newblock Classification of radial solutions of the focusing, energy-critical
  wave equation.
\newblock {\em Cambridge Journal of Mathematics 1}, 1 (2013), 75--144.

\bibitem{DuKeMe19Pc}
{\sc Duyckaerts, T., Kenig, C., and Merle, F.}
\newblock Decay estimates for nonradiative solutions of the energy-critical
  focusing wave equation.
\newblock Preprint, 2019.

\bibitem{DuKeMe19Pa}
{\sc Duyckaerts, T., Kenig, C., and Merle, F.}
\newblock Exterior energy bounds for the critical wave equation close to the
  ground state.
\newblock Preprint, 2019.

\bibitem{DuKeMe19}
{\sc Duyckaerts, T., Kenig, C., and Merle, F.}
\newblock Scattering profile for global solutions of the energy-critical wave
  equation.
\newblock {\em J. Eur. Math. Soc. (JEMS) 21}, 7 (2019), 2117--2162.

\bibitem{Eckhaus86}
{\sc Eckhaus, W.}
\newblock The long-time behaviour for perturbed wave-equations and related
  problems.
\newblock In {\em Trends in applications of pure mathematics to mechanics
  ({B}ad {H}onnef, 1985)}, vol.~249 of {\em Lecture Notes in Phys.} Springer,
  Berlin, 1986, pp.~168--194.

\bibitem{EcSc83}
{\sc Eckhaus, W., and Schuur, P.}
\newblock The emergence of solitons of the {K}orteweg-de {V}ries equation from
  arbitrary initial conditions.
\newblock {\em Math. Methods Appl. Sci. 5}, 1 (1983), 97--116.

\bibitem{FaXiCa11}
{\sc Fang, D., Xie, J., and Cazenave, T.}
\newblock Scattering for the focusing energy-subcritical nonlinear
  {S}chr\"odinger equation.
\newblock {\em Sci. China Math. 54}, 10 (2011), 2037--2062.

\bibitem{FermiPastaUlam55}
{\sc Fermi, E., Pasta, J., and Ulam, S.}
\newblock {L}os {A}lamos {R}eport {LA}-1940, 1955.

\bibitem{GiNiNi81}
{\sc Gidas, B., Ni, W.~M., and Nirenberg, L.}
\newblock Symmetry of positive solutions of nonlinear elliptic equations in
  {${\bf R}^{n}$}.
\newblock In {\em Mathematical analysis and applications, {P}art {A}}, vol.~7
  of {\em Adv. in Math. Suppl. Stud.} Academic Press, New York, 1981,
  pp.~369--402.

\bibitem{GiSoVe92}
{\sc Ginibre, J., Soffer, A., and Velo, G.}
\newblock The global {C}auchy problem for the critical nonlinear wave equation.
\newblock {\em J. Funct. Anal. 110}, 1 (1992), 96--130.

\bibitem{GiVe95}
{\sc Ginibre, J., and Velo, G.}
\newblock Generalized {S}trichartz inequalities for the wave equation.
\newblock {\em J. Funct. Anal. 133}, 1 (1995), 50--68.

\bibitem{Grinis17}
{\sc Grinis, R.}
\newblock Quantization of time-like energy for wave maps into spheres.
\newblock {\em Comm. Math. Phys. 352}, 2 (2017), 641--702.

\bibitem{HiRa12}
{\sc Hillairet, M., and Rapha{\"e}l, P.}
\newblock Smooth type {II} blow-up solutions to the four-dimensional
  energy-critical wave equation.
\newblock {\em Anal. PDE 5}, 4 (2012), 777--829.

\bibitem{IbMaNa11}
{\sc Ibrahim, S., Masmoudi, N., and Nakanishi, K.}
\newblock Scattering threshold for the focusing nonlinear {K}lein-{G}ordon
  equation.
\newblock {\em Anal. PDE 4}, 3 (2011), 405--460.

\bibitem{IvancevicIvancevic10BO}
{\sc Ivancevic, V.~G., and Ivancevic, T.~T.}
\newblock {\em Quantum neural computation}.
\newblock Intelligent Systems, Control and Automation: Science and Engineering,
  40. Springer, New York, 2010.

\bibitem{Jendrej17}
{\sc Jendrej, J.}
\newblock Construction of type {II} blow-up solutions for the energy-critical
  wave equation in dimension 5.
\newblock {\em J. Funct. Anal. 272}, 3 (2017), 866--917.

\bibitem{Jendrej19}
{\sc Jendrej, J.}
\newblock Construction of two-bubble solutions for energy-critical wave
  equations.
\newblock {\em American Journal of Mathematics 141}, 1 (2019), 55--118.

\bibitem{JendrejLawrie19}
{\sc Jendrej, J., and Lawrie, A.}
\newblock Two-bubble dynamics for threshold solutions to the wave maps
  equation.
\newblock {\em Inventiones mathematicae 213}, 3 (2018), 1249--1325.

\bibitem{JiaKenig17}
{\sc Jia, H., and Kenig, C.}
\newblock Asymptotic decomposition for semilinear wave and equivariant wave map
  equations.
\newblock {\em American Journal of Mathematics 139}, 6 (2017), 1521--1603.

\bibitem{KeLaLiSc15}
{\sc Kenig, C., Lawrie, A., Liu, B., and Schlag, W.}
\newblock Channels of energy for the linear radial wave equation.
\newblock {\em Adv. Math. 285\/} (2015), 877--936.

\bibitem{KeLaSc14}
{\sc Kenig, C.~E., Lawrie, A., and Schlag, W.}
\newblock Relaxation of wave maps exterior to a ball to harmonic maps for all
  data.
\newblock {\em Geometric and Functional Analysis 24}, 2 (2014), 610--647.

\bibitem{KeMe06}
{\sc Kenig, C.~E., and Merle, F.}
\newblock Global well-posedness, scattering and blow-up for the
  energy-critical, focusing, non-linear {S}chr\"odinger equation in the radial
  case.
\newblock {\em Invent. Math. 166}, 3 (2006), 645--675.

\bibitem{KeMe08}
{\sc Kenig, C.~E., and Merle, F.}
\newblock Global well-posedness, scattering and blow-up for the energy-critical
  focusing non-linear wave equation.
\newblock {\em Acta Math. 201}, 2 (2008), 147--212.

\bibitem{KrSc07}
{\sc Krieger, J., and Schlag, W.}
\newblock On the focusing critical semi-linear wave equation.
\newblock {\em Amer. J. Math. 129}, 3 (2007), 843--913.

\bibitem{KrSc14}
{\sc Krieger, J., and Schlag, W.}
\newblock Full range of blow up exponents for the quintic wave equation in
  three dimensions.
\newblock {\em J. Math. Pures Appl. (9) 101}, 6 (2014), 873--900.

\bibitem{KrScTa09}
{\sc Krieger, J., Schlag, W., and Tataru, D.}
\newblock Slow blow-up solutions for the {$H\sp 1(\mathbb{R}\sp 3)$} critical
  focusing semilinear wave equation.
\newblock {\em Duke Math. J. 147}, 1 (2009), 1--53.

\bibitem{Kruskal78}
{\sc Kruskal, M.~D.}
\newblock The birth of the soliton.
\newblock In {\em Nonlinear evolution equations solvable by the spectral
  transform}, vol.~26 of {\em Res. Notes in Math.} Pitman, Boston, Mass., 1978,
  pp.~1--8.

\bibitem{LiSo95}
{\sc Lindblad, H., and Sogge, C.~D.}
\newblock On existence and scattering with minimal regularity for semilinear
  wave equations.
\newblock {\em J. Funct. Anal. 130}, 2 (1995), 357--426.

\bibitem{MartelMerle18}
{\sc Martel, Y., and Merle, F.}
\newblock Inelasticity of soliton collisions for the 5{D} energy critical wave
  equation.
\newblock {\em Invent. Math. 214}, 3 (2018), 1267--1363.

\bibitem{MerleZaag12}
{\sc Merle, F., and Zaag, H.}
\newblock Existence and classification of characteristic points at blow-up for
  a semilinear wave equation in one space dimension.
\newblock {\em Amer. J. Math. 134}, 3 (2012), 581--648.

\bibitem{NaSc11Bo}
{\sc Nakanishi, K., and Schlag, W.}
\newblock {\em Invariant manifolds and dispersive Hamiltonian evolution
  equations}.
\newblock European Mathematical Society, 2011.

\bibitem{Novoksenov80}
{\sc Novok\v{s}enov, V.~J.}
\newblock Asymptotic behavior as {$t\rightarrow \infty $} of the solution of
  the {C}auchy problem for a nonlinear {S}chr\"{o}dinger equation.
\newblock {\em Dokl. Akad. Nauk SSSR 251}, 4 (1980), 799--802.

\bibitem{Pohozaev65}
{\sc Pohozaev, S.}
\newblock On the eigenfunctions of the equation $\delta u+ \lambda f(u)=0$.
\newblock In {\em Sov. Math. Doklady\/} (1965), vol.~6, pp.~1408--1411.

\bibitem{Rey90}
{\sc Rey, O.}
\newblock The role of the {G}reen's function in a nonlinear elliptic equation
  involving the critical {S}obolev exponent.
\newblock {\em J. Funct. Anal. 89}, 1 (1990), 1--52.

\bibitem{Rodriguez16}
{\sc Rodriguez, C.}
\newblock Profiles for the radial focusing energy-critical wave equation in odd
  dimensions.
\newblock {\em Adv. Differential Equations 21}, 5/6 (05 2016), 505--570.

\bibitem{Schuur86BO}
{\sc Schuur, P.~C.}
\newblock {\em Asymptotic analysis of soliton problems}, vol.~1232 of {\em
  Lecture Notes in Mathematics}.
\newblock Springer-Verlag, Berlin, 1986.
\newblock An inverse scattering approach.

\bibitem{Segur76}
{\sc Segur, H.}
\newblock Asymptotic solutions and conservation laws for the nonlinear
  {S}chr\"{o}dinger equation. {II}.
\newblock {\em J. Mathematical Phys. 17}, 5 (1976), 714--716.

\bibitem{SegurAblowitz76}
{\sc Segur, H., and Ablowitz, M.~J.}
\newblock Asymptotic solutions and conservation laws for the nonlinear
  {S}chr\"{o}dinger equation. {I}.
\newblock {\em J. Mathematical Phys. 17}, 5 (1976), 710--713.

\bibitem{Shen14}
{\sc Shen, R.}
\newblock On the energy subcritical, nonlinear wave equation in $\mathbb{R}^3$
  with radial data.
\newblock {\em Anal. PDE 6}, 8 (2014), 1929--1987.

\bibitem{StTa10b}
{\sc Sterbenz, J., and Tataru, D.}
\newblock Regularity of wave-maps in dimension {$2+1$}.
\newblock {\em Comm. Math. Phys. 298}, 1 (2010), 231--264.

\bibitem{St77a}
{\sc Strichartz, R.~S.}
\newblock Restrictions of {F}ourier transforms to quadratic surfaces and decay
  of solutions of wave equations.
\newblock {\em Duke Math. J. 44}, 3 (1977), 705--714.

\bibitem{Struwe03b}
{\sc Struwe, M.}
\newblock Equivariant wave maps in two space dimensions.
\newblock {\em Comm. Pure Appl. Math. 56}, 7 (2003), 815--823.
\newblock Dedicated to the memory of J{\"u}rgen K. Moser.

\bibitem{Talenti76}
{\sc Talenti, G.}
\newblock Best constant in {S}obolev inequality.
\newblock {\em Ann. Mat. Pura Appl. (4) 110\/} (1976), 353--372.

\bibitem{Tao07DPDE}
{\sc Tao, T.}
\newblock A (concentration-) compact attractor for high-dimensional non-linear
  schr{\"o}dinger equations.
\newblock {\em Dyn. Partial Differ. Equ. 4}, 1 (2007), 1--53.

\bibitem{Topping97}
{\sc Topping, P.~M., et~al.}
\newblock Rigidity in the harmonic map heat flow.
\newblock {\em J. Differential Geom 45}, 3 (1997), 593--610.

\bibitem{ZabuskyKruskal65}
{\sc Zabusky, N.~J., and Kruskal, M.~D.}
\newblock Interaction of "solitons" in a collisionless plasma and the
  recurrence of initial states.
\newblock {\em Physical Review Letters 15}, 6 (1965), 240.

\bibitem{ZakharovShabat71}
{\sc Zakharov, V.~E., and Shabat, A.~B.}
\newblock Exact theory of two-dimensional self-focusing and one-dimensional
  self-modulation of waves in nonlinear media.
\newblock {\em \v{Z}. \`Eksper. Teoret. Fiz. 61}, 1 (1971), 118--134.

\end{thebibliography}
\end{document}